\documentclass[10pt]{amsart}
\usepackage[margin=1in]{geometry}
\usepackage{amsthm,amsmath,amsfonts,amssymb,euscript,hyperref,graphics,color,slashed,mathrsfs,graphicx,mathtools,enumerate,cite}
\hypersetup{colorlinks, citecolor=blue, urlcolor=blue, linkcolor=blue}
\usepackage[english]{babel}
\usepackage{bm}

\allowdisplaybreaks

\numberwithin{equation}{section}

\newtheorem{Theorem}{Theorem}
\newtheorem{Lemma}[Theorem]{Lemma}
\newtheorem{Proposition}[Theorem]{Proposition}
\newtheorem{Corollary}[Theorem]{Corollary}

\newtheorem*{MainEstimates}{Main Energy Estimates}
\newtheorem*{RMainEstimates}{Refined Energy Estimates}
\newtheorem*{ExistenceTheorem}{Existence of Future Scattering Fields}
\newtheorem*{ExistenceTheoremP}{Existence of Past Scattering Fields}
\newtheorem*{RigidityTheorem1}{The Rigidity Theorem 1}
\newtheorem*{RigidityTheorem2}{The Rigidity Theorem 2}
\newtheorem*{Remark}{Remark}

\newcommand\Emph{\textbf}
\newcommand{\D}{\partial}

\title[Rigidity for Alfv\'en Waves]{On the rigidity from infinity for  nonlinear Alfv\'en waves 
	}

\author[Mengni Li]{Mengni Li}
\address{Department of Mathematics and Yau Mathematical Sciences Center, Tsinghua University\\ Beijing, China}
\email{lmn17@mails.tsinghua.edu.cn}

\author[Pin Yu]{Pin Yu}
\address{Department of Mathematics and Yau Mathematical Sciences Center, Tsinghua University\\ Beijing, China}
\email{yupin@mail.tsinghua.edu.cn}



\begin{document}
\begin{abstract}
The Alfv\'en waves are fundamental wave phenomena in magnetized plasmas and the dynamics of Alfv\'en waves are governed by a system of nonlinear partial differential equations called the MHD system. In this paper, we study the rigidity aspect of the scattering problem for the MHD equations: We prove that the Alfv\'en waves must vanish if their scattering fields vanish at infinities. The proof is based on a careful study of the null structure and a family of weighted energy estimates.
\end{abstract}

\maketitle
\tableofcontents
\section{Introduction}
 Magnetohydrodynamics (MHD) studies mutual interactions between fluids and electromagnetic fields in electrically conducting fluids. Mathematically, we can use a system of nonlinear partial differential equations to study the theory, which is often called the MHD system. It combines the Euler equations and the Maxwell equations to describe the conservation laws of mass, momentum and energy and the laws of electromagnetics. More than just an important branch of fluid mechanics, MHD is much more diverse and complicated than the classical hydrodynamics because it mixes many different wave phenomena. There are two kinds of restoring forces so that the MHD system is perturbed from its stationary states and these forces give rise to different waves. Just as the fluid pressure generates sound waves in fluid mechanics, the magnetic pressure produces magnetoacoustic waves. In addition, the magnetic tension leads to the so called Alfv\'en waves. The  phenomena for the Alfv\'en waves  have fruitful applications in plasma physics, astronomy, industry, etc. However,  the Alfv\'en waves have no counterpart in the ordinary fluid theory since they even exist in the incompressible fluids. Historically, the Alfv\'en waves were studied  in 1942 in \cite{Alfven} by the Swedish plasma physicist Hannes Alfv\'en, who was particularly awarded the Nobel prize in 1970 for this celebrated discovery. 

\smallskip

From a physical point of view \cite{Davidson}, the most interesting plasma physics lives in the regime where a strong magnetic field presents. If the conductivity of the electrically conducting fluid is sufficiently large, we observe that the fluid particles tend to move along the magnetic field lines. Therefore, we assume that the fluid flows along a strong constant background magnetic field $B_0$ and both the fluid viscosity and Ohmic viscosity are zero (i.e., the MHD system is ideal). Without loss of generality, we set both the fluid density and the permeability to be $1$. In such a situation, the ideal incompressible MHD system can be phrased in the following differential relations:
\begin{equation}\label{MHD general}
\begin{cases}
	&\partial_t  v+ v\cdot \nabla v = -\nabla p + (\nabla\times b)\times b, \\
	&\partial_t b + v\cdot \nabla b =  b \cdot \nabla v,\\
	&\operatorname{div} v =0,\\
	&\operatorname{div} b =0,
\end{cases}\end{equation}
where $b(x,t):\mathbb{R}^3\times \mathbb{R}^+\to\mathbb{R}^3$ is the magnetic field, $v(x,t):\mathbb{R}^3\times \mathbb{R}^+\to\mathbb{R}^3$ is the fluid velocity, $p(x,t):\mathbb{R}^3\times \mathbb{R}^+\to\mathbb{R}$ is the fluid pressure. Once we rewrite the Lorentz force term $(\nabla\times b)\times b$ in the following form:
\[(\nabla\times b)\times b=-\nabla(\frac{1}{2}| b|^2)+b \cdot \nabla  b,\]
we can use $p'=p+\frac{1}{2}| b|^2$ to replace $p$ in the \eqref{MHD general}. This yields the following dynamical equations:
\begin{equation}\label{MHD original}
\begin{cases}
	&\partial_t  v+ v\cdot \nabla v = -\nabla p' + b \cdot \nabla b, \\
	&\partial_t b + v\cdot \nabla b =  b \cdot \nabla v.
\end{cases}
\end{equation}
For the sake of simplicity, we will still use $p$ to denote $p'$. We notice that $(v,b)\equiv (0,B_0)$ is a stationary solution (time independent) of the system, $B_0=(0,0,1)$ is a constant vector field. In plasma physics, this is often referred to as a strong magnetic background of the system.

We can employ the so called  Els\"{a}sser variables to diagonalize the system, where the new variables are defined as
\begin{equation*}
\begin{cases}
	&Z_+ = v +b, \\ 
	&Z_- = v-b.
	\end{cases}
\end{equation*}
The three-dimensional ideal incompressible MHD equations now read as
\begin{equation}\label{MHD in Elsasser}\begin{cases}
	&\partial_t  Z_+ +Z_- \cdot \nabla Z_+ = -\nabla p, \\
	&\partial_t  Z_- +Z_+ \cdot \nabla Z_- = -\nabla p,\\
	&\operatorname{div} Z_+ =0,\\
	&\operatorname{div} Z_- =0.
\end{cases}\end{equation}
Since we will study the perturbation of a strong magnetic background, we define
\begin{equation*}
\begin{cases}
	&Z_{+}=z_{+}+B_0,\\ 
	&Z_{-}=z_{-}-B_0,
	\end{cases}
\end{equation*}
where $B_0=(0,0,1)$. Therefore, the MHD equations \eqref{MHD in Elsasser} lead to
\begin{equation}\label{eq:MHD}\begin{cases}
	&\D_{t}z_{+}-B_0\cdot \nabla z_{+} =-\nabla p-z_{-}\cdot\nabla z_{+},\\
	&\D_{t}z_{-}+B_0\cdot \nabla z_{-} =-\nabla p-z_{+}\cdot\nabla z_{-},\\
	&\operatorname{div} z_{+}=0,\\
	&\operatorname{div} z_{-}=0.
\end{cases}\end{equation}
For a vector field $f$ on $\mathbb{R}^3$, its curl is defined as 
$\operatorname{curl} f = \big(\varepsilon_{ijk}\partial_if^j\big) \partial_k$, where $\varepsilon_{ijk}$ is a totally anti-symmetric symbol associated to the volume form of $\mathbb{R}^3$ and repeated indices are understood as summations. If we define
\begin{equation*}
\begin{cases}
&j_{+}=\operatorname{curl} z_{+},\\ 
&j_{-}=\operatorname{curl} z_{-},
\end{cases}
\end{equation*}	
by taking the curl of \eqref{eq:MHD}, we obtain the following system of equations for $(j_{+},j_{-})$:
\begin{equation}\label{eq:curlMHD-2}\begin{cases}
	&\D_{t}j_{+}-B_{0}\cdot \nabla j_{+} =-\nabla z_{-}\wedge\nabla z_{+}-z_{-}\cdot\nabla j_{+},\\
	&\D_{t}j_{-}+B_{0}\cdot \nabla j_{-} =-\nabla z_{+}\wedge\nabla z_{-}-z_{+}\cdot\nabla j_{-},\\
	&\operatorname{div} j_{+}=0,\\
	&\operatorname{div} j_{-}=0.	
\end{cases}\end{equation}
In the above equations, the wedge products are understood as follows:
\begin{equation*}
\begin{cases}
	&\nabla z_- \wedge \nabla z_+ =\big(\varepsilon_{ijk}\partial_i z_-^l\partial_l z_+^j\big)\partial_k, \\ 
	&\nabla z_+ \wedge \nabla z_- =\big(\varepsilon_{ijk}\partial_i z_+^l\partial_l z_-^j\big)\partial_k.
	\end{cases}
\end{equation*}
				
\bigskip

We will briefly summarize the progress on small data theory for three-dimensional incompressible MHD systems with strong magnetic field backgrounds. The pioneering work \cite{Bardos} of Bardos, Sulem and Sulem established the global existence result in the H\"older space $C^{1,\alpha}$ for the ideal case by means of the convolution with fundamental solutions. In the viscous cases, Lin, Xu and Zhang \cite{Lin-Xu-Zhang, Xu-Zhang} used Fourier method to obtain global solutions in the case where the system has the fluid viscosity  but does not have Ohmic dissipation. The smallness of the data is relative to the viscosity so that the method to study Navier-Stokes equations can be adapted. A major step to understand small diffusion regimes was made by He, Xu and Yu in \cite{He-Xu-Yu} by using the energy methods in the physical space. They proved the global nonlinear stability for both the ideal case and the case with small diffusion, where the small diffusion means that the data is independent of the viscosity coefficients. The work \cite{He-Xu-Yu} has been extended in several aspects: one follow-up is the work \cite{Cai-Lei} where the authors also showed the global existence for 2-dimensional MHD systems; another one is the work \cite{Wei-Zhang} where the authors can deal with the case where the fluid viscosity and Ohmic viscosity are slightly different; Xu \cite{Xu} has provided a beautiful approach to derive the 2-dimensional MHD solutions as the limit of 3-dimensional solutions in a thin slab; the method of the proof has been adapted in \cite{G-Yang-Yu} to prove that for semi-linear wave equations on $\mathbb{R}^{1+1}$, if the nonlinearity satisfies the null conditions, the Cauchy problem admits global solutions in the small data regime.

The global existence result is still of great interests in the current work. We will study the scattering behavior of the global solution, i.e., the traces of solutions along the characteristic lines. The scattering theory for waves is an old and classical topic and we refer the readers to the text book \cite{L-P1} for a detailed history. To motivate the rigidity result for Alfv\'en waves, we will give a brief account on several results concerning the scattering uniqueness for free waves. For a free wave $\phi(t,x)$ in three dimensions, that is, a smooth solution of the linear wave equation
\[\Box\phi=0\] 
on $\mathbb{R}^{3+1}$, if the initial data $\phi(0,x)$ and $\big(\partial_t\phi\big)(0,x)$ decay nicely where $|x|\rightarrow \infty$, the solution $\phi(t,x)$ enjoys the following decay estimate:
\[\big|\phi(t,x)\big|\leqslant \frac{C}{1+|t|},\]
where the constant $C$ depends on $\phi(0,x)$ and $\big(\partial_t\phi\big)(0,x)$. We consider $\psi(t,x)=|x|\phi(t,x)$ and we foliate $\mathbb{R}^{1+3}_{t\geqslant 0}$ by outgoing light-cones emanating from the Cauchy hypersurface $\mathbb{R}^3$. We recall that such a light-cone is defined by 
\[\mathbf{C}_c=\big\{(t,x)\big| t-|x| =c\big\},\]
 where $c\in\mathbb{R}$. According to the decay estimate of the free wave, on any outgoing light-cone, 
\[\psi(t,x)=|x|\phi(t,x)=(t-c)\phi(t,x)\]
is a bounded function. The remarkable fact is that  if we let $t\rightarrow \infty$ on a light-cone $\mathbf{C}_c$, the limit exists. More precisely, by virtue of the spherical coordinate system $(r,\theta,\varphi)$ on $\mathbb{R}^3$, the following limit
\[\psi_{\infty}(c,\theta,\varphi)=\lim_{t\rightarrow \infty}r\phi(r+c,r,\theta,\varphi)\]
exists. We refer to $\psi_{\infty}(c,\theta,\varphi)$ as the \emph{scattering field} of the free wave $\phi(t,x)$. Using the above terminologies, the rigidity of free waves from infinity on $\mathbb{R}^{3+1}$ can be stated as follows:

\medskip

\begin{center}
If the scattering field $\psi_{\infty}(c,\theta,\varphi)$ vanishes identically, then $\phi(t,x)\equiv 0$.
\end{center}

\medskip

A standard idea to prove the statement is to use Radon transformation. This is because the explicit solution formula for three dimensional wave equations can be explained as the Radon transform of the initial data. The readers may find details and more related topics and more references in \cite{John, Helgason,L-P1,L-P2,Ludwig,Zalcman}. These types of results are also known as unique continuation (from infinity) theorems for wave equations. We also point out some latest development in the field of unique continuation for wave equations. Ionescu and Klainerman (and later with Alexakis) have proposed a program \cite{Alexakis_Ionescu_Klainerman, Alexakis_Ionescu_Klainerman_Perturbation, Alexakis_Ionescu_Klainerman_Duke, Ionescu_Klainerman_Kerr, Ionescu_Klainerman_Wave, Ionescu_Klainerman_JAMS} to use the unique continuation to study the uniqueness of black holes in a smooth class and not imposing axial symmetry. The paper \cite{Ionescu_Klainerman_Survey} has surveyed the most updated results in this direction. We remark that all these works rely on the Carleman estimates adapted for finite null cones. The Carleman estimates can also be proved from infinity hence applied to study the unique continuation as for the aforementioned example of free wave. In \cite{A-S-S}, based on the construction of pseudo-convex functions and the Carleman estimates, Alexakis, Schlue and Shao proved unique continuation results from infinity for wave equations of the following form
  	\[\Box_g \phi + a^\alpha \partial_\alpha \phi + V \phi = 0\]
  	over Minkowski spacetime: given infinite-order vanishing of the radiation field  at suitable parts of null infinities, then the solution $\phi$ must vanish in an open set in the interior. More precisely, for zero-mass spacetimes, such as perturbations of Minkowski spacetimes, they proved local unique continuation given vanishing on more than half of both future and past null infinity;
  for positive-mass spacetimes, they proved local unique continuation given vanishing on an arbitrarily small part of null infinity near spacelike infinity. The results have also been extended to asymptotically flat spacetimes, such as Schwarzschild spacetimes and Kerr spacetimes. In \cite{A-S}, Alexakis and Shao proved a global rigidity version of the unique continuation result without the assumption of infinite-order vanishing at 
  infinity: for a solution $\phi$ of the following linear wave equation
  	\[\Big(\Box+{V}(t,x)\Big)\phi=0\]
  	over Minkowski spacetime where $V(t,x)$ is a  nice potential or nonlinearity, from half of both future and past null infinity, given finite-order vanishing ($\delta$-order for certain lower-order terms) of the radiation fields of $\phi$  and additional global regularity assumptions, then the solution itself must vanish everywhere.
					
\medskip
					
The rigidity results for Alfv\'en waves in the current paper have the similar flavor as the aforementioned works. Though we study the MHD systems in three dimensions, the Alfv\'en waves in a strong magnetic background behave more like one dimensional waves, e.g., they have no decay in time. Therefore, to motivate the main theorem, it is worth studying a much simpler but enlightening example in details: the one dimensional linear wave equation on $\mathbb{R}^{1+1}$. Let $\Box =-\partial_t^2+\partial_x^2$ be the standard one dimensional wave operator, we consider the following Cauchy problem:
\begin{equation}\label{MAINEQ}\begin{cases}
	&\Box \phi = 0,\\
	&\big(\phi,\partial_t\phi\big)\big|_{t=0}=(\phi_0(x),\phi_1(x)).
\end{cases}\end{equation}
For the sake of simplicity, we may assume that $\phi_0(x)$ and $\phi_1(x)$ are smooth functions with compact support. Its solutions can be written as a superposition of left-traveling and right-traveling waves:
\[\phi(t,x)=\phi_+(x-t)+\phi_-(x+t),\]
where both $\phi_+$ and $\phi_-$ are smooth functions with compact support. In fact, their derivatives are given by
\[\begin{cases}
	&\phi_-'(x)=\frac{1}{2}\left(\phi_0'(x)+\phi_1(x)\right),\\
	& \phi_+'(x)=\frac{1}{2}\left(\phi_0'(x)-\phi_1(x)\right).
	\end{cases}
\]
We use the null coordinates $(u,\underline{u})$:
\begin{equation*}
\begin{cases}
	&u=x-t, \\ 
	&\underline{u}=x+t,
	\end{cases}
\end{equation*}
and the null frame $(L,\underline{L})$:
\begin{equation*}
\begin{cases}
	&L=\partial_t+\partial_x, \\ 
	&\underline{L}=\partial_t-\partial_x.
	\end{cases}
\end{equation*}
On $\mathbb{R}^{1+1}$, for all real numbers $t_0>0$, $u_0$ and $\underline{u}_0$, we define the time slice $\Sigma_{t_0}$, the right-going null curve segment $C_{u_0}^{t_0}$ and the left-going null curve segment $\underline{C}_{\underline{u}_0}^{t_0}$ as follows:
\begin{align*}
	\Sigma_{t_0} &=\big\{(t,x) \,\big|\, t= t_0\big\},\\
	C_{u_0}^{t_0} &=\big\{(t,x) \,\big|\, u=x-t=u_0, 0\leqslant t \leqslant t_0\big\},\\
   \underline{C}_{\underline{u}_0}^{t_0} &=\big\{(t,x) \,\big|\, \underline{u} =x+t=\underline{u}_0, 0\leqslant t\leqslant t_0\big\}.
\end{align*}
Given a point $(0,x_0) \in \Sigma_0$, it determines uniquely a left-traveling characteristic line $\underline{\ell}(x_0)$: 
\[\underline{\ell}(x_0)=\big\{(\underline{u},t)\big|\underline{u}=x_0, t\in \mathbb{R}\big\}.\]
We use $\mathcal{F}$ to denote the collection of all the left-traveling characteristic lines:
\[\mathcal{F}=\big\{\underline{\ell}(\underline{u})\big|\underline{u}\in \mathbb{R}\big\},\]
and we call it \emph{the left future characteristic infinity}. In addition to just being a set, equipped with the global coordinate system $(\underline{u})$, $\mathcal{F}$ can be regarded as a differentiable manifold. 
Similarly, for $(0,x_0) \in \Sigma_0$, it determines uniquely a right-traveling characteristic line ${\ell}(x_0)$: 
\[{\ell}(x_0)=\big\{({u},t)\big|{u}=x_0, t\in \mathbb{R}\big\}.\]
We use ${\underline{\mathcal{F}}}$ to denote the collection of all the right-traveling characteristic lines:
\[{\underline{\mathcal{F}}}=\big\{{\ell}({u})\big| {u}\in \mathbb{R}\big\},\]
and we call it \emph{the right future characteristic infinity}. Using the global coordinate system $({u})$, ${\underline{\mathcal{F}}}$ can also be regarded as a differentiable manifold.

\begin{center}
\includegraphics[width=4in]{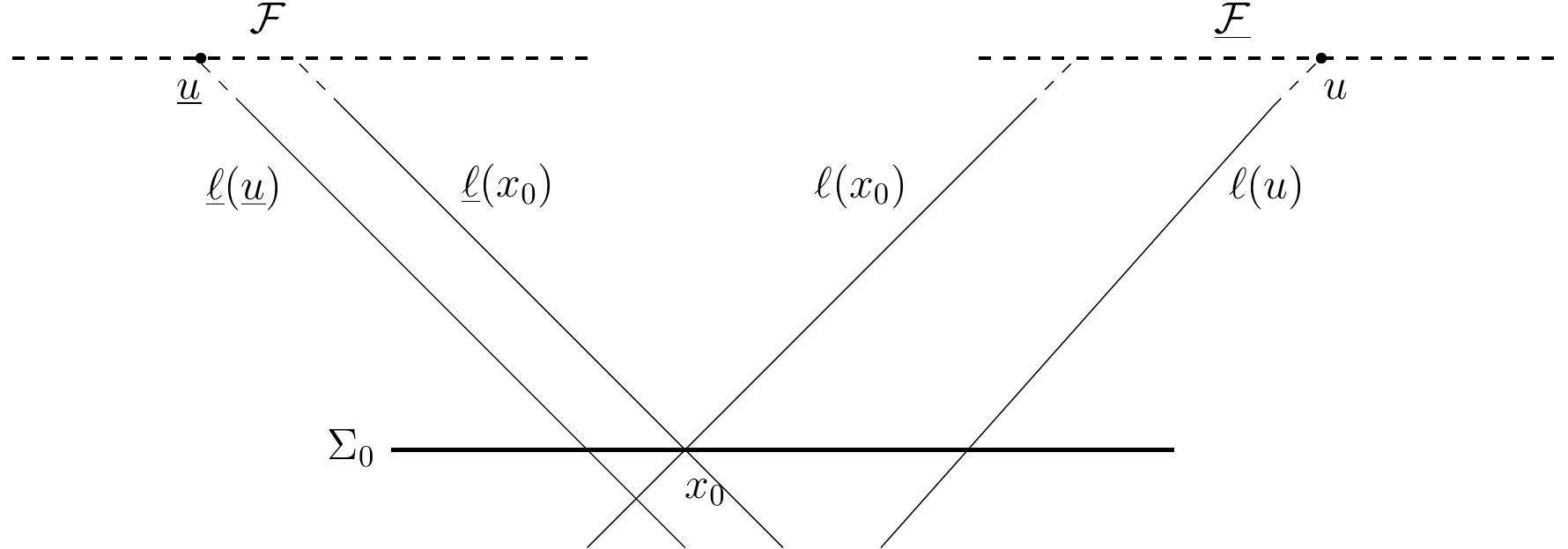}
\end{center}

The geometry is illustrated in the above picture. Heuristically,  a right-traveling characteristic line $\ell(x_0)$ passes the point $(0,{x_0}) \in \Sigma_0$ and hits $\underline{\mathcal{F}}$ at the point $u=x_0$. We can integrate the equation $\Box \phi =0$ on $C_{u}^t$ (this is a segment of $\ell(u)$) and we obtain
\[\underline{L}\phi(t,u+t)=\underline{L}\phi(0,u).\]
Let $t\rightarrow +\infty$. Hence, we define the scattering field $\underline{L}\phi(+\infty;u)$ on $\underline{\mathcal{F}}$ as
\[\underline{L}\phi(+\infty;u)=\lim_{t\to+\infty}\underline{L}\phi(t,u+t)=\underline{L}\phi(0,u),\]
where we use the coordinate system $u$ on $\underline{\mathcal{F}}$.

Similarly, we can define the scattering field ${L}\phi(+\infty;\underline{u})$ on $\mathcal{F}$ as
\[ {L}\phi(+\infty;\underline{u})=\lim_{t\to+\infty} {L}\phi(t,\underline{u}-t)={L}\phi(0,\underline{u}).\] 
The rigidity theorem in this case is obvious: if the scattering fields vanish at infinities, i.e., 
\begin{equation}\label{scattering_model1}\begin{cases}
&\underline{L}\phi(+\infty;u)\equiv 0, \ \ \text{on}~\underline{\mathcal{F}},\\ 
&L\phi(+\infty;\underline{u})\equiv 0, \ \ \text{on}~\mathcal{F},
\end{cases}
\end{equation}
then we have $\underline{L}\phi(0,u)=0$ and $L\phi(0,\underline{u})=0$ for all $t$ and $x$. Hence, $\phi(t,x)\equiv 0$ on $\mathbb{R}^{1+1}$.

\medskip

In the same manner, we can define the past characteristic infinities. We use $\mathcal{P}$ to denote the collection of all the right-traveling (to the past) characteristic lines:
\[\mathcal{P}=\big\{\underline{\ell}(\underline{u})\big|\underline{u}\in \mathbb{R}\big\}.\]
As a set, this is the same as $\mathcal{F}$. We use a different name \emph{the right past characteristic infinity} to call it, because we will consider the traces of the solutions for $t\rightarrow -\infty$. We also use the global coordinate system $(\underline{u})$ to define the differentiable structure on $\mathcal{P}$. 
Similarly, we use ${\underline{\mathcal{P}}}$ to denote the collection of all the left-traveling (to the past) characteristic lines:
\[{\underline{\mathcal{P}}}=\big\{{\ell}({u})\big| {u}\in \mathbb{R}\big\},\]
and we call it \emph{the left past characteristic infinity}. Using the global coordinate system $({u})$, ${\underline{\mathcal{P}}}$ can also be regarded as a differentiable manifold. A similar heuristic argument suggests that a left-traveling characteristic line $\ell(x_0)$ passes the point $(0,{x_0}) \in \Sigma_0$ and hits $\underline{\mathcal{P}}$ at the point $u=x_0$. The geometry can be read off easily from the following picture.

\begin{center}
\includegraphics[width=2in]{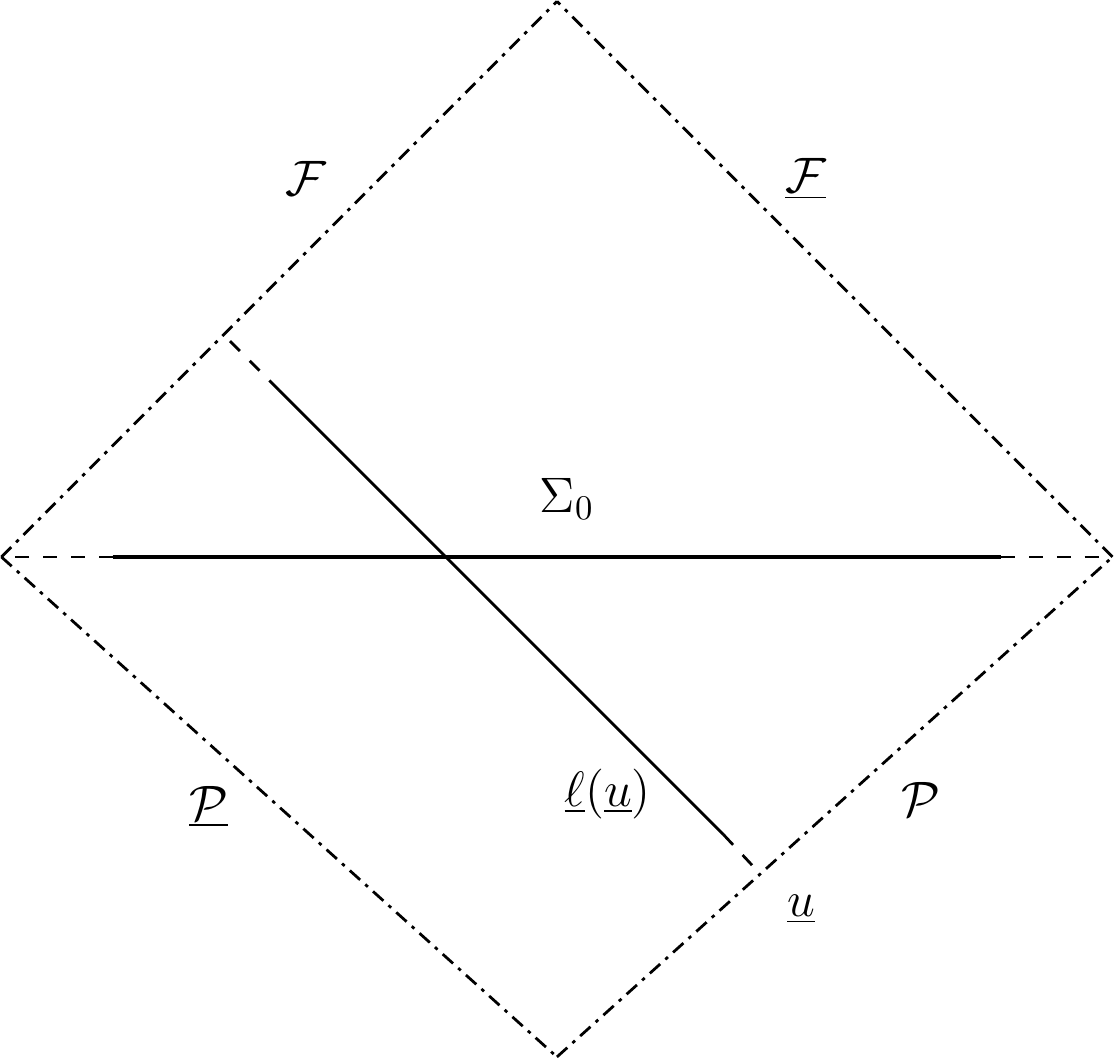}
\end{center}

We define the scattering field $\underline{L}\phi(-\infty;u)$ on $\underline{\mathcal{P}}$ as
\[\underline{L}\phi(-\infty;u)=\lim_{t\to-\infty}\underline{L}\phi(t,u+t)=\underline{L}\phi(0,u),\]
where we use the coordinate system $u$ on $\underline{\mathcal{P}}$. We can also define the scattering field ${L}\phi(-\infty;\underline{u})$ on $\mathcal{P}$ as
\[ {L}\phi(-\infty;\underline{u})=\lim_{t\to-\infty} {L}\phi(t,\underline{u}-t)={L}\phi(0,\underline{u}).\] 
There is another version of rigidity theorem: if the scattering field $\underline{L}\phi(+\infty;u)$ vanishes at the future infinity and the scattering field $L\phi(-\infty;\underline{u})$ vanishes at the past infinity, i.e.,
\begin{equation}\label{scattering_model2}\begin{cases}
&\underline{L}\phi(+\infty;u)\equiv 0, \ \ \text{on}~\underline{\mathcal{F}},\\ 
&L\phi(-\infty;\underline{u})\equiv 0, \ \ \text{on}~\mathcal{P},
\end{cases}
\end{equation}
then  $\underline{L}\phi(0,u)=0$ and $L\phi(0,\underline{u})=0$ for all $t$ and $x$. Hence, $\phi(t,x)\equiv 0$ on $\mathbb{R}^{1+1}$. This second rigidity theorem shares many similar features with results in \cite{Alexakis_Ionescu_Klainerman, Alexakis_Ionescu_Klainerman_Perturbation, Alexakis_Ionescu_Klainerman_Duke, Ionescu_Klainerman_Kerr, Ionescu_Klainerman_Wave, Ionescu_Klainerman_JAMS}. In particular, it resembles \cite{A-S-S} and \cite{A-S}.

\medskip
						
This paper is devoted to the study of the scattering fields of Alfv\'en waves and to the proof of a couple of rigidity theorems of Alfv\'en waves similar to the above examples, in despite of the nonlinear nature of the MHD equations. More precisely, we will not only prove the solution of the MHD system exists globally, but also prove the traces of the solution at characteristics infinities, i.e., the scattering fields, are well-defined. The main statement of the rigidity theorems claims that if the scattering fields of a solution vanish at infinities (at $\mathcal{F}\cup \underline{\mathcal{F}}$ or $\underline{\mathcal{F}}\cup \mathcal{P}$), then the solution itself vanishes identically. Compared to the works \cite{A-S-S} and \cite{A-S}, instead of requiring higher order derivatives of the scattering field to vanish, we only require that the scattering field itself is equal to $0$ at infinity. This is consistent with physical interpretations of scattering fields: the detecting fields of Alv\'en waves are the waves detected from a far-away observer. Therefore, the rigidity theorems have the following physical intuition: if no waves are  detected by the far-away observers, then there are no Alfv\'en waves at all emanating from the plasma. Nevertheless, the nonlinear nature of Alfv\'en waves makes the problem and the approach to it different from the former situations. The underlying idea of the analysis is similar to \cite{He-Xu-Yu,G-Yang-Yu}:  The strong magnetic background provides a null structure for nonlinear terms and this allows us to obtain weighted energy estimates. The rigidity theorems will follow from a careful choice of the weights.
						
\bigskip
					
In the rest of the section, we introduce necessary notations which enable us to give a precise statement of the main theorem.

\medskip
						
On $\mathbb{R}^{3+1}$, we define two characteristic (space-time) vector fields $L_{+}$ and $L_{-}$ as 
\[\begin{cases}&L_{+}=\D_t+B_0\cdot\nabla=\partial_t+\partial_{x_3}, \\  
&L_{-}=\D_t-B_0\cdot\nabla=\partial_t-\partial_{x_3}.
\end{cases}\]
We define two characteristic functions $u_{\pm}$ as 
\[\begin{cases}&u_{+}=x_3-t,\\
&u_{-}=x_3+t.
\end{cases}\]
Given two real numbers $u_{+,0}$ and $u_{-,0}$, the characteristic hypersurfaces $C^{+}_{u_{+,0}}$ and $C^{-}_{u_{-,0}}$ are defined as
\begin{align*}
&C^{+}_{u_{+,0}}=\big\{(t,x)\in \mathbb{R}\times \mathbb{R}^3\big|u_+(t,x) = u_{+,0}\big\},\\ 
&C^{-}_{u_{-,0}}=\big\{(t,x)\in \mathbb{R}\times \mathbb{R}^3\big|u_-(t,x) = u_{-,0}\big\}.
\end{align*}
For $t\geqslant 0$, we also define 
\begin{align*}
&C^{+,t}_{u_+}=C^{+}_{u_+} \cap \big([0,t] \times \mathbb{R}^3\big),\\
&C^{-,t}_{u_-}=C^{-}_{u_-}\cap \big([0,t] \times \mathbb{R}^3\big).
\end{align*} 
For $t^*\geqslant 0$, the spacetime slab $[0,t^*] \times \mathbb{R}^3$ admits a natural time foliation  $\displaystyle \bigcup_{0\leqslant t \leqslant t^*} \Sigma_t$ and two characteristic foliations  $\displaystyle\bigcup_{u_+ \in \mathbb{R}}C^{+,t^*}_{u_+}$ and $\displaystyle\bigcup_{u_- \in \mathbb{R}}C^{-,t^*}_{u_-}$.

We fix a small number $\delta>0$ once for all in this paper ($\delta=0.1$ suffices) and let $\omega=1+\delta$. Let $a\in \mathbb{R}$ be a  constant and it will be determined in the course of proving the rigidity theorems. We call $a$ the position parameter which indeed tracks the centers of the Alfv\'en waves. We remark that the energy estimates derived in the paper will be independent of the choice of $a$. We introduce two weight functions $\langle u_+\rangle$ and $\langle u_-\rangle$ as 
\begin{align*}
	&\langle u_{+}\rangle=(1+|u_+-a|^2)^{\frac{1}{2}}=(1+|x_3-(t+a)|^2)^{\frac{1}{2}},\\
	&\langle u_{-}\rangle=(1+|u_-+a|^2)^{\frac{1}{2}}=(1+|x_3+(t+a)|^2)^{\frac{1}{2}}.
\end{align*}
These two functions depend on $a$ in an obvious way. We remark that 
\[L_+ \langle u_+\rangle = L_- \langle u_-\rangle =0.\]  
						
\smallskip

We turn to the definition of energy norms. For any multi-index $\alpha=(\alpha_1,\alpha_2,\alpha_3)$ with $\alpha_i \in \mathbb{Z}_{\geqslant 0}$ ($i=1,2,3$), we use $\partial^\alpha$ as the shorthand notation for the differential operator $\partial^{\alpha_1}_{x_1}\partial^{\alpha_2}_{x_2}\partial^{\alpha_1}_{x_3}$. We also define $|\alpha|=\alpha_1+\alpha_2+\alpha_3$. For a given multi-index $\alpha$, we define
\begin{equation*}
\begin{cases}
&z_+^{(\alpha)}=\partial^\alpha z_+, \ \ j_{+}^{(\alpha)}=\operatorname{curl} z^{(\alpha)}_{+},\\ 
&z_-^{(\alpha)}=\partial^\alpha z_-, \ \ j_{-}^{(\alpha)}=\operatorname{curl} z^{(\alpha)}_{-}.
\end{cases}
\end{equation*}	

The basic energy norm through $\Sigma_t$ is defined as
\begin{equation*}
	E_{\mp}(t)=\int_{\Sigma_t}\langle u_{\pm}\rangle^{2\omega}| z_{\mp}|^2,
\end{equation*}
and higher order energy norms on $\Sigma_t$ are defined as
\begin{equation*}
	E^{(\alpha)}_{\mp}(t)=\int_{\Sigma_t}\langle u_{\pm}\rangle^{2\omega}\big|j^{(\alpha)}_{\mp}\big|^2,
\end{equation*}
where the integral should be understood as 
\[\int_{\Sigma_{t}}f :=\int_{\Sigma_{t}}f(t,x)dx,\]
i.e., we use the induced measure on $\Sigma_t$ where $\Sigma_t$ is regarded as an embedded linear subspace of $\mathbb{R}\times \mathbb{R}^3$.

For $t\geqslant 0$, the basic flux norms through $C^{\mp,t}_{u_{\mp}}$ are defined as
\begin{equation*}
	 F_{\mp}(t,u_\mp)=\int_{C^{\mp,t}_{u_{\mp}}}\langle u_{\pm}\rangle^{2\omega}|z_{\mp}|^2, \ \ F_{\mp}(t)=\sup_{u_{\mp}\in\mathbb{R}}F_{\mp}(t,u_\mp),
\end{equation*}
and the higher order flux norms through $C^{\mp,t}_{u_{\mp}}$ are defined as
\begin{equation*}
	F^{(\alpha)}_{\mp}(t,u_\mp)=\int_{C^{\mp,t}_{u_{\mp}}}\langle u_{\pm}\rangle^{2\omega}\big|j^{(\alpha)}_{\mp}\big|^2, \ \ F^{(\alpha)}_{\mp}(t)=\sup_{u_{\mp}\in\mathbb{R}}F^{(\alpha)}_{\mp}(t,u_\mp),
\end{equation*}
where the integral should be understood as surface integral
\begin{equation*}
	\int_{C^{\mp,t}_{u_{\mp}}}f:=\int_{C^{\mp,t}_{u_{\mp}}}f(t,x)d\sigma_{\mp},
\end{equation*}
i.e., we use the induced measure on $C^{\mp,t}_{u_{\mp}}$ where $C^{\mp,t}_{u_{\mp}}$ is regarded as an embedded linear subspace of $\mathbb{R}\times \mathbb{R}^3$.

For a given $t_*\in [0,+\infty]$, we also define the total energy norms and total flux norms indexed by a number $k\in \mathbb{Z}_{\geqslant 0}$:
\begin{equation*}\begin{cases}
	&E_{\mp} =\displaystyle\sup_{0\leqslant t\leqslant t^{*}}E_{\mp}(t),\ \ E^{k}_{\mp}=\sup_{0\leqslant t\leqslant t^{*}}\sum_{|\alpha|=k}E^{(\alpha)}_{\mp}(t),\\
	&F_{\mp} =\displaystyle\sup_{0\leqslant t\leqslant t^{*}}F_{\mp}(t),\ \ F^{k}_{\mp}=\sup_{0\leqslant t\leqslant t^{*}}\sum_{|\alpha|=k}F^{(\alpha)}_{\mp}(t).
\end{cases}\end{equation*}

\bigskip

The key ingredient of the work is the following \emph{a priori} energy estimates. The global existence of solutions to the ideal MHD system follows immediately from the estimates. We remark that, though similar estimates have also been established in \cite{He-Xu-Yu}, the estimates of the current work are independent of the position parameter $a$. Moreover, the choice of the position parameter plays a central role in order to prove the rigidity theorem for Alfv\'en waves.
\begin{MainEstimates}
Let $\delta \in(0,\frac{2}{3})$ and $N_* \in \mathbb{Z}_{\geqslant 5}$. There exists a universal constant $\varepsilon_0\in(0,1)$ such that if the initial data $\big(z_+(0,x),z_-(0,x)\big)$ of \eqref{eq:MHD} satisfy
\begin{equation*}
	\mathcal{E}^{N_*}(0) =\sum_{+,-}\sum_{k=0}^{N_*+1}\big\|(1+|x_3\pm a|^2)^{\frac{1+\delta}{2}}\nabla^{k} z_{\pm}(0,x)\big\|_{L^2(\mathbb{R}^3)}^2\leqslant\varepsilon_0^2,
\end{equation*}
then the ideal MHD system \eqref{eq:MHD}  admits a unique global solution $\big(z_+(t,x),z_-(t,x)\big)$. 
Moreover, there is a universal constant $C$ such that the following energy estimates hold:
\begin{equation*}
\begin{cases}
	&\displaystyle
	\sum_{k=0}^{N_*+1}\sup_{t\geqslant 0}\big\|(1+|u_\mp \pm a|^2)^{\frac{1+\delta}{2}}\nabla^{k}z_{\pm}(t,x)\big\|_{L^2(\mathbb{R}^3)}^2 \leqslant C \mathcal{E}^{N_*}(0),\\
&\displaystyle\sup_{t\geqslant 0}\sup_{u_{\pm}}\int_{C_{u_\pm}^{\pm,t}}(1+|u_\mp \pm a|^2)^{1+\delta}|z_{\pm}(t,x)|^2 +\sum_{k=0}^{N_*}\sup_{t\geqslant 0}\sup_{u_{\pm}}\int_{C_{u_\pm}^{\pm,t}}(1+|u_\mp \pm a|^2)^{1+\delta}\big|j^{(k)}_{\pm}(t,x)\big|^2 \leqslant C \mathcal{E}^{N_*}(0).
\end{cases}
\end{equation*}
\end{MainEstimates}

\medskip
The proof of the above main estimates indeed provides a refined estimate. This manifests the null structure of the nonlinear terms of MHD systems with strong magnetic backgrounds.
\begin{RMainEstimates}
	Let $\delta \in(0,\frac{2}{3})$ and $N_* \in \mathbb{Z}_{\geqslant 5}$. There exists a universal constant $\varepsilon_0\in(0,1)$ such that if the initial data $\big(z_+(0,x),z_-(0,x)\big)$ of \eqref{eq:MHD} satisfy
\begin{equation*}
  \mathcal{E}^{N_*}_{\pm}(0) =\sum_{k=0}^{N_*+1}\big\|(1+|x_3\pm a|^2)^{\frac{1+\delta}{2}}\nabla^{k} z_{\pm}(0,x)\big\|_{L^2(\mathbb{R}^3)}^2\leqslant\varepsilon_{\pm,0}^2
\end{equation*}
  and 
  \begin{equation*}
	\mathcal{E}^{N_*}(0)=\mathcal{E}^{N_*}_{+}(0)+\mathcal{E}^{N_*}_{-}(0)\leqslant\varepsilon_{0}^2,
\end{equation*}
	where $\varepsilon_{0}^2=
	\varepsilon_{+,0}^2+\varepsilon_{-,0}^2$,
	then the global solution $\big(z_+(t,x),z_-(t,x)\big)$ to the ideal MHD system \eqref{eq:MHD} satisfies the following estimates: there is a universal constant $C$ such that
\begin{align*}
\displaystyle
	&\sum_{k=0}^{N_*+1}\sup_{t\geqslant 0}\big\|(1+|u_- + a|^2)^{\frac{1+\delta}{2}}\nabla^{k}z_{+}(t,x)\big\|_{L^2(\mathbb{R}^3)}^2 \leqslant C \mathcal{E}^{N_*}_{+}(0)+C\big(\mathcal{E}^{N_*}_{+}(0)\big)^2\mathcal{E}^{N_*}_{-}(0),\\
	&\sum_{k=0}^{N_*+1}\sup_{t\geqslant 0}\big\|(1+|u_+ - a|^2)^{\frac{1+\delta}{2}}\nabla^{k}z_{-}(t,x)\big\|_{L^2(\mathbb{R}^3)}^2 \leqslant C \mathcal{E}^{N_*}_{-}(0)+C\big(\mathcal{E}^{N_*}_{-}(0)\big)^2\mathcal{E}^{N_*}_{+}(0).
\end{align*}
\end{RMainEstimates}

\begin{Remark}
The constants $\varepsilon_0$ and $C$ are independent of the choice of the position parameter $a$.
\end{Remark}							
\begin{Remark}
We point out that these results hold for all $N_* \in \mathbb{Z}_{\geqslant 5}$. In the rest of the paper, we will take $N_*=7$ in the construction of solutions and $N_*=6$ in applications. 
In the Main Energy Estimates, we can also assume that $a=a_0=0$ is fixed so that we construct solutions, although the estimates are independent of $a_0$. 
\end{Remark}	

\bigskip

We now assume that $\big(z_+(t,x),z_-(t,x)\big)$ is the solution constructed from above and we will define the future scattering fields associated to $\big(z_+(t,x),z_-(t,x)\big)$. Towards this goal, we first define the appropriate geometric objects of the scattering fields: the future infinities. 

Given a point $(0, x_1,x_2,x_3) \in \Sigma_0$, it determines uniquely a left-traveling straight line $\ell_-$ parameterized by 
\[\ell_-: \mathbb{R}\rightarrow \mathbb{R}\times \mathbb{R}^3, \ \ t\mapsto (x_1,x_2,x_3+t,t).\]
We remark that  $u_-\big|_{\ell_-}\equiv x_3$.  Since the Cartesian coordinate functions $x_1$ and $x_2$ are also constants on $\ell_-$, we also denote the line by $\ell_-(x_1,x_2,u_-)$ where $x_1$, $x_2$ and $u_-$ are constants. In particular, $\ell_-(x_1,x_2,u_-)\subset C^-_{u_-}$. We use $\mathcal{F}_+$ to denote the collection of all the left-traveling lines:
\[\mathcal{F}_+=\big\{\ell_-(x_1,x_2,u_-)\big|(x_1,x_2,u_-)\in \mathbb{R}^3\big\},\] 
and we call $\mathcal{F}_+$ \emph{the left future infinity}. More than just being a set,  $\mathcal{F}_+$ can be regarded as an \emph{Euclidean space} if we use $(x_1,x_2,u_-)$ as a fixed global coordinate system on $\mathcal{F}_+$.

Similarly, for $(0, x_1,x_2,x_3) \in \Sigma_0$, it defines a right-traveling straight line
\[\ell_+: \mathbb{R}\rightarrow \mathbb{R}\times \mathbb{R}^3, \ \ t\mapsto (x_1,x_2,x_3-t,t).\]
On account of $u_+\big|_{\ell_+}\equiv x_3$, we also denote the line by $\ell_+(x_1,x_2,u_+)$ where $x_1$, $x_2$ and $u_+$ are constants.  We use $\mathcal{F}_-$ to denote the collection of all the right-traveling lines:
\[\mathcal{F}_-=\big\{\ell_+(x_1,x_2,u_+)\big|(x_1,x_2,u_+)\in \mathbb{R}^3\big\},\] 
and we call $\mathcal{F}_-$ \emph{the right future infinity}. We use $(x_1,x_2,u_+)$ as a fixed global coordinate system on $\mathcal{F}_-$ to make $\mathcal{F}_-$ to be an \emph{Euclidean space}.

\begin{center}
\includegraphics[width=4in]{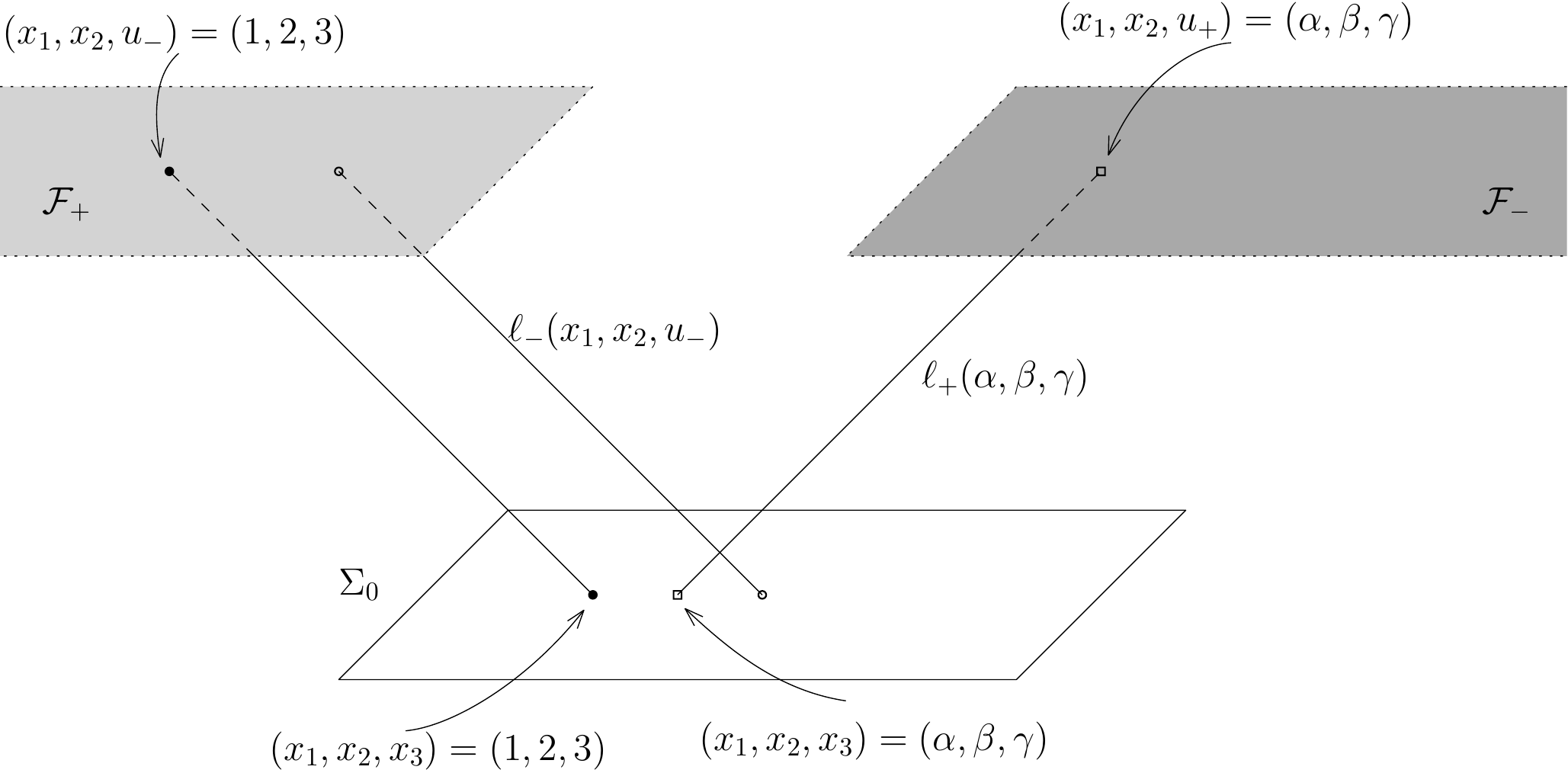}
\end{center}

The previous descriptions can be schematically depicted in the above picture.  We see that $\ell_+(\alpha,\beta,\gamma)$ passes the point $(\alpha,\beta,\gamma) \in \Sigma_0$ and we may think it hits $\mathcal{F}_-$ at the point $(x_1,x_2,u_+)=(\alpha,\beta,\gamma)$. The future infinities $\mathcal{F}_+$ and $\mathcal{F}_-$ are the spaces where the future scattering fields live.		

\bigskip

For a fixed point $p_0=(0,\alpha,\beta,\gamma)$ on $\Sigma_0$, we consider a point $p_{t}=(t,\alpha,\beta,\gamma+t)$ on $\ell_+(\alpha,\beta,\gamma)$, where $t\geqslant 0$. According to \eqref{eq:MHD}, we have
\begin{equation*}
\frac{d}{d\tau}\Big(z_{-}(\tau,\alpha,\beta,\gamma+\tau)\Big) =-\nabla p(\tau,\alpha,\beta,\gamma+\tau)-\big(z_{+}\cdot\nabla z_{-}\big)(\tau,\alpha,\beta,\gamma+\tau).
\end{equation*}
Thus, by integrating this equation along the segment	of $\ell_+(\alpha,\beta,\gamma)$ between $p_0$ and $p_t$, we obtain
\begin{equation*}
	z_-(t,\alpha,\beta,\gamma+t)=z_-(0,\alpha,\beta,\gamma)-\int_0^t\big(\nabla p+z_{+}\cdot\nabla z_{-}\big)(\tau,\alpha,\beta,\gamma+\tau)d\tau.
\end{equation*}
If we take $t\rightarrow +\infty$, we may think $p_t$ converges to $p_{+\infty}=(\alpha,\beta,\gamma)\in \mathcal{F}_-$. We expect that the right hand side of the above equation admits a limit so that it can be regarded as the value of $z_-$ at a point on $\mathcal{F}_-$. Thus, for $(x_1,x_2,u_+)=(\alpha,\beta,\gamma)$, we expect the following expression defines the scattering field $z_-(+\infty;\alpha,\beta,\gamma)$:
\begin{equation*}
	z_-(+\infty;\alpha,\beta,\gamma)=z_-(0,\alpha,\beta,\gamma)-\int_0^{+\infty} \big(\nabla p+z_{+}\cdot\nabla z_{-}\big)(\tau,\alpha,\beta,\gamma+\tau)d\tau.
\end{equation*}
We indeed have
\begin{ExistenceTheorem} For the solution constructed in the statement of the Main Energy Estimates, the following integrals
\begin{equation*}
\begin{cases}
	&\displaystyle z_+(+\infty;x_1,x_2,u_-):=z_+(0,x_1,x_2,u_-)-\int_0^{+\infty} \big(\nabla p+z_{-}\cdot\nabla z_{+}\big)(\tau,x_1,x_2,u_--\tau)d\tau,\\
	&\displaystyle z_-(+\infty;x_1,x_2,u_+):=z_-(0,x_1,x_2,u_+)-\int_0^{+\infty} \big(\nabla p+z_{+}\cdot\nabla z_{-}\big)(\tau,x_1,x_2,u_++\tau)d\tau
	\end{cases}
\end{equation*}
converge. The above formulas define two vector fields $z_+(+\infty;x_1,x_2,u_-)$ and $z_-(+\infty;x_1,x_2,u_+)$ on the future infinities $\mathcal{F}_+$ and $\mathcal{F}_-$ respectively. We call $\big((z_+(+\infty;x_1,x_2,u_-),z_-(+\infty;x_1,x_2,u_+)\big)$ the  {\bf future scattering fields} associated to the solution $\big(z_+(t,x),z_-(t,x)\big)$.
\end{ExistenceTheorem}

\bigskip

We also define the past scattering fields associated to $\big(z_+(t,x),z_-(t,x)\big)$ in the same manner.

We use $\mathcal{P}_+$ to denote the collection of all the right-traveling (to the past) lines:
\[\mathcal{P}_+=\big\{\ell_-(x_1,x_2,u_-)\big|(x_1,x_2,u_-)\in \mathbb{R}^3\big\},\] 
and we call $\mathcal{P}_+$ \emph{the right past infinity}. We remark that, as a set, it is the same as $\mathcal{F}_+$. We use $(x_1,x_2,u_-)$ as the global coordinate system so that $\mathcal{P}_+$ can be regarded as an \emph{Euclidean space}. Similarly, we use $\mathcal{P}_-$ to denote the collection of all the left-traveling (to the past) lines:
\[\mathcal{P}_-=\big\{\ell_+(x_1,x_2,u_+)\big|(x_1,x_2,u_+)\in \mathbb{R}^3\big\},\] 
and we call $\mathcal{P}_-$ \emph{the left past infinity}. We use $(x_1,x_2,u_+)$ as a fixed global coordinate system on $\mathcal{P}_-$ to make $\mathcal{P}_-$ to be an \emph{Euclidean space}. Heuristically, we may think $\ell_+(\alpha,\beta,\gamma)$ passes the point $(\alpha,\beta,\gamma) \in \Sigma_0$ and we may think it hits $\mathcal{P}_-$ at the point $(x_1,x_2,u_+)=(\alpha,\beta,\gamma)$. The past infinities $\mathcal{P}_+$ and $\mathcal{P}_-$ are the spaces where the past scattering fields live.		
	
\bigskip
	
For a given point $p_0=(0,\alpha,\beta,\gamma)$ on $\Sigma_0$, we consider a point $p_{t}=(t,\alpha,\beta,\gamma+t)$ on $\ell_+(\alpha,\beta,\gamma)$, where $t\leqslant 0$. 
An argument analogous to the one used to yield $z_-(+\infty;\alpha,\beta,\gamma)$ naturally follows, with the major change  substituting $\mathcal{P}_-$ for $\mathcal{F}_-$and $t\to-\infty$ for $t\to+\infty$. Namely,
for $(x_1,x_2,u_+)=(\alpha,\beta,\gamma)$, we expect the following expression defines the scattering field $z_-(-\infty;\alpha,\beta,\gamma)$:
\begin{equation*}
	z_-(-\infty;\alpha,\beta,\gamma)=z_-(0,\alpha,\beta,\gamma)-\int_0^{-\infty} \big(\nabla p+z_{+}\cdot\nabla z_{-}\big)(\tau,\alpha,\beta,\gamma+\tau)d\tau.
\end{equation*}
We have a parallel existence theorem as future scattering fields:
\begin{ExistenceTheoremP} For the solution constructed in the statement of the Main Energy Estimates, the following integrals
\begin{equation*}
\begin{cases}
&\displaystyle z_+(-\infty;x_1,x_2,u_-):=z_+(0,x_1,x_2,u_-)-\int_0^{-\infty} \big(\nabla p+z_{-}\cdot\nabla z_{+}\big)(\tau,x_1,x_2,u_--\tau)d\tau,\\
&\displaystyle z_-(-\infty;x_1,x_2,u_+):=z_-(0,x_1,x_2,u_+)-\int_0^{-\infty} \big(\nabla p+z_{+}\cdot\nabla z_{-}\big)(\tau,x_1,x_2,u_++\tau)d\tau
\end{cases}
\end{equation*}
converge. The above formulas define two vector fields $z_+(-\infty;x_1,x_2,u_-)$ and $z_-(-\infty;x_1,x_2,u_+)$ on the past infinities $\mathcal{P}_+$ and $\mathcal{P}_-$ respectively. We call $\big((z_+(-\infty;x_1,x_2,u_-),z_-(-\infty;x_1,x_2,u_+)\big)$ the  {\bf past scattering fields} associated to the solution $\big(z_+(t,x),z_-(t,x)\big)$.
\end{ExistenceTheoremP}	

\bigskip
							
We are ready to state the main results of this paper. There are two versions of the rigidity. The first one matches the situation of \eqref{scattering_model1} as follows:

\begin{RigidityTheorem1}\label{theorem: main theorem}
If the scattering fields $\big(z_+(+\infty;x_1,x_2,u_-),z_-(+\infty;x_1,x_2,u_+)\big)$ vanish on the future infinities, i.e.,
\[\begin{cases}
&z_+(+\infty;x_1,x_2,u_-)\equiv 0 \ \ \text{on} \ \ \mathcal{F}_+,\\
&z_-(+\infty;x_1,x_2,u_+)\equiv 0 \ \ \text{on} \ \ \mathcal{F}_-,
\end{cases}
\]
then the solution itself vanishes identically, i.e., for all $(t,x)\in \mathbb{R}\times \mathbb{R}^3$, $\big(z_+(t,x),z_-(t,x)\big)=(0,0)$. 
\end{RigidityTheorem1}

\bigskip

The second one resembles the situation of \eqref{scattering_model2}. As we mentioned before, this can be viewed as an analogue of the \cite{A-S-S, A-S}.
\begin{RigidityTheorem2}\label{theorem: main theorem2}
If the scattering field $z_-(+\infty;x_1,x_2,u_+)$ vanishes on the future infinity and the scattering field  $z_+(-\infty;x_1,x_2,u_-)$ vanishes on the past infinity, i.e.,
\[\begin{cases}
&z_-(+\infty;x_1,x_2,u_+)\equiv 0 \ \ \text{on} \ \ \mathcal{F}_-,\\
&z_+(-\infty;x_1,x_2,u_-)\equiv 0 \ \ \text{on} \ \ \mathcal{P}_+,
\end{cases}\]
then the solution itself vanishes identically, i.e., for all $(t,x)\in \mathbb{R}\times \mathbb{R}^3$, $\big(z_+(t,x),z_-(t,x)\big)=(0,0)$. 
\end{RigidityTheorem2}

\begin{center}
\includegraphics[width=2in]{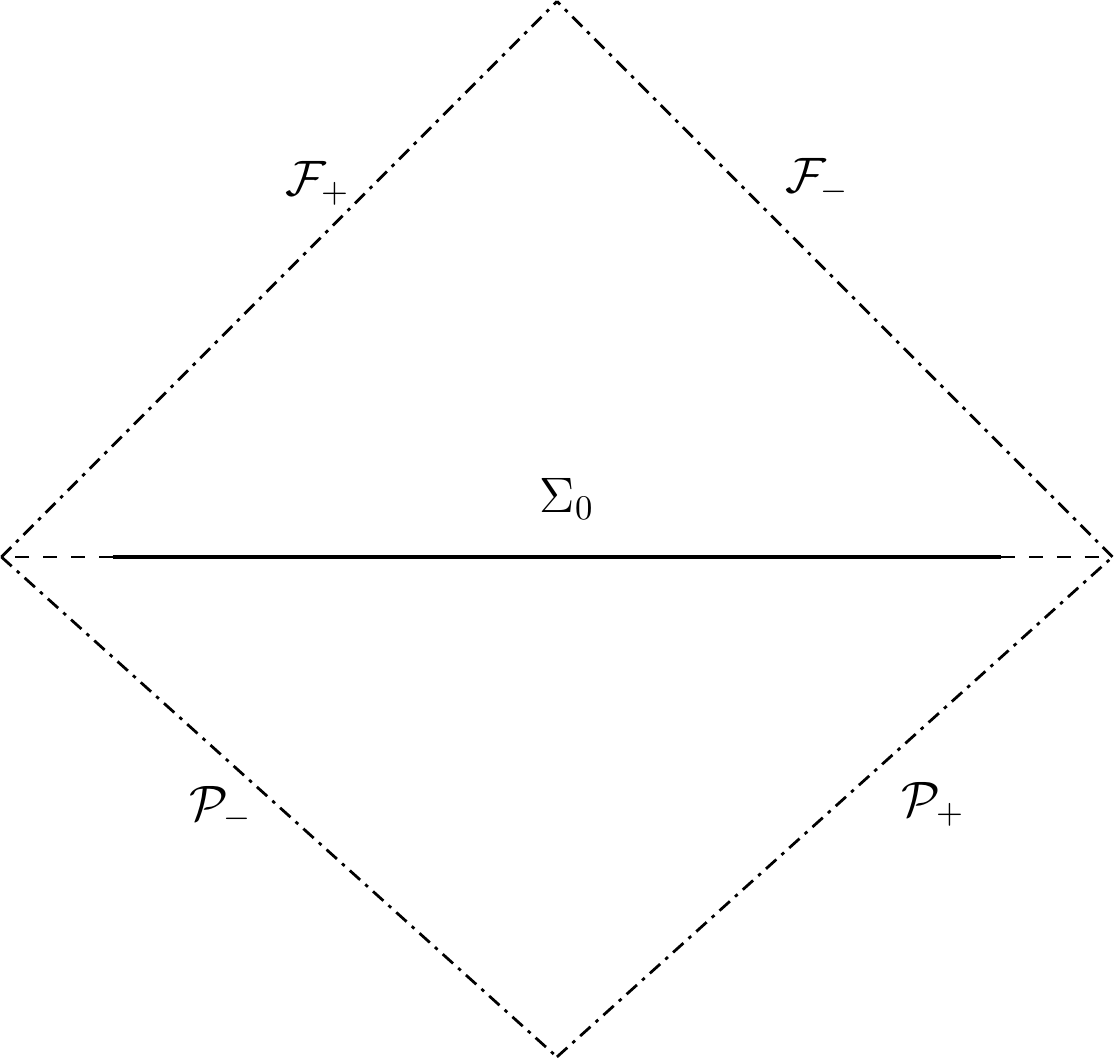}
\end{center}

If we depict the future and past infinities as in the above picture, these two rigidity theorems claim that if scattering fields vanish on two adjacent infinities, then the solution itself vanishes identically.

\bigskip
							
The rest of this paper is organized as follows. In Section \ref{energy estimates}, we establish the main estimates hence the global existence of the solutions. The first part of  Section \ref{main proof} is devoted to the construction of scattering fields as well as functional spaces on the infinities. Finally, we prove the rigidity theorems in the second part of Section \ref{main proof}. The key idea of the proof is to translate the point-wise vanishing properties of the scattering fields at infinities to  $L^2$-smallness conditions for solutions at a (large) finite time.

\section{The energy estimates for Alfv\'en waves}\label{energy estimates}
We assume that the size of the initial energy is given:  $\mathcal{E}^{N_*}(0)=\varepsilon^2$. The parameter $\varepsilon$ is a small positive number and its size will be determined at the end of this section. We fix a positive integer $5\leqslant N_* \leqslant 10$ and we will take $N_*=7$.

We prove the Main Energy Estimates by the standard method of continuity. We assume that there exists a $t^*>0$ such that we have the following energy bound:
\begin{equation}\label{bootstrap assumption}
\sum_{+,-}\bigg(E_{\pm}+F_{\pm}+\sum_{k=0}^{N_{*}}E_{\pm}^k+\sum_{k=0}^{N_{*}}F_{\mp}^k\bigg)\leqslant 2\big(C_1\big)^2\varepsilon^2.
\end{equation}
This is a legitimate assumption: for any $C_1\geqslant 1$,  \eqref{bootstrap assumption} holds for the initial data, hence it remains correct for at least a short time interval $[0,t^*]$. To implement the continuity argument,  we will show that there exist universal constants $\varepsilon_0$ and $C_1\geqslant 1$, under the assumption \eqref{bootstrap assumption}, we can indeed obtain a better bound: 
\begin{equation}\label{bootstrap assumption improved}
\sum_{+,-}\bigg(E_{\pm}+F_{\pm}+\sum_{k=0}^{N_{*}}E_{\pm}^k+\sum_{k=0}^{N_{*}}F_{\mp}^k\bigg)\leqslant \big(C_1\big)^2\varepsilon^2,
\end{equation}
provided for all $\varepsilon<\varepsilon_0$. We emphasize that the constants $\varepsilon_0$ and $C_1$ depend neither on the lifespan $[0,t^*]$ nor on the position parameter $a$ (implicitly written in the energy norms through the weight functions). Therefore, the assumption \eqref{bootstrap assumption} will never be saturated so that we can continue $t^*$ to $+\infty$. The global existence of solutions to \eqref{eq:MHD} also follows. Therefore, it suffices to prove \eqref{bootstrap assumption improved} under \eqref{bootstrap assumption}.

\begin{Remark}
In the following proof, we assume
$\mathcal{E}^{N_*}_{+}(0)=\varepsilon_{+}^2$, $\mathcal{E}^{N_*}_{-}(0)=\varepsilon_{-}^2$, where $\varepsilon_{+}^2+\varepsilon_{-}^2=\varepsilon^2$. Moreover, we can split the assumption \eqref{bootstrap assumption} into
\begin{equation*}\begin{split}
E_{+}+F_{+}+\sum_{k=0}^{N_{*}}E_{+}^k+\sum_{k=0}^{N_{*}}F_{+}^k&\leqslant 2\big(C_1\big)^2\varepsilon_{+}^2,\\
E_{-}+F_{-}+\sum_{k=0}^{N_{*}}E_{-}^k+\sum_{k=0}^{N_{*}}F_{-}^k&\leqslant 2\big(C_1\big)^2\varepsilon_{-}^2.
\end{split}\end{equation*}
In such a way, we can prove the Refined Energy Estimates at the same time.
\end{Remark}

\subsection{Preliminary estimates}
To begin with,  we review a weighted version of  div-curl lemma: 
\begin{Lemma}[div-curl lemma]\label{d-c}
	Let $\lambda(x)\geqslant 1$ be a smooth positive function on $\mathbb{R}^3$ with the additional property that
	 \[|\nabla\lambda(x)|\leqslant C \lambda(x)\]
	 for all $x\in \mathbb{R}^3$, where $C$ is a universal constant. For any \emph{divergence free} smooth vector field $v (x)$ on ${\mathbb{R}^3}$ , we have 
\begin{equation}\label{eq:d-c2}
\big\|\sqrt\lambda\nabla v\big\|_{L^2(\mathbb{R}^3)}^2 \lesssim \big\|\sqrt\lambda\operatorname{curl }v\big\|_{L^2(\mathbb{R}^3)}^2+ \big\|\sqrt\lambda v\big\|_{L^2(\mathbb{R}^3)}^2,
\end{equation}
provided $\sqrt{\lambda}\nabla v\in L^2({\mathbb{R}^3})$ and ${\sqrt{\lambda}}v\in L^2({\mathbb{R}^3})$.
\end{Lemma}
We refer the readers to the Lemma 2.6 in \cite{He-Xu-Yu} for proof. In applications, for any multi-index $\gamma$ with $0\leqslant|\gamma| \leqslant N_*$, we will take $v=\nabla z_\pm^{(\gamma-1)}$. This is a divergence free vector field and we obtain from \eqref{eq:d-c2} that
	\begin{equation}\label{eq:d-c3}
	\big\|\sqrt\lambda\nabla z_\pm^{(\gamma)}\big\|_{L^2(\Sigma_\tau)}^2\lesssim \big\|\sqrt\lambda j_\pm^{(\gamma)}\big\|_{L^2(\Sigma_\tau)}^2+\big\|\sqrt\lambda
	\nabla z_\pm^{(\gamma-1)}\big\|_{L^2(\Sigma_\tau)}^2.
	\end{equation}
By induction on $\gamma$, we have
	\begin{equation}\label{z to j inequality}
	\big\|\sqrt\lambda\nabla z_\pm^{(\gamma)}\big\|_{L^2(\Sigma_\tau)}^2 \lesssim \big\|\sqrt\lambda
	z_\pm\big\|_{L^2(\Sigma_\tau)}^2+\sum_{k=0}^{|\gamma|} \big\|\sqrt\lambda
	j_\pm^{( k)}\big\|_{L^2(\Sigma_\tau)}^2,
	\end{equation}
where $0\leqslant|\gamma| \leqslant N_*$. 
Here and in the sequel, the notation $A\lesssim B$ means that there is a universal constant $C$ such that $A\leqslant CB$.

\smallskip

Later on, the weight function $\lambda$ will be constructed from $\langle u_+ \rangle$ and $\langle u_- \rangle$.  Recall that 
\begin{equation*}
\begin{cases}
	&\langle u_{+}\rangle=(1+|x_3-(t+a)|^2)^{\frac{1}{2}},\\
	&\langle u_{-}\rangle=(1+|x_3+(t+a)|^2)^{\frac{1}{2}}.
\end{cases}
\end{equation*}
We collect elementary properties of $\langle u_\pm \rangle$:
\begin{Lemma}\label{lemma weights}
For all $t \in \mathbb{R}$ and $\omega=1+\delta$, we have the following inequalities:
\begin{enumerate}[(i)]
\item	For $l=1,2$, there holds
	\begin{equation}\label{differentiate weights}
	|\nabla^l  u_{\pm}| \leqslant 1.
	\end{equation}
\item For $l=1,2$, there hold
\begin{equation}\label{differentiate weights coro}
\begin{split}
&|\nabla^l \langle u_{\pm}\rangle| \leqslant 1\leqslant\langle u_{\pm}\rangle,  \\
&|\nabla^l \langle u_\pm\rangle^\omega|\lesssim \langle u_\pm\rangle^\omega, \\
&|\nabla^l \langle u_\pm\rangle^{2\omega}|\lesssim \langle u_\pm\rangle^{2\omega}, \\
&\Big|\nabla^l\Big(\frac{\langle u_{\pm}\rangle^{\omega}}{\langle u_{\mp}\rangle^{\frac{\omega}{2}}}\Big)\Big|\lesssim \frac{\langle u_{\pm}\rangle^{\omega}}{\langle u_{\mp}\rangle^{\frac{\omega}{2}}}.
\end{split}
\end{equation}
\item For the product of $\langle u_{+}\rangle$ and $\langle u_{-}\rangle$, there holds
\begin{equation}\label{u+u-}
\langle u_{+}\rangle\langle u_{-}\rangle\gtrsim 1+|t+a|.
\end{equation} 
\item For $|x-x'|\leqslant 2$, we have 
\begin{equation}\label{eq:xleq2}
	\langle u_{\pm}\rangle(\tau,x)\lesssim\langle u_{\pm}\rangle(\tau,x'),
\end{equation}
and
\begin{equation}\label{eq:xleq2coro}
\begin{split}
	& \langle u_{\pm}\rangle^{\omega}(\tau,x) \lesssim \langle u_{\pm}\rangle^{\omega}(\tau,x'),\\
	&(\langle u_{\mp}\rangle^{\omega}\langle u_{\pm}\rangle^{\frac{\omega}{2}})(\tau,x)\lesssim (\langle u_{\mp}\rangle^{\omega}\langle u_{\pm}\rangle^{\frac{\omega}{2}})(\tau,x').
	\end{split}
\end{equation}
\item For $|x-x'|\geqslant 1$, we have 
\begin{equation}\label{eq:xgeq1}
	\langle u_{\pm}\rangle(\tau,x)\lesssim\langle u_{\pm}\rangle(\tau,x')+|x-x'|,
\end{equation}
and
\begin{equation}\label{eq:xgeq1coro}
\begin{split}
	&\langle u_{\pm}\rangle^{\omega}(\tau,x) \lesssim \langle u_{\pm}\rangle^{\omega}(\tau,x')+|x-x'|^{\omega},\\
	&(\langle u_{\mp}\rangle^{\omega}\langle u_{\pm}\rangle^{\frac{\omega}{2}})(\tau,x)\lesssim (\langle u_{\mp}\rangle^{\omega}\langle u_{\pm}\rangle^{\frac{\omega}{2}})(\tau,x')
	+|x-x'|^{\frac{3\omega}{2}}.
	\end{split}
\end{equation}
The implicit constants in the $\lesssim$'s are all independent of $a$.
\end{enumerate}

\end{Lemma}
\begin{proof} 
(i) and (ii) follow from a direct computation. 

For (iii), according to the definition, we can derive
\begin{equation*}
\langle u_{+}\rangle\langle u_{-}\rangle\gtrsim\left( 1+\big|x_3-(t+a)\big|\right)\left(1+\big|x_3+(t+a)\big|\right).
\end{equation*}
Since for all $x_3 \in \mathbb{R}$, at least one of the following inequalities 
\[|x_3-(t+a)|\geqslant \dfrac{|t+a|}{2}, \ \ |x_3+(t+a)|\geqslant \dfrac{|t+a|}{2}\]
holds. This yields \eqref{u+u-}.

For (iv) and (v), it suffices to check for $u_+$ because the inequalities for $u_-$ can be derived in the same manner. By the mean value theorem, we have
\begin{align*}
	|u_+(\tau,x)-a|&\leqslant|u_+(\tau,x')-a|+|u_+(\tau,x)-u_+(\tau,x')|\\
	&\leqslant|u_+(\tau,x')-a|+\|\nabla u_+\|_{L^\infty(\mathbb{R}^3)}|x-x'| \\
	&\leqslant|u_+(\tau,x')-a|+|x-x'|.
\end{align*}
Thus, for $|x-x'|\leqslant 2$, we have
\begin{align*}
\big(1+ |u_{+}-a|^2\big)^\frac{1}{2}(\tau,x)&\leqslant\big(1 + (|u_{+}-a|+2)^2\big)^\frac{1}{2}(\tau,x')\\
	&\leqslant 4\big(1 + |u_{+}-a|^2\big)^\frac{1}{2}(\tau,x')=4\langle u_{+}\rangle(\tau,x'),
\end{align*}
and therefore \eqref{eq:xleq2} is proved. \eqref{eq:xleq2coro} follows immediately from \eqref{eq:xleq2}.

For $|x-x'|\geqslant 1$, we apply the mean value theorem for $\langle u_{+}\rangle(\tau,x)$ and we obtain
\begin{align*}
	\langle u_{+}\rangle(\tau,x) &\leqslant \langle u_{+}\rangle(\tau,x')+\|\nabla\langle u_+\rangle\|_{L^\infty(\mathbb{R}^3)}|x-x'|\\
	 &\leqslant \langle u_{+}\rangle(\tau,x')+|x-x'|,
\end{align*}
which implies \eqref{eq:xgeq1}. \eqref{eq:xgeq1coro} is an immediate consequence. This completes the proof of the lemma.
\end{proof}

By virtue of the preceding two lemmas, we have the following weighted Sobolev inequalities:
\begin{Lemma}[Sobolev inequalities]\label{SI}
For all $k\leqslant N_*-2$ and multi-indices $\alpha$ with $|\alpha|=k$, we have 
\begin{equation}\label{eq:sobolev0}
\begin{split}
	\big|\langle u_{\pm}\rangle^{\omega}z_\mp\big|^2     &{\lesssim}  \big\|\langle u_{\pm}\rangle^{\omega}  z_\mp\big\|^2_{L^2}+\sum_{ |\beta|\leqslant 1} \big\|\langle u_{\pm}\rangle^{\omega} j_\mp^{(\beta)}\big\|^2_{L^2},\\
	\big|\langle u_{\pm}\rangle^{\omega}\nabla z^{(\alpha)}_\mp\big|^2     &{\lesssim}  \big\|\langle u_{\pm}\rangle^{\omega}  z_\mp\big\|^2_{L^2}+\sum_{ |\beta|\leqslant k+2} \big\|\langle u_{\pm}\rangle^{\omega} j_\mp^{(\beta)}\big\|^2_{L^2}.
\end{split}
\end{equation}
In particular, we have
\begin{equation}\label{eq:sobolev}
	\big|z_{\mp}\big|+\big|\nabla z_{\mp}^{(\alpha)}\big|\lesssim \frac{C_1 \varepsilon_{\mp}}{\langle u_{\pm}\rangle^\omega}, 
\end{equation}
where $C_1$ is the constant from \eqref{bootstrap assumption}.
\end{Lemma}
\begin{proof} 
Since the proof for higher order derivatives is the same as that for $z_\mp$, we only give the details for the first inequality.  According to  the standard Sobolev inequality on $\mathbb{R}^3$, we have
\begin{align*}
	\big|\langle u_{\pm}\rangle^{\omega}z_\mp(t,x)\big|^2 &\lesssim  \big\|\langle u_{\pm}\rangle^{\omega} z_\mp \big\|_{H^2(\mathbb{R}^3)}^2= \sum_{|\beta|\leqslant 2}\big\|\D^\beta \big(\langle u_{\pm}\rangle^{\omega} z_\mp\big)\big\|^2_{L^2}
	\\
	&\leqslant \sum_{|\beta|\leqslant 2}\Big\|
	\sum_{\gamma \leqslant \beta}\nabla^{\gamma} \big(\langle u_{\pm}\rangle^{\omega}\big) z_\mp^{(\beta-\gamma)}\Big\|^2_{L^2}\\
	&\stackrel{\eqref{differentiate weights coro}}{\lesssim}\sum_{|\beta|\leqslant 2}\Big\| \sum_{\gamma \leqslant \beta} \langle u_{\pm}\rangle^{\omega} z_\mp^{(\beta-\gamma)}\Big\|^2_{L^2}
	\lesssim  \sum_{|\beta|\leqslant 2}\big\|\langle u_{\pm}\rangle^{\omega} z^{(\beta)}_\mp\big\|^2_{L^2}\\
    &\stackrel{\eqref{z to j inequality}}{\lesssim}\sum_{ |\beta|\leqslant 2} \Big(\big\|\langle u_{\pm}\rangle^{\omega}  z_\mp\big\|^2_{L^2}+\sum_{l=0}^{ |\beta|-1}\big\|\langle u_{\pm}\rangle^{\omega} j_\mp^{(l)}\big\|^2_{L^2}\Big).
\end{align*}
This gives the first inequality in \eqref{eq:sobolev0}. By virtue of \eqref{bootstrap assumption}, we have
\[	|z_{\mp}|\lesssim \frac{1}{\langle u_{\pm}\rangle^\omega}(E_{\mp}+E^0_{\mp}+E^1_{\mp})^{\frac{1}{2}}\lesssim\frac{C_1\varepsilon_{\mp}}{\langle u_{\pm}\rangle^\omega}.\]
This proves the lemma.
\end{proof}

As a corollary, we can measure the separation of $z_\pm$ in terms of decay in $t$.
\begin{Lemma}[Separation estimates]\label{SE}
	For all $\alpha$ and $\beta$ with $|\alpha|$, $|\beta|\leqslant N_*-1$, we have
	\begin{equation}\label{eq:separation}
	\big|z_+^{(\alpha)}(t,x)z_-^{(\beta)}(t,x)\big|\lesssim \frac{\big(C_1\big)^2\varepsilon_+\varepsilon_-}{\left(1+|t+a|\right)^\omega}.
	\end{equation}
\end{Lemma}
\begin{proof} In view of \eqref{eq:sobolev}, we obtain
	\begin{equation*}
	\big|z_+^{(\alpha)}(t,x) z_-^{(\beta)}(t,x)\big|\lesssim\frac{\big(C_1\big)^2\varepsilon_+\varepsilon_-}{\langle u_{+}\rangle^\omega\langle u_{-}\rangle^\omega}.
	\end{equation*}
Combined with \eqref{u+u-}, we derive \eqref{eq:separation} immediately.
\end{proof}

We now prove bounds for a set of spacetime integrals which will be useful for energy estimates.
\begin{Lemma}\label{flux} For all $0\leqslant k\leqslant N_*$, we have
\begin{equation}\label{eq:flux}
\begin{split}
	&\int_{0}^{t^*}\int_{\Sigma_{\tau}}\frac{\langle u_{\mp}\rangle^{2\omega}}{\langle u_{\pm}\rangle^{\omega}}|z_{\pm}|^2\lesssim \big(C_1\big)^2\varepsilon_{\pm}^2,\\
	&\int_{0}^{t^*}\int_{\Sigma_{\tau}}\frac{\langle u_{\mp}\rangle^{2\omega}}{\langle u_{\pm}\rangle^{\omega}}\big|j^{(k)}_{\pm}\big|^2\lesssim \big(C_1\big)^2 \varepsilon_{\pm}^2,\\
	&\int_{0}^{t^*}\int_{\Sigma_{\tau}}\frac{\langle u_{\mp}\rangle^{2\omega}}{\langle u_{\pm}\rangle^{\omega}}\big|\nabla z^{(k)}_{\pm}\big|^2 \lesssim \big(C_1\big)^2\varepsilon_{\pm}^2.
	\end{split}
\end{equation} 
\end{Lemma}
\begin{proof}
By the symmetry considerations, we only give the details for the estimates on  $z_+$ (and its higher order derivatives). If we parametrize $C_{u_+}^{+,t}$ by $(x_1,x_2,t)$, the surface measure $d\sigma_+$ can be written as 
\[d\sigma_+=\sqrt{2}d{x_1}d{x_2}dt,\]
therefore, we have
	\begin{align*}
	\int_{0}^{t^*}\int_{\Sigma_{\tau}}\frac{\langle u_{-}\rangle^{2\omega}}{\langle u_{+}\rangle^{\omega}}|z_{+}|^2dxd\tau
	&\lesssim\int_{u_{+}}\Big(\int_{C^{+,t^*}_{u_{+}}}\frac{\langle u_{-}\rangle^{2\omega}}{\langle u_{+}\rangle^{\omega}}|z_{+}|^2 d\sigma_{+}\Big)du_{+}\\
	&\lesssim \sup_{u_{+}\in \mathbb{R}}\Big[\int_{C^{+,t^*}_{u_{+}}}\langle u_{-}\rangle^{2\omega}|z_{+}|^2d\sigma_{+}\Big]\int_{\mathbb{R}}\frac{1}{\langle u_{+}\rangle^\omega}du_{+}\\
	&\lesssim F_{+}\lesssim \big(C_1\big)^2\varepsilon_{+}^2.
	\end{align*}
	In the last step, we used the fact that $\displaystyle \int_{\mathbb{R}}\frac{1}{\langle u_{+}\rangle^\omega}du_{+}$ is bounded by a universal constant which is independent of $a$. The bound on $\displaystyle \int_{0}^{t^*}\int_{\Sigma_{\tau}}\frac{\langle u_{-}\rangle^{2\omega}}{\langle u_{+}\rangle^{\omega}}\big|j^{(k)}_{+}\big|^2dxd\tau$ can be derived exactly in the same way. 
	
For the last inequality, by combining div-curl lemma and the previous two estimates,  we have 
	\begin{align*}
	\int_{0}^{t^*}\int_{\Sigma_{\tau}}\frac{\langle u_{-}\rangle^{2\omega}}{\langle u_{+}\rangle^{\omega}}\big|\nabla z^{(k)}_{+}\big|^2dxd\tau
	&\stackrel{\eqref{z to j inequality}}{\lesssim}
	\int_{0}^{t^*}\bigg(\Big\|\frac{\langle u_{-}\rangle^{\omega}}{\langle u_{+}\rangle^{\frac{\omega}{2}}}z_{+}\Big\|_{L^2(\Sigma_\tau)}^2
	+\sum_{l=0}^k\Big\|\frac{\langle u_{-}\rangle^{\omega}}{\langle u_{+}\rangle^{\frac{\omega}{2}}}j^{(l)}_{+}\Big\|_{L^2(\Sigma_\tau)}^2\bigg)d\tau\\
	&\lesssim \big(C_1\big)^2\varepsilon_{+}^2.
	\end{align*}
We remark that, thanks to \eqref{differentiate weights coro}, the weight functions used above satisfy the conditions of the div-curl lemma. The proof of the lemma is now complete.
\end{proof}

\smallskip

We end this subsection by an abstract energy estimate for the following linear system of equations:
\begin{equation}\label{eq:LS}\begin{cases}
		& \D_{t}f_{+}-B_{0}\cdot \nabla f_{+}  =\rho_{+},\\
		& \D_{t}f_{-}+B_{0}\cdot \nabla f_{-}  =\rho_{-}.
	\end{cases}\end{equation}
We remark that $f_\pm$ and $\rho_\pm$ are smooth vector fields defined on $\mathbb{R}\times \mathbb{R}^3$ with sufficiently fast decay in $x$-variables.
\begin{Proposition}
For all weight functions $\lambda_{\pm}$ defined on $[0,t^*]\times\mathbb{R}^3$
with the properties $L_{\pm}\lambda_{\mp}=0$, where $L_{\pm}=\partial_t\pm \partial_{x_3}$, for all $t\in [0,t^*]$, we have 
\begin{equation}\label{eq:LEE}
	\int_{\Sigma_{t}}\lambda_{\pm}|f_{\pm}|^2 dx +\frac{\sqrt{2}}{2}\sup_{u_{\pm}}\int_{C_{u_{\pm}}^{\pm,t}}\lambda_{\pm}|f_{\pm}|^2 d\sigma_{\pm}\leqslant  \int_{\Sigma_{0}}\lambda_{\pm}|f_{\pm}|^2 dx +2\Big|\int_{0}^{t}\int_{\Sigma_{\tau}}\lambda_{\pm}f_{\pm}\cdot\rho_{\pm}dxd\tau\Big|.
\end{equation}
\end{Proposition}
The idea is to multiply the first and second equations by $\lambda_{+} f_{+}$ and $\lambda_{-} f_{-}$ respectively and then integrate by parts on regions bounded by $\Sigma_0$, $\Sigma_t$ and $C_{u_{\pm}}^{\pm,t}$. We refer the readers to Section 2.3 of \cite{He-Xu-Yu} for details of the proof. We remark here that except for the coefficients of the first terms on the both sides of \eqref{eq:LEE}, the exactly numerical constants are irrelevant to the rest of the proof.

\subsection{The bound on pressure and lowest order energy estimates}
We begin to prove \eqref{bootstrap assumption improved}. The proof is divided into three steps. The current subsection is the first step which is devoted to controlling the size of $E_\pm$ and $F_\pm$.

We apply \eqref{eq:LEE} to \eqref{eq:MHD}, i.e., we specialize \eqref{eq:LEE} to: 
\[f_{\pm}=z_{\pm},\  \rho_{\pm}=-\nabla p-z_{\mp}\cdot\nabla z_{\pm},\  \lambda_{\pm}=\langle u_{\mp}\rangle^{2\omega}.\]
Hence,
\begin{equation*}
  \int_{\Sigma_{t}}\langle u_{\mp}\rangle^{2\omega}|z_{\pm}|^2 +\sup_{u_{\pm}\in \mathbb{R}}\int_{C_{u_{\pm}}^{\pm,t}}\langle u_{\mp}\rangle^{2\omega}|z_{\pm}|^2 
  \lesssim \int_{\Sigma_{0}}\langle u_{\mp}\rangle^{2\omega}|z_{\pm}|^2 +\int_{0}^{t}\int_{\Sigma_{\tau}}\langle u_{\mp}\rangle^{2\omega}|z_{\pm}| |\nabla p+z_{\mp}\cdot\nabla z_{\pm}|,
\end{equation*}
which implies that
\begin{equation}\label{eq:LOW}
  E_{\pm}(t)+F_{\pm}(t)\lesssim E_{\pm}(0)+\underbrace{\int_{0}^{t}\int_{\Sigma_{\tau}}\langle u_{\mp}\rangle^{2\omega}|z_{\pm}| |\nabla p|}_{\mathbf{I}_{\bm{\pm}}}+\underbrace{\int_{0}^{t}\int_{\Sigma_{\tau}}\langle u_{\mp}\rangle^{2\omega}|z_{\pm}| |z_{\mp}| |\nabla z_{\pm}|}_{\mathbf{J}_{\bm{\pm}}}.
\end{equation}
By symmetry considerations, to bound $\mathbf{J}_{\bm{\pm}}$, it suffices to control $\mathbf{J}_{\bm{+}}$. Thanks to H\"older inequality, we have
\begin{equation}\label{the term J'}\begin{split}
  \mathbf{J}_{\bm{+}}&\leqslant \Big(\int_{0}^{t}\int_{\Sigma_{\tau}}\!\!\!\frac{\langle u_{-}\rangle^{2\omega}}{\langle u_{+}\rangle^{\omega}}|z_{+}|^2\!\Big)^{\frac{1}{2}}\Big(\int_{0}^{t}\int_{\Sigma_{\tau}}\!\!\langle u_{-}\rangle^{2\omega}\langle u_{+}\rangle^\omega |z_{-}|^2 |\nabla z_{+}|^2\Big)^{\frac{1}{2}}\\
  & \stackrel{\eqref{eq:flux}}{\lesssim}C_1\varepsilon_+\Big(\underbrace{\int_{0}^{t}\int_{\Sigma_{\tau}}\!\!\langle u_{-}\rangle^{2\omega}\langle u_{+}\rangle^\omega |z_{-}|^2 |\nabla z_{+}|^2}_{\mathbf{J'}}\Big)^{\frac{1}{2}}.
\end{split}\end{equation}
We observe that $\mathbf{J'}$ can be bounded as follows:
\begin{equation*}\begin{split}
  \mathbf{J'}&\leqslant \int_{0}^{t} \big\|\langle u_{+}\rangle^{\omega}z_{-}\big\|_{L^\infty(\Sigma_\tau)}^2\Big\|\frac{\langle u_{-}\rangle^{\omega}}{\langle u_{+}\rangle^{\frac{\omega}{2}}}\nabla z_{+}\Big\|_{L^2(\Sigma_\tau)}^2\\
  &\stackrel{\eqref{eq:sobolev}}{\lesssim}\big(C_1\big)^2\varepsilon_-^2\int_{0}^{t}\Big\|\frac{\langle u_{-}\rangle^{\omega}}{\langle u_{+}\rangle^{\frac{\omega}{2}}}\nabla z_{+}\Big\|_{L^2(\Sigma_\tau)}^2.
\end{split}\end{equation*}
In view of \eqref{eq:flux}, we can bound the last term above using flux norm and hence
\begin{equation}\label{final estimate on J'}
	\mathbf{J'}\lesssim\big(C_1\big)^4\varepsilon_+^2\varepsilon_-^2.
\end{equation}
We conclude that
\begin{equation}\label{eq:J+}
\mathbf{J}_{\bm{+}}\lesssim \big(C_1\big)^3\varepsilon_+^2\varepsilon_-.
\end{equation}

It remains to bound the pressure terms $\mathbf{I}_{\bm{\pm}}$. Once again, it suffices to bound  $\mathbf{I}_{\bm{+}}$.  First of all, we have
\begin{equation}\label{the term I'}
\begin{split}
   \mathbf{I}_{\bm{+}} &\leqslant \Big(\int_{0}^{t}\int_{\Sigma_{\tau}}\frac{\langle u_{-}\rangle^{2\omega}}{\langle u_{+}\rangle^{\omega}}|z_{+}|^2\Big)^{\frac{1}{2}}
   \Big(\int_{0}^{t}\int_{\Sigma_{\tau}}\langle u_{-}\rangle^{2\omega}\langle u_{+}\rangle^\omega |\nabla p|^2 \Big)^{\frac{1}{2}}\\
   &\stackrel{\eqref{eq:flux}}{\lesssim}C_1 \varepsilon_+\Big(\underbrace{\int_{0}^{t}\int_{\Sigma_{\tau}}\langle u_{-}\rangle^{2\omega}\langle u_{+}\rangle^\omega |\nabla p|^2 }_{\mathbf{I'}}\Big)^{\frac{1}{2}}.
   \end{split}
\end{equation}
We need to study $\nabla p$ in order to bound $\mathbf{I'}$. Since div$ z_{\pm}=0$,  the divergence of the first equation in \eqref{eq:MHD} yields 
\begin{equation*}
  -\Delta p=\D_{i}z^{j}_{-}\D_{j}z^{i}_{+}.
\end{equation*}
Therefore, by use of the Newtonian potential, on each time slice $\Sigma_{\tau}$ we have the following decomposition:
\begin{align*}
  \nabla p(\tau,x)&=
   \frac{1}{4\pi}\nabla\int_{\mathbb{R}^3}\frac{1}{|x-x'|}\cdot\big(\D_{i}z^{j}_{-}\D_{j}z^{i}_{+}\big)(\tau,x')dx'\\
   &=\frac{1}{4\pi}\int_{\mathbb{R}^3}\nabla\frac{1}{|x-x'|}\cdot \theta(|x-x'|)\cdot \big(\D_{i}z^{j}_{-}\D_{j}z^{i}_{+}\big)(\tau,x')dx'\\
   &\ \ +\frac{1}{4\pi}\int_{\mathbb{R}^3}\nabla\frac{1}{|x-x'|}\cdot \big(1-\theta(|x-x'|)\big)\cdot\big(\D_{i}z^{j}_{-}\D_{j}z^{i}_{+}\big)(\tau,x')dx',
\end{align*}
where the smooth cut-off function $\theta(r)$ is chosen so that $\theta(r)=1$ for $r\leqslant 1$ and $\theta(r)=0$ for $r\geqslant 2$. In view of the fact that $\operatorname{div} z_\pm=0$,  integration by parts gives
\begin{equation}\label{eq:nabla p}
\begin{split}
   \nabla p(\tau,x)&=
   \frac{1}{4\pi}\int_{\mathbb{R}^3}\nabla\frac{1}{|x-x'|}\cdot\theta(|x-x'|)\cdot\big(\D_{i}z^{j}_{-}\D_{j}z^{i}_{+}\big)(\tau,x')dx'\\
   &\ \ +\frac{1}{4\pi}\int_{\mathbb{R}^3}\D_{j}\D_{i}\Big(\nabla\frac{1}{|x-x'|} \cdot\big(1-\theta(|x-x'|)\big)\Big)\cdot\big(z^{j}_{-}z^{i}_{+}\big)(\tau,x')dx'.
   \end{split}
\end{equation}
In view of the property of the cut-off function $\theta(r)$, we can bound $\nabla p$ by
\begin{align*}
\!\!|\nabla p(\tau,x)|
&\lesssim \underbrace{\int_{|x-x'|\leqslant 2}\!\!\!\!\!\!\!\!\!\!\!\!\!\!\frac{|(\nabla z_{-}\cdot\nabla z_{+})(\tau,x')|}{|x-x'|^2}dx'}_{\mathbf{A_1}}+\underbrace{\int_{|x-x'|\geqslant 1}\!\!\!\!\!\!\!\!\!\!\frac{|(z_{-}\cdot z_{+})(\tau,x')|}{|x-x'|^4}dx'}_{\mathbf{A_2}} +\underbrace{\int_{1\leqslant |x-x'|\leqslant 2}\!\! {|(z_{-}\cdot z_{+})(\tau,x')|}dx'}_{\mathbf{A_3}}\!\!.\stepcounter{equation}\tag{\theequation}\label{decomposition of nabla p}
\end{align*}
It follows that 
\begin{equation}\label{eq:I'}
\mathbf{I'} \lesssim 
\underbrace{\int_{0}^{t}\int_{\Sigma_{\tau}}\langle u_{-}\rangle^{2\omega}\langle u_{+}\rangle^\omega |\mathbf{A_1}|^2}_{\mathbf{I_1}}
+\underbrace{\int_{0}^{t}\int_{\Sigma_{\tau}}\langle u_{-}\rangle^{2\omega}\langle u_{+}\rangle^\omega |\mathbf{A_2}|^2}_{\mathbf{I_2}}
+\underbrace{\int_{0}^{t}\int_{\Sigma_{\tau}}\langle u_{-}\rangle^{2\omega}\langle u_{+}\rangle^\omega |\mathbf{A_3}|^2}_{\mathbf{I_3}}.
\end{equation}
For $\mathbf{I_1}$, according to  the definition of $\mathbf{A_1}$, we have
\begin{align*}
\langle u_-\rangle^\omega\langle u_+\rangle^{\frac{\omega}{2}}|\mathbf{A_1}|
&=
\int_{|x-x'|\leqslant 2}\frac{\big(\langle u_-\rangle^\omega\langle u_+\rangle^{\frac{\omega}{2}}\big)(\tau,x)|(\nabla z_{-}\cdot\nabla z_{+})(\tau,x')|}{|x-x'|^2}dx'\\
&\stackrel{\eqref{eq:xleq2coro}}{\lesssim}
 \int_{|x-x'|\leqslant 2}\frac{\big(\langle u_-\rangle^\omega\langle u_+\rangle^{\frac{\omega}{2}}\big)(\tau,x')|(\nabla z_{-}\cdot\nabla z_{+})(\tau,x')|}{|x-x'|^2}dx'\\
&\leqslant \big\|\langle u_{+}\rangle^{\omega}\nabla z_-\big\|_{L^\infty}\int_{|x-x'|\leqslant2}\frac{\langle u_{-}\rangle^{{\omega}}(\tau,x')|\nabla z_+(\tau,x')|}{\langle u_{+}\rangle^{\frac{\omega}{2}}(\tau,x')|x-x'|^2}dx'\\
&\stackrel{\eqref{eq:sobolev}}{\lesssim}
C_1 \varepsilon_-\int_{|x-x'|\leqslant2}\frac{1}{|x-x'|^2}\frac{\langle u_{-}\rangle^{\omega}(\tau,x')}{\langle u_{+}\rangle^{\frac{\omega}{2}}(\tau,x')} |\nabla z_+(\tau,x')|dx'.
\end{align*}
We note that $\dfrac{1}{|x|^2}\text{\Large$\text{\Large$\chi$}$}_{|x|\leqslant 2}\in L^1(\mathbb{R}^3)$ where
$\text{\Large$\text{\Large$\chi$}$}_{|x|\leqslant 2}$ is the characteristic function of the ball of radius $2$ centered at the origin. Thanks to Young's inequality, we have
\begin{align*}
	\big\|\langle u_-\rangle^\omega\langle u_+\rangle^{\frac{\omega}{2}}\mathbf{A_1}\big\|^2_{L^2(\Sigma_\tau)}
	&\lesssim\big(C_1\big)^2\varepsilon_-^2\Big\|\frac{1}{|x|^2}\text{\Large$\text{\Large$\chi$}$}_{|x|\leqslant 2}\Big\|^2_{L^1(\Sigma_\tau)} \Big\|\frac{\langle u_{-}\rangle^{\omega}}{\langle u_{+}\rangle^{\frac{\omega}{2}}}\nabla z_{+}\Big\|^2_{L^2(\Sigma_\tau)}\\
	&\lesssim\big(C_1\big)^2\varepsilon_-^2\Big\|\frac{\langle u_{-}\rangle^{\omega}}{\langle u_{+}\rangle^{\frac{\omega}{2}}}
	\nabla z_{+}\Big\|^2_{L^2(\Sigma_\tau)}.
\end{align*}
Thus,  we can apply \eqref{eq:flux} to derive
\begin{equation}\label{eq:I1}
\mathbf{I_{1}} \lesssim \big(C_1\big)^4\varepsilon_+^2\varepsilon_-^2.
\end{equation}
We remark that $\mathbf{I_3}$ can be treated exactly in the same manner. So we have
\begin{equation}\label{eq:I3}
\mathbf{I_3}\lesssim \big(C_1\big)^4\varepsilon_+^2\varepsilon_-^2.
\end{equation}
For $\mathbf{I_2}$, we have
\begin{align*}
\mathbf{I_2}
&= 
\int_{0}^{t}\Big\|\langle u_{-}\rangle^{\omega}(\tau,x)\langle u_{+}\rangle^{\frac{\omega}{2}}(\tau,x)\int_{|x-x'|\geqslant 1}\frac{|(z_{-}\cdot z_{+})(\tau,x')|}{|x-x'|^4}dx'\Big\|_{L^2(\Sigma_\tau)}^2\\
&\stackrel{\eqref{eq:xgeq1coro}}{\lesssim} \int_{0}^{t}\Big\|\int_{|x-x'|\geqslant 1}\Big(\big(\langle u_{-}\rangle^{\omega}\langle u_{+}\rangle^{\frac{\omega}{2}}\big)(\tau,x')+|x-x'|^{\frac{3\omega}{2}}\Big)\frac{|(z_{-}\cdot z_{+})(\tau,x')|}{|x-x'|^4}dx'\Big\|_{L^2(\Sigma_\tau)}^2\\
&\lesssim
\underbrace{\int_{0}^{t}\Big\|\int_{|x-x'|\geqslant 1}\!\!\!\!\!\!\big(\langle u_{-}\rangle^{\omega}\langle u_{+}\rangle^{\frac{\omega}{2}}\big)(\tau,x')\frac{|(z_{-}\cdot z_{+})(\tau,x')|}{|x-x'|^4}dx'\Big\|_{L^2(\Sigma_\tau)}^2}_{\mathbf{I_{21}}}+\underbrace{\int_{0}^{t}\Big\|\int_{|x-x'|\geqslant 1}\!\!\!\!\!\frac{|(z_{-}\cdot z_{+})(\tau,x')|}{|x-x'|^{4-\frac{3\omega}{2}}}dx'\Big\|_{L^2(\Sigma_\tau)}^2}_{\mathbf{I_{22}}}.
\end{align*}
For $\mathbf{I_{21}}$, since $\dfrac{1}{|x|^4}\text{\Large$\chi$}_{|x|\geqslant 1}\in L^1(\mathbb{R}^3)$, we have
\begin{align*}
\mathbf{I_{21}}&
\stackrel{\text{Young's}}{\lesssim}
\int_{0}^{t}\Big\|\frac{1}{|x|^4}\text{\Large$\chi$}_{|x|\geqslant 1}\Big\|_{L^1(\Sigma_\tau)}^2 \big\|\langle u_{+}\rangle^{\omega}z_{-}\big\|_{L^\infty(\Sigma_\tau)}^2\Big\|\frac{\langle u_{-}\rangle^{\omega}}{\langle u_{+}\rangle^{\frac{\omega}{2}}}
z_{+}\Big\|_{L^2(\Sigma_\tau)}^2 \\
&\stackrel{\eqref{eq:sobolev}}{\lesssim}(C_1
 \varepsilon_-)^2\int_{0}^{t}\Big\|\frac{\langle u_{-}\rangle^{\omega}}{\langle u_{+}\rangle^{\frac{\omega}{2}}}z_{+}\Big\|_{L^2(\Sigma_\tau)}^2\\
 &\stackrel{\eqref{eq:flux}}{\lesssim} \big(C_1\big)^4\varepsilon_+^2\varepsilon_-^2.
\end{align*}
For $\mathbf{I_{22}}$, since $\dfrac{1}{|x|^{4-\frac{3\omega}{2}}}\text{\Large$\chi$}_{|x|\geqslant 1}\in L^2(\mathbb{R}^3)$ when $\omega\in(1,\frac{5}{3})$ (this is the place where we have constraints on $\delta$), we have
\begin{align*}
\mathbf{I_{22}}&
\stackrel{\text{Young's}}{\lesssim}
\int_{0}^{t}\Big\|\frac{1}{|x|^{4-\frac{3\omega}{2}}}\text{\Large$\chi$}_{|x|\geqslant 1}\Big\|_{L^2(\Sigma_\tau)}^2\big\|z_{-}z_{+}\big\|_{L^1(\Sigma_\tau)}^2\\
&\lesssim \int_{0}^{t}\big\|z_{-} z_{+}\big\|_{L^1(\Sigma_\tau)}^2.
\end{align*}
Since \eqref{u+u-}, we have $\langle u_{+}\rangle\langle u_{-}\rangle\gtrsim 1+|\tau+a|$ and then
\begin{align*}
\mathbf{I_{22}}&\lesssim  \int_{0}^{t}\Big\|\frac{1}{\langle u_{+}\rangle^{\omega}\langle u_{-}\rangle^{\omega}}\langle u_{+}\rangle^{\omega}z_{-}\langle u_{-}\rangle^{\omega} z_{+}\Big\|_{L^1(\Sigma_\tau)}^2\\
&\lesssim \int_{0}^{t}\frac{1}{\big(1+|\tau+a|\big)^{2\omega}}\big\|\langle u_{+}\rangle^{\omega}z_{-}\big\|_{L^2(\Sigma_\tau)}^2\big\|\langle u_{-}\rangle^{\omega}z_{+}\big\|_{L^2(\Sigma_\tau)}^2\\
&\lesssim \big(C_1\big)^2\varepsilon_{-}^2\big(C_1\big)^2\varepsilon_{+}^2\int_{0}^{t}\frac{1}{\big(1+|\tau+a|\big)^{2\omega}}d\tau\lesssim \big(C_1\big)^4\varepsilon_+^2\varepsilon_-^2.
\end{align*}
Consequently, we obtain
\begin{equation}\label{eq:I2}
\mathbf{I_2}\lesssim \big(C_1\big)^4\varepsilon_+^2\varepsilon_-^2.
\end{equation}
By putting \eqref{eq:I1}, \eqref{eq:I3} and \eqref{eq:I2} together, in view of \eqref{eq:I'}, we get
\begin{equation}\label{final estimate on I'}
\mathbf{I'}\lesssim \big(C_1\big)^4\varepsilon_+^2\varepsilon_-^2.
\end{equation}
This leads to
\begin{equation*}
\mathbf{I}_{\bm{+}}\lesssim \big(C_1\big)^3\varepsilon_+^2\varepsilon_-.
\end{equation*}
Combined with \eqref{eq:LOW} and \eqref{eq:J+}, we take supremum over all $t\in [0,t^*]$ and then obtain
\begin{equation}\label{eq:LOWEE}
E_{\pm}+F_{\pm}\leqslant C_0 E_{\pm}(0)+C_0\big(C_1\big)^3\varepsilon_\pm^2\varepsilon_\mp,
\end{equation}
where $C_0$ is a universal constant and the constant $C_1$ is from the bootstrap assumption \eqref{bootstrap assumption}. We remark that the universal constant $C_0$ will appear many times in the sequel and they may be different. 

\subsection{The higher order energy bounds}

To derive higher order energy estimates, we commute derivatives with the vorticity equations. For a given multi-index $\beta$ with $0\leqslant|\beta|\leqslant N_*$, we  apply $\D^{\beta}$ to \eqref{eq:curlMHD-2} and we obtain
\begin{equation}\begin{cases}\label{eq:DcurlMHD-2}
\D_t j^{(\beta)}_{+}-B_{0}\cdot\nabla j^{(\beta)}_{+}&=\rho^{(\beta)}_{+},\\
\D_t j^{(\beta)}_{-}+B_{0}\cdot\nabla j^{(\beta)}_{-}&=\rho^{(\beta)}_{-},
\end{cases}\end{equation}
where source terms $\rho_\pm^{(\beta)}$ are given by
\begin{equation}\label{definition of rhobeta}
\rho_\pm^{(\beta)}=-\D^\beta(\nabla z_\mp\wedge\nabla z_\pm)-[\D^\beta,z_\mp\cdot\nabla] j_\pm-z_\mp\cdot\nabla j_\pm^{(\beta)}.
\end{equation}
We apply \eqref{eq:LEE} to \eqref{eq:DcurlMHD-2} with the weight functions $\lambda_{\pm}=\langle u_{\mp}\rangle^{2\omega}$. This yields
\begin{equation*}
	\int_{\Sigma_{t}}\langle u_{\mp}\rangle^{2\omega}\big|j^{(\beta)}_{\pm}\big|^2  +\sup_{u_{\pm}\in \mathbb{R}}\int_{C_{u_{\pm}}^{\pm,t}}\langle u_{\mp}\rangle^{2\omega}\big|j^{(\beta)}_{\pm}\big|^2d\sigma_{\pm} 
	\lesssim \int_{\Sigma_{0}}\langle u_{\mp}\rangle^{2\omega}\big|j^{(\beta)}_{\pm}\big|^2  +\underbrace{\Big|\int_{0}^{t}\int_{\Sigma_{\tau}}\langle u_{\mp}\rangle^{2\omega}j^{(\beta)}_{\pm}\cdot\rho^{(\beta)}_{\pm} \Big|}_{\mathbf{K}_{\bm{\pm}}} ,
\end{equation*}
i.e.,
\begin{equation}\label{eq:HIGH}
	E^{(\beta)}_{\pm}(t)+F^{(\beta)}_{\pm}(t)\lesssim E^{(\beta)}_{\pm}(0) +\mathbf{K}_{\bm{\pm}}.
\end{equation}
Based on the symmetry considerations, it suffices to control $\mathbf{K}_{\bm{+}}$. We note that the source term $\rho_+^{(\beta)}$ in $\eqref{eq:DcurlMHD-2}$ can be bounded by
\begin{align*}
\big|\rho_+^{(\beta)}\big|&\leqslant\sum_{\gamma\leqslant\beta} C_\beta^\gamma\big|\nabla z_-^{(\gamma)}\big|\big|\nabla z_+^{(\beta-\gamma)}\big|
+\sum_{0\neq\gamma\leqslant\beta} C_\beta^\gamma\big|z_-^{(\gamma)}\big|\big|\nabla j_+^{(\beta-\gamma)}\big| \ +\ \big|z_-\cdot\nabla j_+^{(\beta)}\big|\\
&\lesssim\sum_{k\leqslant|\beta|} \big|\nabla z_-^{(k)}\big|\big|\nabla z_+^{(|\beta|-k)}\big|\ +\ \big|z_-\cdot\nabla j_+^{(\beta)}\big|.
\stepcounter{equation}\tag{\theequation}\label{rho beta explicit}
\end{align*}
As a consequence, we can bound $\mathbf{K}_{\bm{+}}$ by
\begin{equation*}
	\mathbf{K}_{\bm{+}}\lesssim \underbrace{\sum_{k\leqslant |\beta|}\int_{0}^{t}\int_{\Sigma_{\tau}}\langle u_{-}\rangle^{2\omega}\big|j^{(\beta)}_{+}\big|\big|\nabla z^{(k)}_{-} \big|\big|\nabla z^{(|\beta|-k)}_{+}\big|}_{\mathbf{K_1}}
	+\underbrace{\Big|\int_0^t\int_{\Sigma_\tau}\langle u_{-}\rangle^{2\omega}j^{(\beta)}_{+}\cdot \big(z_{-}\cdot \nabla j_{+}^{(\beta)}\big)\Big|}_{\mathbf{K_2}}.
\end{equation*}

\medskip

We first bound $\mathbf{K_1}$. According to the number of derivatives, we distinguish two cases:

\bigskip

\Emph{Case 1:} $0\leqslant|\beta|\leqslant N_{*}-2$. We can use Sobolev inequality on $\nabla z_{-}^{(k)}$ because $k+2\leqslant N_{*}$. Hence,
\begin{align*}
	\mathbf{K_1} 
	&\lesssim\sum_{k\leqslant|\beta|}\int_{0}^{t}
	\bigg(\big\| \langle u_{+}\rangle^{\omega}\nabla z^{(k)}_{-} \big\|_{L^\infty(\Sigma_{\tau})}
	\int_{\Sigma_{\tau}}\frac{\langle u_{-}\rangle^{\omega}
	\big|j^{(\beta)}_{+}\big|\cdot	\langle u_{-}\rangle^{\omega}\big|\nabla z^{(|\beta|-k)}_{+}\big|}{\langle u_{+}\rangle^{\omega}} \bigg)d\tau\\
	&\stackrel{\eqref{eq:sobolev}}{\lesssim}C_1\varepsilon_-\sum_{k\leqslant|\beta|}\int_{0}^{t}\int_{\Sigma_{\tau}}
	\frac{\langle u_{-}\rangle^{\omega}\big|j^{(\beta)}_{+}\big|\cdot
	\langle u_{-}\rangle^{\omega}\big|\nabla z^{(|\beta|-k)}_{+}\big|}{\langle u_{+}\rangle^{\omega}}\\
   &\lesssim C_1 \varepsilon_-
	\sum_{k\leqslant|\beta|}\int_{0}^{t}\bigg(\Big\| \frac{\langle u_{-}\rangle^{\omega}}{\langle u_{+}\rangle^{\frac{\omega}{2}}}j^{(\beta)}_{+}\Big\|^2_{L^2(\Sigma_\tau)}
	+\Big\|\frac{\langle u_{-}\rangle^{\omega}}
	{\langle u_{+}\rangle^{\frac{\omega}{2}}}\nabla z^{(|\beta|-k)}_{+}\Big\|^2_{L^2(\Sigma_\tau)}\bigg)d\tau\\
	&\stackrel{\eqref{eq:flux}}{\lesssim}\big(C_1\big)^3\varepsilon_+^2\varepsilon_-.
\end{align*}

\medskip

\Emph{Case 2:} $|\beta|= N_{*}-1$ or $N_{*}$. We rewrite $\mathbf{K_1}$ as
\begin{equation*}
	\mathbf{K_1}\lesssim\bigg(\underbrace{\sum_{k\leqslant N_{*}-2}}_{\mathbf{K_{11}}}+\underbrace{\sum_{N_{*}-1\leqslant k\leqslant|\beta|}}_{\mathbf{K_{12}}}\bigg)\int_{0}^{t}\int_{\Sigma_{\tau}}\langle u_{-}\rangle^{2\omega}\big|j^{(\beta)}_{+}\big|\big|\nabla z^{(k)}_{-} \big|\big|\nabla z^{(|\beta|-k)}_{+} \big|.
\end{equation*}
We notice that $\mathbf{K_{11}}$ can be controlled in the same manner as in \Emph{Case 1} so that
\begin{equation*}
	\mathbf{K_{11}}\lesssim \big(C_1\big)^3\varepsilon_+^2\varepsilon_-.
\end{equation*}
For $\mathbf{K_{12}}$, the $L^\infty$ bounds on $\nabla z_-^{(k)}$ are not available since one can not afford more than $N_*$ derivatives (via Sobolev inequality). Instead, we will use $L^\infty$ bounds of $\nabla z_+^{(|\beta|-k)}$ in a different way:
\begin{align*}
	\mathbf{K_{12}}
	&=\sum_{N_{*}-1\leqslant k\leqslant|\beta|}\int_{0}^{t}\int_{\Sigma_{\tau}}	\langle u_{+}\rangle^{\omega}
	\big|\nabla z^{(k)}_{-}\big|\cdot\frac{\langle u_{-}\rangle^{\omega}}{\langle u_{+}\rangle^{\frac{\omega}{2}}}\big|j^{(\beta)}_{+}\big|\cdot
	\frac{\langle u_{-}\rangle^{\omega}}{\langle u_{+}\rangle^{\frac{\omega}{2}}}
	\big|\nabla z^{(|\beta|-k)}_{+} \big|\\
	&\!\!\!\!\!\stackrel{\text{H\"older}}{\lesssim}\sum_{N_{*}-1\leqslant k\leqslant|\beta|}{\Big\|\langle u_{+}\rangle^{\omega}\nabla z^{(k)}_{-}	\Big\|_{L^\infty_{\tau}L^2_x} }
	\Big\|\frac{\langle u_{-}\rangle^{\omega}}{\langle u_{+}\rangle^{\frac{\omega}{2}}}j^{(\beta)}_{+}\Big\|_{L^2_{\tau}L^2_x} \Big\|\frac{\langle u_{-}\rangle^{\omega}}{\langle u_{+}\rangle^{\frac{\omega}{2}}}
	\nabla z^{(|\beta|-k)}_{+} \Big\|_{L^2_{\tau}L^\infty_x}.
\end{align*}
By \eqref{z to j inequality} and \eqref{bootstrap assumption}, we have
\[ \Big\|\langle u_{+}\rangle^{\omega}\nabla z^{(k)}_{-}\Big\|_{L^\infty_{\tau}L^2_x} \lesssim C_1\varepsilon_-.\]
By  \eqref{eq:flux}, we have
\[ \Big\|\frac{\langle u_{-}\rangle^{\omega}}{\langle u_{+}\rangle^{\frac{\omega}{2}}}j^{(\beta)}_{+}\Big\|_{L^2_{\tau}L^2_x} \lesssim C_1\varepsilon_+.\]
Therefore, we obtain
\begin{align*}
	\mathbf{K_{12}}
	&\lesssim \big(C_1\big)^2\varepsilon_+\varepsilon_-
	\!\!\sum_{k=0}^{|\beta|-(N_{*}-1)}\!\Big\|\frac{\langle u_{-}\rangle^{\omega}}{\langle u_{+}\rangle^{\frac{\omega}{2}}}
	\nabla z^{(k)}_{+} \Big\|_{L^2_{\tau}L^\infty_x}.
\end{align*}
Similar to the proof of Lemma \ref{SI}, the standard Sobolev inequalities give rise to
\begin{align*}
   \Big\|\frac{\langle u_{-}\rangle^{\omega}}{\langle u_{+}\rangle^{\frac{\omega}{2}}}
	\nabla z^{(k)}_{+} \Big\|_{L^2_\tau L^\infty_x}
	& \lesssim
	\Big\|\frac{\langle u_{-}\rangle^{\omega}}{\langle u_{+}\rangle^{\frac{\omega}{2}}}
	\nabla z^{(k)}_{+} \Big\|_{L^2_\tau L^2_x}\!\!
	+\Big\|\nabla^2\Big(\frac{\langle u_{-}\rangle^{\omega}}{\langle u_{+}\rangle^{\frac{\omega}{2}}}\nabla z^{(k)}_{+} \Big)
	\Big\|_{L^2_\tau L^2_x}\\
	&\!\!\stackrel{\eqref{differentiate weights coro}}{\lesssim}
	\sum_{l=k}^{k+2}\Big\|\frac{\langle u_{-}\rangle^{\omega}}{\langle u_{+}\rangle^{\frac{\omega}{2}}}
	\nabla z^{(l)}_{+}\Big\|_{L^2_\tau L^2_x}.
\end{align*}
Therefore, according to \eqref{eq:flux}, we have
\begin{equation}\label{eq:instead sobolev}
\Big\|\frac{\langle u_{-}\rangle^{\omega}}{\langle u_{+}\rangle^{\frac{\omega}{2}}}
	\nabla z^{(k)}_{+} \Big\|_{L^2_\tau L^\infty_x}\lesssim C_1\varepsilon_+.
\end{equation}
It follows that
\begin{equation*}
\mathbf{K_{12}}\lesssim \big(C_1\big)^3\varepsilon_+^2\varepsilon_-.
\end{equation*}
Thus, we conclude that for $0\leqslant|\beta|\leqslant N_*$,
\begin{equation*}
	\mathbf{K_{1}}\lesssim  \big(C_1\big)^3\varepsilon_+^2\varepsilon_-.
\end{equation*}
\bigskip

We turn to bound $\mathbf{K_{2}}$ now.  According to the number of derivatives, we distinguish two cases:

\bigskip

\Emph{Case 1}: $0\leqslant|\beta|\leqslant N_*-1$. In this case, we can use  Lemma \ref{flux} because $|\beta|+1\leqslant N_{*}$. Hence,
\begin{align*}
	\mathbf{K_{2}} &
	\lesssim\int_{0}^{t}\big\| \langle u_{+}\rangle^{\omega}z_{-} \big\|_{L^\infty(\Sigma_{\tau})}
	\int_{\Sigma_{\tau}}\frac{\langle u_{-}\rangle^{\omega}
	\big|j^{(\beta)}_{+}\big|\cdot\langle u_{-}\rangle^{\omega}\big|\nabla j^{(\beta)}_{+}\big|}{\langle u_{+}\rangle^{\omega}} \\
	&\stackrel{\eqref{eq:sobolev}}{\lesssim}
	C_1\varepsilon_-\int_0^t\int_{\Sigma_{\tau}}\frac{\langle u_{-}\rangle^{\omega}\big|j^{(\beta)}_{+}\big|\cdot
	\langle u_{-}\rangle^{\omega}\big|\nabla j^{(\beta)}_{+}\big|}{\langle u_{+}\rangle^{\omega}} \\
	&\lesssim 
	C_1 \varepsilon_-\bigg(\int_{0}^{t}\Big\| \frac{\langle u_{-}\rangle^{\omega}}{\langle u_{+}\rangle^{\frac{\omega}{2}}}j^{(\beta)}_{+}\Big\|^2_{L^2(\Sigma_\tau)}
	+\int_{0}^{t}\Big\| \frac{\langle u_{-}\rangle^{\omega}}{\langle u_{+}\rangle^{\frac{\omega}{2}}}\nabla j^{(\beta)}_{+}\Big\|^2_{L^2(\Sigma_\tau)}\bigg)\\
	&\stackrel{\eqref{eq:flux}}{\lesssim}\big(C_1\big)^3\varepsilon_+^2\varepsilon_-.
\end{align*}

\Emph{Case 2}: $|\beta|=N_*$. This is the top order term contributed from the quasi-linear part of the equations.  We need to apply integration by parts to save one derivative. In view of the fact that $\operatorname{div}z_-=0$,  we have
\begin{align*}
	\int_0^t\int_{\Sigma_\tau}\langle u_{-}\rangle^{2\omega}j^{(\beta)}_{+}\cdot \big(z_{-}\cdot\nabla j_{+}^{(\beta)}\big)
	&=\frac{1}{2}\int_0^t\int_{\Sigma_\tau}\langle u_{-}\rangle^{2\omega}z_{-}\cdot\nabla\big(\big| j_{+}^{(\beta)}\big|^2\big)\\
	&=-\frac{1}{2}\int_0^t\int_{\Sigma_\tau}\nabla\langle u_{-}\rangle^{2\omega} z_{-}\cdot\big| j_{+}^{(\beta)}\big|^2.
\end{align*}
In view of \eqref{differentiate weights coro}, we have $|\nabla\langle u_{-}\rangle^{2\omega}|\lesssim |\langle u_{-}\rangle^{2\omega}|$. Hence,
\begin{align*}
	\mathbf{K_{2}} & \lesssim\int_0^t\int_{\Sigma_\tau}|z_{-}|\cdot\langle u_{-}\rangle^{2\omega}\big|j_{+}^{(\beta)}\big|^2\\
	&\lesssim\int_0^t\int_{\Sigma_\tau}\big\|\langle u_{+}\rangle^{\omega} z_{-}\big\|_{L^\infty(\Sigma_\tau)}\cdot
	\frac{\langle u_{-}\rangle^{2\omega}}{\langle u_{+}\rangle^{\omega}}\big|j_{+}^{(\beta)}\big|^2\\
	&\stackrel{\eqref{eq:sobolev}}{\lesssim}C_1\varepsilon_-\int_0^t\int_{\Sigma_{\tau}}\frac{\langle u_{-}\rangle^{\omega}}{\langle u_{+}\rangle^{\frac{\omega}{2}}}\big|j^{(\beta)}_{+}\big|^2.
\end{align*}
Therefore, in view of \eqref{eq:flux} and the \Emph{Case 1}, for $0\leqslant|\beta|\leqslant N_*$, we always have
\begin{equation*}
	\mathbf{K_{2}}\lesssim \big(C_1\big)^3\varepsilon_+^2\varepsilon_-.
\end{equation*}
Hence, for $0\leqslant|\beta|\leqslant N_*$, we have
\begin{equation*}
\mathbf{K}_{\bm{+}}\lesssim \big(C_1\big)^3\varepsilon_+^2\varepsilon_-.
\end{equation*}
Combined with \eqref{eq:HIGH}, we obtain that for all $t\in [0,t^*]$ and for all $\beta$ with $0\leqslant|\beta|\leqslant N_*$, 
\begin{equation}\label{eq:H}
E_{\pm}^{(\beta)}(t)+F^{(\beta)}_{\pm}(t)\leqslant C_0 E_{\pm}^{(\beta)}(0)+C_0\big(C_1\big)^3\varepsilon_\pm^2\varepsilon_\mp,
\end{equation}
where $C_0$ is a universal constant. 
Summing up \eqref{eq:H} for all $0\leqslant |\beta|\leqslant N_{*}$ and taking  supremum over all  $t\in [0,t^*]$, we obtain  that
\begin{equation}\label{eq:HIGHEE}
	\sum_{k=0}^{ N_{*}}E_{\pm}^k+\sum_{k=0}^{ N_{*}} F_{\mp}^k\leqslant C_0 \sum_{k=0}^{ N_{*}}E_{\pm}^k(0)+C_0\big(C_1\big)^3\varepsilon_{\pm}^2\varepsilon_{\mp}.
\end{equation}

\subsection{End of the bootstrap argument}
According to \eqref{bootstrap assumption}, \eqref{eq:LOWEE} and \eqref{eq:HIGHEE}, for a universal constant $C_0$, we have
\begin{align*}
	\sum_{+,-}\bigg(E_{\pm}+F_{\pm}+\sum_{k=0}^{N_{*}}E_{\pm}^k+\sum_{k=0}^{N_{*}}F_{\mp}^k\bigg)
	&\leqslant\sum_{+,-}\Big( C_0\varepsilon_\pm^2+C_0\big(C_1\big)^3\varepsilon_\pm^2\varepsilon_\mp\Big)\\
	&=C_0\varepsilon^2+C_0\big(C_1\big)^3\varepsilon_+^2\varepsilon_-+C_0\big(C_1\big)^3\varepsilon_-^2\varepsilon_+\\
	&\leqslant
	C_0\varepsilon^2+C_0\big(C_1\big)^3\varepsilon^3.
\end{align*}
We then take $C_1=(3 C_0)^{\frac{1}{2}}$, $\varepsilon_0=\dfrac{2}{(3C_0)^{\frac{3}{2}}}$. Thus, for all $\varepsilon< \varepsilon_0$, the above inequality implies
\begin{equation*}
	\sum_{+,-}\bigg(E_{\pm}+F_{\pm}+\sum_{k=0}^{N_{*}}E_{\pm}^k+\sum_{k=0}^{N_{*}}F_{\mp}^k\bigg)\leqslant \big(C_1\big)^2\varepsilon^2,
\end{equation*}
which is the improved estimate \eqref{bootstrap assumption improved}. The proof of the Main Energy Estimates is now complete.

\begin{Remark}
The Main Energy Estimates is also independent of the choice of the initial slice $\Sigma_0$, namely, if we pose initial data on $\Sigma_{t_0}$ for some $t_0$, all the constants $C_0$, $C_1$, $\varepsilon_0$ and the estimates obtained in this section still hold. This is due to the time translation invariance of the MHD equations. 
\end{Remark}

\begin{Remark}
We also remark that, since we have completed the bootstrap argument, the constant $C_1$ from now on can be treated as a universal constant.
\end{Remark}

\begin{Remark}
	In view of \eqref{eq:LOWEE} and \eqref{eq:HIGHEE}, we have already proved the Refined Energy Estimates.
\end{Remark}

\medskip

We have two corollaries of the Main Energy Estimates which will be useful in the study of scattering fields in the next section.

\begin{Corollary}\label{coro:bound on p}
	For the solution constructed in the Main Energy Estimates, for $l=1,2$, for all $(\tau,x)\in \mathbb{R}\times \mathbb{R}^3$, we have the following bounds on pressure:
	\begin{equation}\label{eq:nablap}
	|\nabla^l p(\tau,x)|\lesssim\frac{\varepsilon_+\varepsilon_-}{\big(1+|\tau+a|\big)^\omega}.
	\end{equation}
\end{Corollary}
\begin{proof}
For $l=1$, we can use \eqref{decomposition of nabla p} and we bound $\nabla p$ by
\begin{equation*}
	|\nabla p(\tau,x)|
	\lesssim \underbrace{\int_{|x-x'|\leqslant 2}\!\!\!\!\frac{|(\nabla z_{-}\cdot\nabla z_{+})(\tau,x')|}{|x-x'|^2}dx'}_{\mathbf{A_1}}+\underbrace{\int_{|x-x'|\geqslant 1}\!\!\!\!\frac{|(z_{-}\cdot z_{+})(\tau,x')|}{|x-x'|^4}dx'}_{\mathbf{A_2}} +\underbrace{\int_{1\leqslant |x-x'|\leqslant 2}\! |(z_{-}\cdot z_{+})(\tau,x')|dx'}_{\mathbf{A_3}}.
\end{equation*}
We now estimate these three terms one by one. For $\mathbf{A_1}$, we have
\begin{align*}
\langle u_-\rangle^\omega\langle u_+\rangle^{\omega}|\mathbf{A_1}|
&=\int_{|x-x'|\leqslant 2}\frac{\big(\langle u_-\rangle^\omega\langle u_+\rangle^{\omega}\big)(\tau,x)|(\nabla z_{-}\cdot\nabla z_{+})(\tau,x')|}{|x-x'|^2}dx'\\
&\stackrel{\eqref{eq:xleq2}}{\lesssim}
 \int_{|x-x'|\leqslant 2}\frac{\big(\langle u_-\rangle^\omega\langle u_+\rangle^{ \omega}\big)(\tau,x')|(\nabla z_{-}\cdot\nabla z_{+})(\tau,x')|}{|x-x'|^2}dx'.
\end{align*}
Since we have proved that $\big|\langle u_\mp\rangle^\omega\nabla z_\pm\big|\lesssim \varepsilon_\pm$ in \eqref{eq:sobolev}, we conclude that
\begin{align*}
\langle u_-\rangle^\omega\langle u_+\rangle^{\omega}|\mathbf{A_1}|
& \lesssim
 \int_{|x-x'|\leqslant 2}\frac{\varepsilon_+\varepsilon_-}{|x-x'|^2}dx'\\
 &\lesssim \varepsilon_+\varepsilon_-.
\end{align*}
We can bound $\mathbf{A_3}$ in the same manner as for $\mathbf{A_1}$ and we obtain
\begin{equation*}
\langle u_+\rangle^\omega\langle u_-\rangle^\omega|\mathbf{A_3}|
\lesssim\varepsilon_+\varepsilon_-.
\end{equation*}
For $\mathbf{A_2}$, since \eqref{u+u-}, we have 
\[1+|\tau+a| \lesssim \langle u_{+}\rangle\langle u_{-}\rangle,\] 
and then
\begin{align*}
	\big(1+|\tau+a|\big)^{{\omega}}|\mathbf{A_2}|&\lesssim\int_{|x-x'|\geqslant 1}\frac{\langle u_+\rangle^\omega(\tau,x')\langle u_-\rangle^\omega(\tau,x')| z_{-}(\tau,x')||z_{+}(\tau,x')|}{|x-x'|^4}dx'\\
	&\lesssim \int_{|x-x'|\geqslant 1}\frac{\varepsilon_+\varepsilon_-}{|x-x'|^4}dx'\\
	&\lesssim \varepsilon_+\varepsilon_-.
\end{align*}
Putting the estimates of all the $\mathbf{A_i}$'s together, we obtain
\begin{equation}\label{nablap}
	|\nabla p(\tau,x)|\lesssim\frac{\varepsilon_+\varepsilon_-}{\big(1+|\tau+a|\big)^\omega}.
\end{equation}

To bound $\nabla^2 p$, we apply $\nabla$ on \eqref{eq:nabla p}. Similar to the derivation of \eqref{decomposition of nabla p}, we obtain
\begin{align*}
   |\nabla^2 p(\tau,x)|
	&\lesssim \underbrace{\sum_{l_{1},l_{2}=1}^2 \int_{|x-x'|\leqslant 2}\frac{1}{|x-x'|^2}\big|(\nabla^{l_1} z_{-}\cdot\nabla^{l_2} z_{+})(\tau,x')\big|dx'}_{\mathbf{B_1}}\\
	&\ \ 
	+\underbrace{\int_{|x-x'|\geqslant 1}\frac{1}{|x-x'|^4}\big| (z_{-}\cdot\nabla z_{+})(\tau,x')\big|dx'}_{\mathbf{B_2}} +\underbrace{\int_{1\leqslant |x-x'|\leqslant 2}\frac{1}{|x-x'|^3}\big|( z_{-}\cdot\nabla z_{+})(\tau,x')\big|dx'}_{\mathbf{B_3}},
\end{align*}
where $(l_1,l_2)=(1,1)$, $(1,2)$ or $(2,1)$. We then repeat the procedure of the estimates on  $\mathbf{A_i}$'s. This gives the estimates on $\mathbf{B_i}$'s and thus on $\nabla^2 p$.
\end{proof}

\begin{Corollary} For the solution constructed in the Main Energy Estimates, we have the following space-time estimates:
\begin{equation}\label{eq:LOW2}
\int_{\mathbb{R}\times \mathbb{R}^3} \langle u_{\mp}\rangle^{2\omega}\langle u_{\pm}\rangle^\omega |\nabla p +z_{\mp}\cdot\nabla z_{\pm}|^2\lesssim\varepsilon_+^2\varepsilon_-^2,
\end{equation}
and for each multi-index $\beta$ with $0\leqslant|\beta|\leqslant N_*-1$,
\begin{equation}\label{eq:HIGH2}
  \int_{\mathbb{R}\times \mathbb{R}^3}\langle u_{\mp}\rangle^{2\omega}\langle u_{\pm}\rangle^\omega \big|\rho^{(\beta)}_{\pm} \big|^2 \lesssim \varepsilon_+^2\varepsilon_-^2,
\end{equation}
where we refer to \eqref{definition of rhobeta} for the definition of $\rho_\pm^{(\beta)}$.
\end{Corollary}
\begin{proof}
By the symmetry considerations, we only give details for the estimates on  $\nabla p +z_{-}\cdot\nabla z_{+}$ and $\rho_+^{(\beta)}$.

For the first estimate, we have
\begin{align*}
    \int_{\mathbb{R}\times \mathbb{R}^3}\langle u_{-}\rangle^{2\omega}\langle u_{+}\rangle^\omega |\nabla p +z_{-}\cdot\nabla z_{+}|^2\leqslant\int_{\mathbb{R}\times \mathbb{R}^3}\langle u_{-}\rangle^{2\omega}\langle u_+\rangle^\omega|\nabla p|^2 +\int_{\mathbb{R}\times \mathbb{R}^3}\langle
	u_{-}\rangle^{2\omega}\langle u_+\rangle^\omega|z_-|^2|\nabla z_{+}|^2.
\end{align*}
The first term on the right hand side can be bounded exactly as the term $\mathbf{I}'$ in \eqref{the term I'}. Therefore, we can use the estimate from \eqref{final estimate on I'} and bound it by $\varepsilon_+^2\varepsilon_-^2$ (up to a universal constant). 
The second term on the right hand side can be bounded exactly as the term $\mathbf{J}'$ in \eqref{the term J'} and hence can be bounded by $\varepsilon_+^2\varepsilon_-^2$ (up to a universal constant) due to \eqref{final estimate on J'} . 
As a conclusion, we have
\[\int_{\mathbb{R}\times \mathbb{R}^3}\langle u_{-}\rangle^{2\omega}\langle u_{+}\rangle^\omega |\nabla p +z_{-}\cdot\nabla z_{+}|^2\lesssim\varepsilon_+^2\varepsilon_-^2.
\]

For the second estimate, for any multi-index $\beta$ with $0\leqslant|\beta|\leqslant N_*-1$, let 
\[\mathbf{II}=\int_{\mathbb{R}}\int_{\Sigma_{\tau}}\langle u_{-}\rangle^{2\omega}\langle u_{+}\rangle^\omega \big|\rho^{(\beta)}_{+} \big|^2.\]
In view of \eqref{rho beta explicit}, we have
\begin{equation}\label{term II1}
\mathbf{II}\lesssim\underbrace{\sum_{k\leqslant|\beta|}\int_{\mathbb{R}}\int_{\Sigma_{\tau}}\langle u_{-}\rangle^{2\omega}\langle u_{+}\rangle^\omega\big|\nabla z^{(k)}_{-} \big|^2\big|\nabla z^{(|\beta|-k)}_{+} \big|^2}_{\mathbf{II_1}}
+\underbrace{\int_{\mathbb{R}}\int_{\Sigma_\tau}\langle u_{-}\rangle^{2\omega}\langle u_{+}\rangle^\omega\big|z_{-}\big|^2\big|\nabla j_{+}^{(\beta)}\big|^2}_{\mathbf{II_2}}.
\end{equation}

\smallskip

To bound $\mathbf{II_1}$, we have two cases:

\smallskip

\Emph{Case 1:} $0\leqslant|\beta|\leqslant N_{*}-2$. We can bound $\nabla z_{-}^{(k)}$ in $L^\infty$ because $k+2\leqslant N_{*}$. Hence,
\begin{align*}
\mathbf{II_1} &\lesssim\sum_{k\leqslant|\beta|}   \int_{\mathbb{R}} \big\|\langle u_{+}\rangle^{\omega}\nabla z^{(k)}_{-}\big\|_{L^\infty(\Sigma_\tau)}^2\int_{\Sigma_{\tau}}\frac{\langle u_{-}\rangle^{2\omega}}{\langle u_{+}\rangle^{\omega}}\big|\nabla z^{(|\beta|-k)}_{+}\big|^2 \\
&\stackrel{\eqref{eq:sobolev}}{\lesssim}
\varepsilon_-^2\sum_{k\leqslant|\beta|}\int_{\mathbb{R}}\int_{\Sigma_{\tau}}\frac{\langle u_{-}\rangle^{2\omega}}{\langle u_{+}\rangle^{\omega}}\big|\nabla z^{(|\beta|-k)}_{+}\big|^2\\
&\stackrel{\eqref{eq:flux}}{\lesssim}\varepsilon_+^2\varepsilon_-^2.
\end{align*}

\smallskip

\Emph{Case 2:} $|\beta|= N_{*}-1$. We rewrite $\mathbf{II_1}$ as
\begin{equation*}
\mathbf{II_1}\lesssim\underbrace{\sum_{k\leqslant N_{*}-2}\int_{\mathbb{R}}\int_{\Sigma_{\tau}}\langle u_{-}\rangle^{2\omega}\langle u_{+}\rangle^\omega\big|\nabla z^{(k)}_{-} \big|^2\big|\nabla z^{(|\beta|-k)}_{+} \big|^2}_{\mathbf{II_{11}}}
+\underbrace{\int_{\mathbb{R}}\int_{\Sigma_{\tau}}\langle u_{-}\rangle^{2\omega}\langle u_{+}\rangle^\omega\big|\nabla z^{(|\beta|)}_{-} \big|^2\big|\nabla z_{+} \big|^2}_{\mathbf{II_{12}}}.
\end{equation*}

The terms $\mathbf{II_{11}}$ can be controlled in the same manner as in \Emph{Case 1} so that
\begin{equation*}
\mathbf{II_{11}}\lesssim\varepsilon_+^2\varepsilon_-^2.
\end{equation*}

For $\mathbf{II_{12}}$, we use $L^\infty$ estimates on $\nabla z_+$ to derive
\begin{align*}
\mathbf{II_{12}}&
=\int_{\mathbb{R}}\int_{\Sigma_{\tau}}
\langle u_{+}\rangle^{2\omega}
\big|\nabla z^{(|\beta|)}_{-}\big|^2\cdot
\frac{\langle u_{-}\rangle^{2\omega}}{\langle u_{+}\rangle^{\omega}}
\big|\nabla z_{+} \big|^2\\
& \lesssim
\Big\|\langle u_{+}\rangle^{\omega}
\nabla z^{(|\beta|)}_{-}
\Big\|^2_{L^\infty_{\tau}L^2_x}
\Big\|\frac{\langle u_{-}\rangle^{\omega}}{\langle u_{+}\rangle^{\frac{\omega}{2}}}
\nabla z_{+} 
\Big\|^2_{L^2_{\tau}L^\infty_x}\\
&\stackrel{\eqref{z to j inequality}, \eqref{bootstrap assumption}}{\lesssim}\varepsilon_-^2\Big\|\frac{\langle u_{-}\rangle^{\omega}}{\langle u_{+}\rangle^{\frac{\omega}{2}}}
\nabla z_{+} \Big\|^2_{L^2_{\tau}L^\infty_x}\\
&\stackrel{\eqref{eq:instead sobolev}}{\lesssim}\varepsilon_+^2\varepsilon_-^2.
\end{align*}

Therefore,
\begin{equation*}
\mathbf{II_1}\lesssim  \varepsilon_+^2\varepsilon_-^2.
\end{equation*}

To bound $\mathbf{II_2}$, we proceed as follows:
\begin{align*}
\mathbf{II_2} &\lesssim\int_{\mathbb{R}} \Big(\big\|\langle u_{+}\rangle^{\omega} z_{-}
\big\|_{L^\infty(\Sigma_\tau)}^2\int_{\Sigma_{\tau}}\frac{\langle u_{-}\rangle^{2\omega}}{\langle u_{+}\rangle^{\omega}}\big|\nabla j^{(\beta)}_{+}\big|^2 \Big)d\tau\\
&\stackrel{\eqref{eq:sobolev}}{\lesssim}
\varepsilon_-^2
\int_{\mathbb{R}\times \mathbb{R}^3}\frac{\langle u_{-}\rangle^{2\omega}}{\langle u_{+}\rangle^{\omega}}\big|\nabla j^{(\beta)}_{+}\big|^2\\
&\stackrel{\eqref{eq:flux}}{\lesssim}\varepsilon_+^2\varepsilon_-^2.
\end{align*}
Therefore, the estimates of $\mathbf{II_1}$ and $\mathbf{II_2}$ give the desired bound for $\mathbf{II}$. This completes the proof of the corollary.
\end{proof}

\section{The proof of the rigidity theorems}\label{main proof}

\subsection{The construction of  scattering fields and weight energy spaces at infinities}\label{SF}
By the symmetry considerations, we only consider the future scattering fields.
Given a point $(x_1,x_2,u_-) \in \mathcal{F}_+$ and
a point $(x_1,x_2,u_+) \in \mathcal{F}_-$, for the solution $\big(z_+(t,x),z_-(t,x)\big)$ constructed in the previous section, we  show that 
\begin{equation}\label{eq:def scattering}
\begin{cases}
	&\displaystyle z_+(+\infty;x_1,x_2,u_-):=z_+(0,x_1,x_2,u_-)-\int_0^{+\infty} \big(\nabla p+z_{-}\cdot\nabla z_{+}\big)(\tau,x_1,x_2,u_--\tau)d\tau,\\
	&\displaystyle z_-(+\infty;x_1,x_2,u_+):=z_-(0,x_1,x_2,u_+)-\int_0^{+\infty} \big(\nabla p+z_{+}\cdot\nabla z_{-}\big)(\tau,x_1,x_2,u_++\tau)d\tau
	\end{cases}
\end{equation}
are well-defined. Hence,  they define the scattering fields $z_+(+\infty;x_1,x_2,u_-)$ on $\mathcal{F}_+$ and $z_-(+\infty;x_1,x_2,u_+)$ on $\mathcal{F}_-$. In what follows, based on the symmetry considerations, we will only show the convergence for the integral in the definition of $z_+(+\infty;x_1,x_2,u_-)$.

In view of Lemma \ref{SE}, we have
\begin{equation*}
|(z_-\cdot\nabla z_+)(\tau,x_1,x_2,u_--\tau)|\lesssim \frac{\varepsilon_+\varepsilon_-}{\big(1+|\tau+a|\big)^\omega}.
\end{equation*}
In view of Corollary \ref{coro:bound on p}, we have 
\begin{equation*}
|\nabla p(\tau,x_1,x_2,u_--\tau)|\lesssim \frac{\varepsilon_+\varepsilon_-}{\big(1+|\tau+a|\big)^\omega}.
\end{equation*}
Thus, 
\begin{equation}
\label{eq:nablapz-nablaz+}
|(\nabla p+z_-\cdot\nabla z_+)(\tau,x_1,x_2,u_--\tau)|\lesssim\frac{\varepsilon_+\varepsilon_-}{\big(1+|\tau+a|\big)^\omega}.
\end{equation}
We observe that as a function of $\tau$,
$\big(1+|\tau+a|\big)^{-\omega} \!\in L^1(\mathbb{R})$ 
and its $L^1$-norm is independent of $a$. Therefore, the integral in the definition of $z_+(+\infty;x_1,x_2,u_-)$ in \eqref{eq:def scattering} converges. This implies that the scattering field $z_+(\infty;x_1,x_2,u_-)$ is well-defined point-wisely on $\mathcal{F}_+$.

\bigskip

On $\mathcal{F}_+$, we define the weighted Sobolev norms/spaces as follows: for any vector field $f$ on $\mathcal{F}_+$, we use the weighted measure 
\[\langle u_- \rangle^{2\omega} d\mu_- = \langle u_- \rangle^{2\omega} dx_1 dx_2 d u_-\] 
on $\mathcal{F}_+$. This leads to the definition of $L^2$-typespace $L^2(\mathcal{F}_+,\langle u_- \rangle^\omega d\mu_-)$. Similarly, by setting
\[\langle u_+ \rangle^{2\omega} d\mu_+ = \langle u_+ \rangle^{2\omega} dx_1 dx_2 d u_+,\]
we obtain the $L^2$-typespace $L^2(\mathcal{F}_-,\langle u_+ \rangle^{2\omega} d\mu_+)$ on $\mathcal{F}_-$. The key property of the scattering fields, which are point-wisely defined, is that they live in the Sobolev spaces based on the above measures:

\begin{Proposition}\label{prop:9}For all multi-indices $\beta$ with $0\leqslant\left|\beta\right|\leqslant N_*$, we have
\[\nabla^\beta z_{\pm}(+\infty;x_1,x_2,u_\mp)\in L^{2}(\mathcal{F}_\pm,\langle u_\mp \rangle^{2\omega} d\mu_{\mp}).\]
\end{Proposition}
\begin{Remark}
There is also a counterpart for this statement at past infinities:
	\[\nabla^\beta z_{\pm}(-\infty;x_1,x_2,u_\mp)\in L^{2}(\mathcal{P}_\pm,\langle u_\mp \rangle^{2\omega} d\mu_{\mp}),\]
	where all the objects are constructed in the same manner.
\end{Remark}
\begin{proof}
The proof is divided into two steps. The first step deals with the case where $\beta=0$. The second step deals with the cases with $|\beta|\geqslant 1$. By the symmetry considerations, we only study $z_+$ and its derivatives on the future infinity $\mathcal{F}_+$.

\medskip
\noindent
\Emph{Step 1: } We show that $z_+(+\infty;x_1,x_2,u_-)\in L^2(\mathcal{F}_+,\langle u_- \rangle^{2\omega} d\mu_-)$.

\medskip
\noindent
By definition, we have\begin{align*}
	& \ \ \ \int_{\mathcal{F}_+}|z_{+}(\infty;x_1,x_2,u_-)|^2\langle u_{-}\rangle^{2\omega}d\mu_-\\
	&=\int_{\mathbb{R}^3}\Big|z_+(0,x_1,x_2,u_-)-\int_0^\infty(\nabla p+z_{-}\cdot\nabla z_{+})(\tau,x_1,x_2,u_--\tau)d\tau\Big|^2 \langle u_{-}\rangle^{2\omega}dx_1dx_2du_-\\
	&\lesssim\underbrace{\int_{\mathbb{R}^3}|z_{+}(0,x_1,x_2,u_-)|^2\langle u_{-}\rangle^{2\omega}dx_1dx_2du_-}_{\mathbf{P_1}}\\
	&\ \ 
	+\underbrace{\int_{\mathbb{R}^3} \Big|\int_{0}^{\infty}|(\nabla p+z_{-}\cdot\nabla z_{+})(\tau,x_1,x_2,u_--\tau)|d\tau\Big|^2\langle u_{-}\rangle^{2\omega}dx_1dx_2du_-}_{\mathbf{P_2}}.
\end{align*}
For $\mathbf{P_1}$, we have
\[\mathbf{P_1}= \int_{\mathbb{R}^{3}}\langle u_{-}\rangle^{2\omega}|z_{+}(0,x)|^2dx.
\]
This is $E_+(0)$. Thus, 
\[\mathbf{P_1}\lesssim \varepsilon_+^2.
\]
Before we treat  $\mathbf{P_2}$, we first recall that we can use the following four coordinate systems on $\mathbb{R}\times \mathbb{R}^3$: the Cartesian coordinates $(t,x_1,x_2,x_3)$, two characteristic coordinates $(t,x_1,x_2,u_-)$ and $(t,x_1,x_2,u_+)$ and the double characteristic coordinates $(x_1,x_2,u_-,u_+)$.

Therefore, in the double characteristic coordinates, we can write the nonlinear term in the definition of the scattering fields \eqref{eq:def scattering} as
\[\int_0^\infty \big(\nabla p+z_{-}\cdot\nabla z_{+}\big)(\tau,x_1,x_2,u_--\tau)d\tau=\int_{u_-}^{-\infty} \big(\nabla p+z_{-}\cdot\nabla z_{+}\big)(x_1,x_2,u_-,u_+)du_+.\]

We can then bound $\mathbf{P}_2$ as follows:
\begin{align*}
	\mathbf{P_2}
	&\leqslant\int_{\mathbb{R}^3}\Big|\int_{\mathbb{R}}|(\nabla p+z_{-}\cdot\nabla z_{+})(x_1,x_2,u_-,u_+)|du_+\Big|^2 \langle u_{-}\rangle^{2\omega}dx_1dx_2du_-\\
	&\lesssim\int_{\mathbb{R}^3} \Big(\int_{\mathbb{R}}\frac{1}{\langle u_{+}\rangle^{\omega}}du_+\Big)\Big(\int_{\mathbb{R}}\langle u_+\rangle^\omega|(\nabla p+z_{-}\cdot\nabla z_{+})(x_1,x_2,u_-,u_+)|^2du_+\Big)\langle u_{-}\rangle^{2\omega}dx_1dx_2du_-\\
    &\lesssim\int_{\mathbb{R}\times\mathbb{R}^3}\langle u_{-}\rangle^{2\omega}\langle u_+\rangle^\omega|(\nabla p+z_{-}\cdot\nabla z_{+})(x_1,x_2,u_-,u_+)|^2dx_1dx_2du_-du_+.
\end{align*}
We change the double characteristic coordinates $(x_1,x_2,u_-,u_+)$ to the Cartesian coordinates $(\tau,x_1,x_2,x_3)$. According to \eqref{eq:LOW2}, we obtain
\begin{equation}\label{eq:z+infty}
\begin{split}
	\mathbf{P_2}
	&\lesssim\int_{\mathbb{R}\times \mathbb{R}^3}\langle u_{-}\rangle^{2\omega}\langle u_+\rangle^\omega|\nabla p+z_{-}\cdot\nabla z_{+}|^2 dx_1dx_2dx_3d\tau\\
	&\lesssim\varepsilon_+^2\varepsilon_-^2.
	\end{split}
\end{equation}
Putting the estimates on $\mathbf{P_1}$ and $\mathbf{P_2}$ together, we have proved that $z_+(+\infty;x_1,x_2,u_-)\in L^2(\mathcal{F}_+,\langle u_- \rangle^\omega d\mu_-)$. 

\bigskip

\noindent
\Emph{Step 2: } We show that for all multi-indices $\beta$ with $1\leqslant|\beta|\leqslant N_*$, we have
\[(\nabla^{\beta}z_+)(+\infty;x_1,x_2,u_-)\in L^2(\mathcal{F}_+,\langle u_- \rangle^{2\omega} d\mu_-).\]
Before we proceed, it is important to clarify the meaning of differential operators defined on $\mathcal{F}_+$: all the differentiations are done with respect to the given coordinates $(x_1,x_2,u_-)$. For example, for a vector field $f=(f_1,f_2,f_3)=f_1\partial_{x_1}+f_2\partial_{x_2}+f_3\partial_{u_-}$ defined on $\mathcal{F}_+$, its divergence and curl are defined as
\begin{equation*}\begin{split}
\operatorname{div} f &=\partial_{x_1}f_1+\partial_{x_2}f_2+ \partial_{u_-}f_3,\\
\operatorname{curl} f &=\big(\partial_{x_2}f_3-\partial_{u_-}f_2, \partial_{u_-}f_1-\partial_{x_1}f_3,\partial_{x_1}f_2-\partial_{x_2}f_1\big).
\end{split}\end{equation*}
Moreover, we differentiate the integral in the definition of $z_+(\infty;x_1,x_2,u_-)$ as follows:
\begin{Lemma}\label{Claim 1} For any partial derivative $D\in \big\{\partial_{x_1},\partial_{x_2}, \partial_{u_-}\big\}$, we have 
	\begin{equation*}
	D\Big(\int_0^{+\infty}(\nabla p+z_{-}\cdot\nabla z_{+})(\tau,x_1,x_2,u_--\tau)d\tau\Big)=\int_0^{+\infty} D(\nabla p+z_{-}\cdot\nabla z_{+})(\tau,x_1,x_2,u_--\tau)d\tau.
	\end{equation*}
	We remark that on the left hand side of the equation, $D$ is defined with respect to the coordinate system $(x_1,x_2,u_-)$ on $\mathcal{F}_+$, while on the right hand side of the equation, $D$ is defined with respect to the coordinate system $(t, x_1,x_2,u_-)$ on $\mathbb{R}\times \mathbb{R}^3$. Furthermore, we have 
\begin{equation*}\begin{split}
	\operatorname{div}\Big(
	\int_0^{+\infty}(\nabla p+z_{-}\cdot\nabla z_{+})(\tau,x_1,x_2,u_--\tau)d\tau\Big)&=0,\\
	\operatorname{curl}\Big(
	\int_0^{+\infty}(\nabla p+z_{-}\cdot\nabla z_{+})(\tau,x_1,x_2,u_--\tau)d\tau\Big)&=\int_0^{+\infty}\operatorname{curl } (z_{-}\cdot\nabla z_{+})(\tau,x_1,x_2,u_--\tau)d\tau.
	\end{split}\end{equation*}
\end{Lemma}
\begin{proof}
	We take $D=\partial_{x_1}$. For a fixed $\tau$, to prove the identity, we use the definition of derivatives: 
	\begin{equation*}
	\partial_{x_1}\!(\nabla p+z_{-}\cdot\nabla z_{+})(\tau,x_1,x_2,u_-\!-\tau\!) \!=\!\lim_{h\to 0}\!\frac{(\nabla p+\!z_{-}\!\cdot\!\nabla z_{+})
		(\tau,x_1\!+\!h,x_2,u_-\!\!-\!\tau\!)\!-\!\!(\nabla p+\!z_{-}\!\cdot\!\nabla z_{+})
		(\tau,x_1,x_2,u_-\!\!-\!\tau\!)}{h}.
	\end{equation*}
	By virtue of the mean value theorem, there exists $\theta\in(0,1)$ so that
	\begin{align*}
	&\ \ \ \  \Big|\frac{(\nabla p+z_{-}\cdot\nabla z_{+})
		(\tau,x_1+h,x_2,u_--\tau)-(\nabla p+z_{-}\cdot\nabla z_{+})
		(\tau,x_1,x_2,u_--\tau)}{h}\Big|\\
	&=\big|\partial_{x_1}(\nabla p+z_{-}\cdot\nabla z_{+})
	(\tau,x_1+\theta h,x_2,u_--\tau)\big|\\
	&\leqslant 
	\Big|\big(|\nabla^2 p|+|\nabla z_{-}\cdot\nabla z_{+}| +|z_{-}\cdot\nabla^2 z_{+}|\big)
	(\tau,x_1+\theta h,x_2,u_--\tau)\Big|.
	\end{align*}
	According to \eqref{eq:nablap} and \eqref{eq:separation}, we have
	\begin{equation*}
	 \Big|\frac{(\nabla p+z_{-}\cdot\nabla z_{+})
		(\tau,x_1+h,x_2,u_--\tau)-(\nabla p+z_{-}\cdot\nabla z_{+})
		(\tau,x_1,x_2,u_--\tau)}{h}\Big|\lesssim \frac{\varepsilon_+\varepsilon_-}{\big(1+|\tau+a|\big)^\omega}.
	\end{equation*}
	The dominant function $\dfrac{\varepsilon_+\varepsilon_-}{\big(1+|\tau+a|\big)^\omega}$ is integrable in $\tau$. Thus, we can apply the Lebesgue's dominated convergence theorem to commute the limit and the integral:
	\begin{align*}
	&\ \ \ \partial_{x_1}\Big(\int_0^{+\infty}(\nabla p+z_{-}\cdot\nabla z_{+})(\tau,x_1,x_2,u_--\tau)d\tau\Big)\\	
	&=\lim_{h\to 0}\int_0^{+\infty}\frac{(\nabla p+z_{-}\cdot\nabla z_{+})(\tau,x_1+h,x_2,u_--\tau)-(\nabla p+z_{-}\cdot\nabla z_{+})(\tau,x_1,x_2,u_--\tau)}{h}
	d\tau\\
	&=\int_0^{+\infty}\lim_{h\to 0}\frac{(\nabla p+z_{-}\cdot\nabla z_{+})
		(\tau,x_1+h,x_2,u_--\tau)-(\nabla p+z_{-}\cdot\nabla z_{+})(\tau,x_1,x_2,u_--\tau)}{h}d\tau\\
	&=\int_0^{+\infty} \partial_{x_1}(\nabla p+z_{-}\cdot\nabla z_{+})(\tau,x_1,x_2,u_--\tau)d\tau.
	\end{align*}
	This proves the identity for $D=\partial_{x_1}$. For the other $D\in \big\{\partial_{x_1},\partial_{x_2}, \partial_{u_-}\big\}$, the proof is exactly the same.
	
	Finally, we remark that in the Cartesian coordinates, for a vector field $f$, we have $\operatorname{div}f=\D_i f^i $ and $\operatorname{curl}f=\varepsilon_{ijk}\D_i f^j \D_k$. By changing to the characteristic coordinates $(t,x_1,x_2,u_-)$, it is straightforward to check that
	\begin{equation*}\begin{split}
	\operatorname{div} f &=\partial_{x_1}f_1+\partial_{x_2}f_2+ \partial_{u_-}f_3,\\
	\operatorname{curl} f &=\big(\partial_{x_2}f_3-\partial_{u_-}f_2, \partial_{u_-}f_1-\partial_{x_1}f_3,\partial_{x_1}f_2-\partial_{x_2}f_1\big).
	\end{split}\end{equation*}
	Therefore, we have
	\begin{equation*}\begin{split}
	\operatorname{div}\Big(
	\int_0^{+\infty}(\nabla p+z_{-}\cdot\nabla z_{+})(\tau,x_1,x_2,u_--\tau)d\tau\Big)&=\int_0^{+\infty}\operatorname{div } (\nabla p+z_{-}\cdot\nabla z_{+})(\tau,x_1,x_2,u_--\tau)d\tau,\\
	\operatorname{curl}\Big(
	\int_0^{+\infty}(\nabla p+z_{-}\cdot\nabla z_{+})(\tau,x_1,x_2,u_--\tau)d\tau\Big)&=\int_0^{+\infty}\operatorname{curl } (\nabla p+z_{-}\cdot\nabla z_{+})(\tau,x_1,x_2,u_--\tau)d\tau.
	\end{split}\end{equation*}
	To be more explicit, we compute the integrands on the right side of the equalities above. For the first one,
	by taking the divergence of the first equation in \eqref{eq:MHD}, we derive from $\operatorname{div}z_+=0$ that
	$$\operatorname{div}(\nabla p+z_-\cdot\nabla z_+ )=0$$
	and thus
	\begin{equation*}
	\operatorname{div}\Big(\int_0^{+\infty}(\nabla p+z_{-}\cdot\nabla z_{+})(\tau,x_1,x_2,u_--\tau)d\tau\Big)=0.
	\end{equation*}
	For the second one, we note that $\operatorname{curl}\nabla p=0$
	since $p$ is a scalar function, which yields
	\begin{equation*}
	\operatorname{curl}\Big(\int_0^{+\infty}(\nabla p+z_{-}\cdot\nabla z_{+})(\tau,x_1,x_2,u_--\tau)d\tau\Big)=\int_0^{+\infty}\operatorname{curl}(z_{-}\cdot\nabla z_{+})
	(\tau,x_1,x_2,u_--\tau)d\tau.
	\end{equation*}	
	This completes the proof of the lemma.
\end{proof}
\begin{Remark}
	According to the definition in \eqref{eq:def scattering}, we obtain
	\begin{equation*}
	\operatorname{div }\nabla^{\beta-1}\!z_{+}(+\infty;x_1,x_2,u_-)	=\operatorname{div }\nabla^{\beta-1}\!z_+(0,x_1,x_2,u_-)-\operatorname{div }\!\nabla^{\beta-1}\!\Big(\!\int_0^{+\infty}\!\!\!\!\!\!\!\underbrace{(\nabla p+z_{-}\cdot\nabla z_{+})(\tau,x_1,x_2,u_--\tau)}_{\text{use characteristic coordinates}}\!d\tau\!\Big).
	\end{equation*}
	Since $\operatorname{div}z_+=0$, the first term on the right side above becomes zero. By virtue of Lemma \ref{Claim 1}, the second term on the right side above also vanishes. 
	As a consequence, for all multi-indices $\beta$ with $1\leqslant|\beta|\leqslant N_*$, we have the divergence free property
	\[\operatorname{div} (\nabla^{\beta-1}z_+)(+\infty;x_1,x_2,u_-)=0.\]
\end{Remark}
Now we return to show that for all multi-indices $\beta$ with $1\leqslant|\beta|\leqslant N_*$, we have
\[(\nabla^{\beta}z_+)(+\infty;x_1,x_2,u_-)\in L^2(\mathcal{F}_+,\langle u_- \rangle^{2\omega} d\mu_-).\]
In view of Lemma \ref{d-c} and the divergence free property, it suffices to show that
\[\operatorname{curl} (\nabla^{\beta-1}z_+)(+\infty;x_1,x_2,u_-)\in L^2(\mathcal{F}_+,\langle u_- \rangle^{2\omega} d\mu_-).\]
In fact, we can bound the above quantity in $L^2$ as follows:
\begin{align*}
&\ \ \ \
\int_{\mathcal{F}_+} \big|\operatorname{curl }\nabla^{\beta-1}z_{+}(+\infty;x_1,x_2,u_-)\big|^2\langle u_{-}\rangle^{2\omega}d\mu_-\\
&=\int_{\mathcal{F}_+}\Big|\operatorname{curl }\nabla^{\beta-1}z_+(0,x_1,x_2,u_-)-\operatorname{curl }\nabla^{\beta-1}\Big(\int_0^{+\infty}\underbrace{(\nabla p+z_{-}\cdot\nabla z_{+})(\tau,x_1,x_2,u_--\tau)}_{\text{use characteristic coordinates}}d\tau\Big) \Big|^2\langle u_{-}\rangle^{2\omega}d\mu_-\\
&\lesssim \underbrace{\int_{\mathbb{R}^3} \big|\nabla^{\beta}z_+(0,x_1,x_2,u_-)\big|^2\langle u_{-}\rangle^{2\omega}dx_1dx_2du_-}_{\mathbf{Q_1}}\\
& \ \ \ +\underbrace{\int_{\mathbb{R}^3} \Big|\operatorname{curl }\nabla^{\beta-1}\Big(\int_0^{+\infty}(\nabla p+z_{-}\cdot\nabla z_{+})(\tau,x_1,x_2,u_--\tau)d\tau\Big)\Big|^2\langle u_{-}\rangle^{2\omega}dx_1dx_2du_-}_{\mathbf{Q_2}}.
\end{align*}
We can bound $\mathbf{Q_1}$ by the initial energy norms:
\begin{align*}
\mathbf{Q_1}&\lesssim \int_{\mathbb{R}^{3}}\langle u_{-}\rangle^{2\omega}\big|\nabla^{\beta} z_{+}(0,x_1,x_2,x_3)\big|^2dx_1 dx_2 dx_3\\
&\lesssim \int_{\mathbb{R}^{3}}\langle u_{-}\rangle^{2\omega}\big| j^{(\beta-1)}_{+}(0,x_1,x_2,x_3)\big|^2dx_1 dx_2 dx_3,
\end{align*}
where we have used Lemma \ref{d-c}. Thus, we can use $E^{(\beta-1)}_+(0)$ to control the last integral. Finally, we have
\[\mathbf{Q_1}\lesssim\varepsilon_+^2.\]

\medskip

We now prepare a sequence of lemmas to  bound $\mathbf{Q_2}$.

\begin{Lemma}\label{inL2-0}
For all multi-indices $\beta$ with $1\leqslant\left|\beta\right|\leqslant N_*$,  we have 
\[\int_0^{+\infty} \underbrace{\nabla^{\beta-1}\operatorname{curl }\big( z_{-}\cdot\nabla z_{+}\big)(\tau,x_1,x_2,u_--\tau)}_{\text{use characteristic coordinates}} d\tau\in L^2\big(\mathcal{F}_+,\langle u_-\rangle^{2\omega}d\mu_-\big).\]
\end{Lemma}
\begin{proof}
The proof is similar to that of the bound of $\mathbf{P_2}$ in the previous step:
\begin{align*}
	&\ \ \ \ 
	\int_{\mathbb{R}^3}\Big|\int_0^{+\infty } \nabla^{\beta-1}\operatorname{curl }( z_{-}\cdot\nabla z_{+})(\tau,x_1,x_2,u_--\tau) d\tau\Big|^2\langle u_{-}\rangle^{2\omega}d\mu_-&\\
	&\lesssim \int_{\mathbb{R}^3} \left(\int_{\mathbb{R}}\frac{1}{\langle u_{+}\rangle^{\omega}}du_+\right)\left(\int_{\mathbb{R}}\langle u_+\rangle^\omega\big|\nabla^{\beta-1}\operatorname{curl }( z_{-}\cdot\nabla z_{+})(x_1,x_2,u_-,u_+)\big|^2du_+\right)\langle u_{-}\rangle^{2\omega}dx_1dx_2du_-\\
	&\lesssim \int_{\mathbb{R}\times \mathbb{R}^3}\langle u_{-}\rangle^{2\omega}\langle u_{+}\rangle^\omega\big|\nabla^{\beta-1}\operatorname{curl }( z_{-}\cdot\nabla z_{+})(x_1,x_2,u_-,u_+)\big|^2dx_1dx_2du_-du_+.
	\end{align*}
This is exactly the term $\displaystyle{\int_{\mathbb{R}\times \mathbb{R}^3}\langle u_{-}\rangle^{2\omega}\langle u_{+}\rangle^\omega\big|\rho_+^{(\beta-1)}\big|^2}$ in \eqref{eq:HIGH2}. Therefore, the above integral is bounded by $\varepsilon_+^2\varepsilon_-^2$ up to a universal constant. This completes the proof of the lemma.
\end{proof}

\begin{Lemma}\label{inL2''-0}
For all multi-indices $\beta$ with $1\leqslant\left|\beta\right|\leqslant N_*$,  for all partial derivatives $D\in \big\{\partial_{x_1},\partial_{x_2}, \partial_{u_-}\big\}$,we have 
\[D\Big(\int_0^{+\infty}\underbrace{\nabla^{\beta-2}\operatorname{curl }(z_{-}\cdot\nabla z_{+})(\tau,x_1,x_2,u_--\tau)}_{\text{use characteristic coordinates}}d\tau\Big)\in L^2\big(\mathcal{F}_+,\langle u_-\rangle^{2\omega}d\mu_-\big).\]
\end{Lemma}
\begin{proof}
We consider the case where $D=\partial_{x_1}$. The other cases can be treated in the same way. Let
\[\mathbf{U}= \int_{\mathcal{F}_+}\Big|\partial_{x_1} \Big(\int_0^{+\infty}\nabla^{\beta-2}\operatorname{curl }(z_{-}\cdot\nabla z_{+})(\tau,x_1,x_2,u_--\tau)d\tau\Big)\Big|^2\langle u_{-}\rangle^{2\omega}d\mu_-.\]
By definition, we have
\begin{align*}
	\mathbf{U}&\!=\!\int_{\mathbb{R}^3}\!\Big|\!\lim_{h\to 0}\int_0^{+\infty}\!\!\frac{\nabla^{\beta-2}\!\operatorname{curl }(z_{-}\!\cdot\!\nabla z_{+})(\tau,x_1\!+\!h,x_2,u_-\!-\!\tau\!)\!-\!\nabla^{\beta-2}\!\operatorname{curl }(z_{-}\!\cdot\!\nabla z_{+})(\tau,x_1,x_2,u_-\!-\!\tau\!)\!}{h}d\tau\Big|^2\!\!\!\langle u_{-}\rangle^{2\omega}\! d\mu_-\\
	&\!=\!\int_{\mathbb{R}^3}\!\lim_{h\to 0}\!\Big|\!\int_0^{+\infty}\!\!\frac{\nabla^{\beta-2}\!\operatorname{curl }(z_{-}\!\cdot\!\nabla z_{+})(\tau,x_1\!+\!h,x_2,u_-\!-\!\tau\!)\!-\!\nabla^{\beta-2}\!\operatorname{curl }(z_{-}\!\cdot\!\nabla z_{+})(\tau,x_1,x_2,u_-\!-\!\tau\!)\!}{h}d\tau\Big|^2\!\!\!\langle u_{-}\rangle^{2\omega}\! d\mu_-.
\end{align*}
Therefore, by Fatou's lemma, we obtain
\begin{equation*}
	\mathbf{U}\! \leqslant\!\liminf_{h\to 0}\!\!\int_{\mathbb{R}^3}\!\!\Big|\!\!\int_0^{+\infty}\!\!\frac{\nabla^{\beta-2}\!\operatorname{curl }(z_{-}\!\cdot\!\nabla z_{+})(\tau,x_1\!+\!h,x_2,u_-\!-\!\tau\!)\!-\!\!\nabla^{\beta-2}\!\operatorname{curl }(z_{-}\!\cdot\!\nabla z_{+})(\tau,x_1,x_2,u_-\!-\!\tau\!)\!\!}{h}d\tau\!\Big|^{2}\!\!\!\langle u_{-}\rangle^{2\omega}\!d\mu_-.
\end{equation*}
By virtue of Newton-Leibniz formula, we have
\begin{align*}
	\mathbf{U}&\leqslant\liminf_{h\to 0}\underbrace{\int_{\mathbb{R}^3}\Big|\int_0^{+\infty}\int_{0}^{1}\partial_{x_1}\nabla^{\beta-2}\operatorname{curl}(z_{-}\cdot\nabla z_{+})(\tau,x_1+\theta h,x_2,u_--\tau)d\theta d\tau\Big|^2 \langle u_{-}\rangle^{2\omega}d\mu_-}_{\mathbf{U_h}}.
\end{align*}
For all $h$, we have
\begin{align*}
	\mathbf{U}_h&\leqslant\int_{\mathbb{R}^3}\Big|\int_{0}^{1}\int_0^{+\infty}\big|\partial_{x_1}\nabla^{\beta-2}\operatorname{curl}(z_{-}\cdot\nabla z_{+})(\tau,x_1+\theta h,x_2,u_--\tau)\big|d\tau d\theta\Big|^2 \langle u_{-}\rangle^{2\omega}dx_1dx_2du_-\\
	&\leqslant\int_{0}^{1}\underbrace{\int_{\mathbb{R}^3}\Big|\int_0^{+\infty}\big|\partial_{x_1}\nabla^{\beta-2}\operatorname{curl}(z_{-}\cdot\nabla z_{+})(\tau,x_1+\theta h,x_2,u_--\tau)\big|d\tau\Big|^2 \langle u_{-}\rangle^{2\omega}dx_1dx_2du_-}_{\mathbf{U}_{h,\theta}}d\theta.
\end{align*}
In $\mathbf{U}_{h,\theta}$, we observe that $\theta$ and $h$ do not depend on $x_1$. By change of variable $x_1\rightarrow X_1=x_1+\theta h$, we have
\begin{align*}
	\mathbf{U}_{h,\theta}&\lesssim\int_{\mathbb{R}^3}\Big|\int_0^{+\infty} \big|\partial_{x_1}\nabla^{\beta-2}\operatorname{curl }(z_{-}\cdot\nabla z_{+})(\tau,X_1,x_2,u_--\tau)\big| d\tau\Big|^2 \langle u_{-}\rangle^{2\omega}dX_1dx_2du_-\\
	&\lesssim \varepsilon_+^2\varepsilon_-^2.
\end{align*}
For the last step, we have used the proof of Lemma \ref{inL2-0}. As an immediate consequence, $\mathbf{U}$ is finite. Hence, the proof of the lemma is complete.
\end{proof}

\begin{Lemma}\label{claim-0}
For all multi-indices $\beta$ with $2\leqslant\left|\beta\right|\leqslant N_*$, for all $(x_1,x_2,u_-)\in \mathcal{F}_+$, for all partial derivatives $D\in \big\{\partial_{x_1},\partial_{x_2}, \partial_{u_-}\big\}$, as vector fields in $L^2(\mathcal{F}_+,\langle u_{-}\rangle^{2\omega} d\mu_-)$, we have
\begin{equation*}
	D\Big(\!\int_0^{+\infty}\!\!\!\!\!\nabla^{\beta-2}\!\operatorname{curl }(z_{-}\cdot\nabla z_{+})(\tau,x_1,x_2,u_--\tau)d\tau\Big)  \stackrel{L^2(\mathcal{F}_+,\langle u_{-}\rangle^{2\omega} d\mu_-)}{=\joinrel=\joinrel=\joinrel=} \!\int_0^{+\infty}\!\!\!\!\! D\nabla^{\beta-2}\!\operatorname{curl }(z_{-}\cdot\nabla z_{+})(\tau,x_1,x_2,u_--\tau)d\tau,
\end{equation*}
and therefore 
\begin{equation*}
\nabla\Big(\!\int_0^{+\infty}\!\!\!\!\nabla^{\beta-2}\!\operatorname{curl }(z_{-}\cdot\nabla z_{+})(\tau,x_1,x_2,u_--\tau)d\tau\Big)  \stackrel{L^2(\mathcal{F}_+,\langle u_{-}\rangle^{2\omega} d\mu_-)}{=\joinrel=\joinrel=\joinrel=} \!\int_0^{+\infty}\! \nabla^{\beta-1}\!\operatorname{curl }(z_{-}\cdot\nabla z_{+})(\tau,x_1,x_2,u_--\tau)d\tau.
\end{equation*}
\end{Lemma}
\begin{proof}
In view of Lemma \ref{inL2-0} and Lemma \ref{inL2''-0}, it suffices to show that
\[D\Big(\int_0^{+\infty}\nabla^{\beta-2}\operatorname{curl }(z_{-}\cdot\nabla z_{+})(\tau,x_1,x_2,u_--\tau)d\tau\Big)  \stackrel{\mathcal{D}'(\mathcal{F}_+)}{=\joinrel=\joinrel=} \int_0^{+\infty}D\nabla^{\beta-2}\operatorname{curl }(z_{-}\cdot\nabla z_{+})(\tau,x_1,x_2,u_--\tau)d\tau\]
in the sense of distributions. Similar as before, we may assume $D=\partial_{x_1}$.

We take a vector field $\varphi\in \mathcal{D}(\mathbb{R}^3)$ and we define the pairing $\mathbf{V}$ in the sense of distributions:
\begin{align*}
	\mathbf{V}&= \bigg\langle D\Big(\int_0^{+\infty}\nabla^{\beta-2}\operatorname{curl }(z_{-}\cdot\nabla z_{+})(\tau,x_1,x_2,u_--\tau)d\tau\Big),\varphi(x_1,x_2,u_-)\bigg\rangle\\
	&=-\bigg\langle\int_0^{+\infty}\nabla^{\beta-2}\operatorname{curl }(z_{-}\cdot\nabla z_{+})(\tau,x_1,x_2,u_--\tau)d\tau, D \varphi(x_1,x_2,u_-)\bigg\rangle\\
	&=-\int_{\mathbb{R}^3}\Big(\int_0^{+\infty}\nabla^{\beta-2}\operatorname{curl }(z_{-}\cdot\nabla z_{+})(\tau,x_1,x_2,u_--\tau)d\tau\Big)\cdot D\varphi(x_1,x_2,u_-)d\mu_-.
\end{align*}
In the last step, we have used Lemma \ref{inL2-0} to conclude that $\displaystyle\int_0^{+\infty}\nabla^{\beta-2}\operatorname{curl }(z_{-}\cdot\nabla z_{+})(\tau,x_1,x_2,u_--\tau)d\tau$ is a locally integrable function. We will apply Fubini's theorem to $\mathbf{V}$ to commute the integrals. Therefore, it is natural to consider the following spacetime integral:
\begin{equation}\label{eq:a1}
\begin{split}
&\ \ \ \
\int_{[0,+\infty)\times \mathbb{R}^3}\big|\nabla^{\beta-2}\operatorname{curl }(z_{-}\cdot\nabla z_{+})(\tau,x_1,x_2,u_--\tau)\cdot D\varphi(x_1,x_2,u_-)\big|d\tau dx_1dx_2du_-\\
&\lesssim \int_{\mathbb{R}^3}\int_{\mathbb{R}}\big|\nabla^{\beta-2}\operatorname{curl }(z_{-}\cdot\nabla z_{+})(x_1,x_2,u_-,u_+)\big|\cdot|D\varphi(x_1,x_2,u_-)|dx_1dx_2du_-du_+\\
&\lesssim
\bigg(\underbrace{\int_{\mathbb{R}\times \mathbb{R}^3}\langle  u_+\rangle^{\omega}\big|\nabla^{\beta-2}\operatorname{curl }(z_{-}\cdot\nabla z_{+})(x_1,x_2,u_-,u_+)\big|^2 }_{\mathbf{V_1}}\bigg)^{\frac{1}{2}}
\Bigg(\underbrace{\int_{\mathbb{R}\times \mathbb{R}^3}\frac{|D\varphi(x_1,x_2,u_-)|^2}{\langle u_+\rangle^{\omega}}}_{\mathbf{V_2}}\Bigg)^{\frac{1}{2}}.
\end{split}
\end{equation}
In view of the fact that $\langle u_-\rangle^{2\omega}\geqslant 1$, we can bound $\mathbf{V_1}$ as follows:
\begin{equation*}
\mathbf{V_1}
\lesssim
\int_{\mathbb{R}\times \mathbb{R}^3}\langle u_-\rangle^{2\omega}\langle u_+\rangle^{\omega}\big|\nabla^{\beta-2}\operatorname{curl }(z_{-}\cdot\nabla z_{+})(x_1,x_2,u_-,u_+)\big|^2.
\end{equation*}
This is the term treated in Lemma \ref{inL2-0}. Therefore,
\begin{align*}
\mathbf{V_1}&\lesssim \varepsilon_+^2\varepsilon_-^2.
\end{align*}
We note that $\varphi\in \mathcal{D}(\mathbb{R}^3)$, then $D\varphi\in \mathcal{D}(\mathbb{R}^3)\subset L^2(\mathbb{R}^3)$ and we can use the flux to bound
\begin{equation*}
\mathbf{V_2}= \int_{\mathbb{R}}\frac{1}{\langle u_+\rangle^{\omega}}\Big(\int_{C^{+,t}_{u_+}} |D \varphi(x_1,x_2,u_-)|^2  d\sigma_+ \Big)du_+.
\end{equation*}
Since $\varphi$ is a smooth function with compact support, we obtain
\begin{equation*}
\mathbf{V_2}\lesssim \|\nabla \varphi\|_{L^1}.
\end{equation*}
Thus, we conclude that
\begin{equation*}
\int_{[0,+\infty)\times \mathbb{R}^3}\big|\nabla^{\beta-2}\operatorname{curl }(z_{-}\cdot\nabla z_{+})(\tau,x_1,x_2,u_--\tau)\cdot D\varphi(x_1,x_2,u_-)\big|d\tau dx_1dx_2du_-<\infty.
\end{equation*}
Therefore, we can apply Fubini's theorem to $\mathbf{V}$ to derive
\begin{align*}
	\mathbf{V}&=-\int_{[0,+\infty) \times \mathbb{R}^3}\nabla^{\beta-2}\operatorname{curl }(z_{-}\cdot\nabla z_{+})(\tau,x_1,x_2,u_--\tau)\cdot D\varphi(x_1,x_2,u_-)d\tau dx_1dx_2du_-\\
	&=-\int_{[0,+\infty)} \bigg(\int_{\mathbb{R}^3}\nabla^{\beta-2}\operatorname{curl }(z_{-}\cdot\nabla z_{+})(\tau,x_1,x_2,u_--\tau)\cdot D\varphi(x_1,x_2,u_-)dx_1dx_2du_-\bigg)d\tau \\
	&=\int_{[0,+\infty)} \bigg(\int_{\mathbb{R}^3}D\nabla^{\beta-2}\operatorname{curl }(z_{-}\cdot\nabla z_{+})(\tau,x_1,x_2,u_--\tau)\cdot \varphi(x_1,x_2,u_-)dx_1dx_2du_-\bigg)d\tau \\
	&=\int_{[0,+\infty)\times \mathbb{R}^3}D\nabla^{\beta-2}\operatorname{curl }(z_{-}\cdot\nabla z_{+})(\tau,x_1,x_2,u_--\tau)\cdot \varphi(x_1,x_2,u_-)dx_1dx_2du_-d\tau.
\end{align*}
We can repeat the argument in \eqref{eq:a1} to show that 
\[\int_{[0,+\infty)\times \mathbb{R}^3}\big|D\nabla^{\beta-2}\operatorname{curl }(z_{-}\cdot\nabla z_{+})(\tau,x_1,x_2,u_--\tau)\cdot \varphi(x_1,x_2,u_-)\big|dx_1dx_2du_-d\tau<\infty.\]
Therefore, we can use Fubini's theorem again to derive
\begin{align*}
	\mathbf{V}&=\int_{\mathbb{R}^3}\Big(\int_{0}^{+\infty} D\nabla^{\beta-2}\operatorname{curl }(z_{-}\cdot\nabla z_{+})(\tau,x_1,x_2,u_--\tau)d\tau\Big)\cdot \varphi(x_1,x_2,u_-)dx_1dx_2du_-\\
	&=\bigg\langle\int_0^{+\infty} D\nabla^{\beta-2}\operatorname{curl }(z_{-}\cdot\nabla z_{+})(\tau,x_1,x_2,u_--\tau)d\tau,\varphi(x_1,x_2,u_-)\bigg\rangle.
\end{align*}
This proves the lemma.
\end{proof}
By induction on $\beta$, the above lemma has the following immediate consequence:
\begin{Corollary}\label{Claim 2-0} For all multi-indices $\beta$ with $2\leqslant\left|\beta\right|\leqslant N_*$,  as vector fields in $L^2(\mathcal{F}_+,\langle u_{-}\rangle^{2\omega} d\mu_-)$, we have
\begin{equation}\label{L2beta-0}
\nabla^{\beta-1}\!\Big(\!\!\int_0^{+\infty}\!\!\!\!\!\!\operatorname{curl}(z_{-}\cdot\nabla z_{+})(\tau,x_1,x_2,u_-\!-\tau)d\tau\!\Big)
\!\stackrel{L^2(\mathcal{F}_+,\langle u_{-}\rangle^{2\omega} d\mu_-)}{=\joinrel=\joinrel=\joinrel=}\!\!\!\! \int_0^{+\infty}\!\!\!\!\!\nabla^{\beta-1}\!\operatorname{curl}(z_{-}\cdot\nabla z_{+})(\tau,x_1,x_2,u_-\!-\tau)d\tau.
\end{equation}
\end{Corollary}

\bigskip

We return to the proof of Proposition \ref{prop:9}. It remains to bound $\mathbf{Q_2}$:
\begin{align*}
\mathbf{Q_2}
&=\int_{\mathbb{R}^3}\left|\nabla^{\beta-1}\operatorname{curl}\Big(\int_0^{+\infty}(\nabla p+z_{-}\cdot\nabla z_{+})(\tau,x_1,x_2,u_--\tau)d\tau\Big)\right|^2\left<u_{-}\right>^{2\omega}dx_1dx_2du_-\\
&\stackrel{\text{Lemma \ref{Claim 1}}}{=}
\int_{\mathbb{R}^3}\left|\nabla^{\beta-1}\Big(\int_0^{+\infty}\operatorname{ curl }(z_{-}\cdot\nabla z_{+})(\tau,x_1,x_2,u_--\tau)d\tau\Big)\right|^2\left<u_{-}\right>^{2\omega}dx_1dx_2du_-\\
&\stackrel{\text{Corollary \ref{Claim 2-0}}}{=}\int_{\mathbb{R}^3}
\left|\int_0^{+\infty} \nabla^{\beta-1}\operatorname{curl}(z_{-}\cdot\nabla z_{+})(\tau,x_1,x_2,u_--\tau)d\tau\right|^2\left<u_{-}\right>^{2\omega}dx_1dx_2du_-.
\end{align*}		
Thus, we can apply the proof of Lemma \ref{inL2-0} to conclude that
\[\mathbf{Q_2}\lesssim \varepsilon_+^2\varepsilon_-^2.\]
Together with the estimate of $\mathbf{Q_1}$, we obtain
\[\operatorname{curl} (\nabla^{\beta-1}z_+)(\infty;x_1,x_2,u_-)\in L^2(\mathcal{C}_+,\langle u_- \rangle^{2\omega} d\mu_-).\]
This completes the proof of Proposition \ref{prop:9}.
\end{proof}

\subsection{The large time behavior of the solution}
In the previous subsection, we have showed that for all multi-indices $\beta$ with $0\leqslant|\beta|\leqslant N_*$, the scattering fields  $z_{\pm}(+\infty;x_1,x_2,u_\mp)$ satisfy:
\[\nabla^{\beta}z_{\pm}(+\infty;x_1,x_2,u_\mp)\in L^2\big(\mathcal{F}_\pm,\langle u_\mp\rangle^{2\omega}d\mu_\mp\big).\] 
We will prove in the following two lemmas that for large $T>0$, the real solutions $z_{\pm}(T,x_1,x_2,u_\mp\mp T)$ converging to the scattering fields also make sense in the corresponding weighted Sobolev spaces.
\begin{Lemma}\label{lemma:15}When $T$ approaches $+\infty$, we have
\[\lim_{T\rightarrow +\infty}\int_{\mathbb{R}^3}|z_{\pm}(+\infty;x_1,x_2,u_\mp)-z_{\pm}(T,x_1,x_2,u_\mp\mp T)|^2\langle u_\mp\rangle^{2\omega}dx_1dx_2du_\mp = 0.\]
Here, we remark: we use a common coordinate system $(x_1,x_2,u_-)$ so that we can compare two vector fields $z_{+}(+\infty;x_1,x_2,u_-)$ and $z_{+}(T,x_1,x_2,u_--T)$ (they are defined on different spaces);
we use a common coordinate system $(x_1,x_2,u_+)$ so that we can compare two vector fields $z_{-}(+\infty;x_1,x_2,u_+)$ and $z_{-}(T,x_1,x_2,u_++T)$ (they are defined on different spaces).	
\end{Lemma}
\begin{proof}
Based on the symmetry considerations, we only study $z_+$. First of all, we have
\begin{align*}
\mathbf{X}_T&=
\int_{\mathbb{R}^3}|z_{+}(+\infty;x_1,x_2,u_-)-z_{+}(T,x_1,x_2,u_--T)|^2\langle u_-\rangle^{2\omega}dx_1dx_2du_- \\
&=\int_{\mathbb{R}^3} \Big|\int_{T}^{+\infty}(\nabla p+z_{-}\cdot\nabla z_{+})(\tau,x_1,x_2,u_--T)d\tau\Big|^2\langle u_{-}\rangle^{2\omega}dx_1dx_2du_- \\
&\leqslant\int_{\mathbb{R}^3} \Big|\int_{\mathbb{R}}(\nabla p+z_{-}\cdot\nabla z_{+})(x_1,x_2,u_-,u_+)\cdot \text{\Large$\text{\Large$\chi$}$}_{\tau\geqslant T}(x_1,x_2,u_-,u_+)du_+\Big|^2 \langle u_{-}\rangle^{2\omega}dx_1dx_2du_-.
\end{align*}
We remark that the characteristic function $\text{\Large$\text{\Large$\chi$}$}_{\tau\geqslant T}$ is defined on $\mathbb{R}\times \mathbb{R}^3$. We then use the following trick again:
\begin{align*}
\mathbf{X}_T
&\lesssim\int_{\mathbb{R}^3}\!\! \Big(\!\!\int_{\mathbb{R}}\frac{du_+}{\langle u_{+}\rangle^{\omega}}\!\Big)\Big(\!\!\int_{\mathbb{R}}\langle u_+\rangle^\omega|(\nabla p+z_{-}\cdot\nabla z_{+})(x_1,x_2,\!u_-,\!u_+)|^2\!\cdot\! \text{\Large$\text{\Large$\chi$}$}_{\tau\geqslant T}(x_1,x_2,\!u_-,\!u_+)du_+\!\Big)\!\langle u_{-}\rangle^{2\omega} \!dx_1dx_2du_-\\
&\lesssim\int_{\mathbb{R}\times \mathbb{R}^3}\langle u_{-}\rangle^{2\omega}\langle u_+\rangle^\omega|(\nabla p+z_{-}\cdot\nabla z_{+})(x_1,x_2,u_-,u_+)|^2\cdot \text{\Large$\text{\Large$\chi$}$}_{\tau\geqslant T}(x_1,x_2,u_-,u_+)dx_1dx_2du_-du_+\\
&=
\int_{T}^{+\infty}\int_{\mathbb{R}^3}\langle u_{-}\rangle^{2\omega}\langle u_+\rangle^\omega|(\nabla p+z_{-}\cdot\nabla z_{+})(\tau,x)|^2 dxd\tau.
\end{align*}
By \eqref{eq:LOW2}, the above integral is finite. Therefore, when $T\rightarrow +\infty$, the above integral decays to $0$. This proves the lemma.
\end{proof}

\begin{Lemma}\label{lemma:16}
For all multi-indices $\beta$ with $1\leqslant|\beta|\leqslant N_*$, we have
\[
\lim_{T\rightarrow +\infty}\int_{\mathbb{R}^3}\big|\nabla^{\beta}z_{\pm}(+\infty;x_1,x_2,u_\mp)-\nabla^{\beta}z_{\pm}(T,x_1,x_2,u_\mp\mp T)\big|^2\langle u_\mp\rangle^{2\omega}dx_1dx_2du_\mp=0.
\]\end{Lemma}
\begin{proof}
Based on the symmetry considerations, we only study the derivatives of $z_+$.	
First of all, using the divergence free property, we have
\begin{align*}
\mathbf{Y}_T&=\int_{\mathbb{R}^3}\big|\nabla^{\beta}z_{+}(+\infty;x_1,x_2,u_-)-\nabla^{\beta}z_{+}(T,x_1,x_2,u_--T)\big|^2\langle u_-\rangle^{2\omega}dx_1dx_2du_-\\
&\lesssim\int_{\mathbb{R}^3} \Big|\operatorname{curl }\nabla^{\beta-1}\big(z_{+}(+\infty;x_1,x_2,u_-)-z_{+}(T,x_1,x_2,u_--T)\big)\Big|^2\langle u_{-}\rangle^{2\omega}dx_1dx_2du_-\\
&=\int_{\mathbb{R}^3}
\Big|\int_T^{+\infty} \nabla^{\beta-1}\operatorname{curl }(z_{-}\cdot\nabla z_{+})(\tau,x_1,x_2,u_--\tau)d\tau\Big|^2\langle u_{-}\rangle^{2\omega}dx_1dx_2du_-.
\end{align*}
We then proceed in the same manner as $\mathbf{X}_T$ in the previous lemma:
\begin{align*}
\mathbf{Y}_T
&\lesssim\int_{\mathbb{R}^3}\!\!\Big(\!\!\int_{\mathbb{R}}\frac{du_+}{\langle u_{+}\rangle^{\omega}}\Big)\Big(\!\!\int_{\mathbb{R}}\langle u_+\rangle^\omega\big|\nabla^{\beta-1}\!\operatorname{curl}(z_{-}\cdot\nabla z_{+})(x_1,x_2,\!u_-,\!u_+)\big|^2\!\cdot\!\text{\Large$\text{\Large$\chi$}$}_{\tau\geqslant T}(x_1,x_2,\!u_-,\!u_+)\Big) \langle u_{-}\rangle^{2\omega}\!dx_1dx_2du_-\\
&\lesssim
\int_{T}^{+\infty}\int_{\mathbb{R}^3}\langle u_{-}\rangle^{2\omega}\langle u_+\rangle^\omega\big|\nabla^{\beta-1}\operatorname{curl }(z_{-}\cdot\nabla z_{+})(\tau,x)\big|^2 dxd\tau.
\end{align*}
We have already seen in the proof of Lemma \ref{inL2-0} that this term is bounded above by $\varepsilon_+^2\varepsilon_-^2$ up to a universal constant. Therefore, when $T\rightarrow +\infty$, the above integral decays to $0$. This proves the lemma.
\end{proof}
Similarly, we also have a statement concerning the convergence of  solutions towards the past infinites:
\begin{Lemma}\label{lemma:16-}
	For all multi-indices $\beta$ with $0\leqslant|\beta|\leqslant N_*$, we have
	\[
	\lim_{T\rightarrow +\infty}\int_{\mathbb{R}^3}\big|\nabla^{\beta}z_{\pm}(-\infty;x_1,x_2,u_\mp)-\nabla^{\beta}z_{\pm}(-T,x_1,x_2,u_\mp\pm T)\big|^2\langle u_\mp\rangle^{2\omega}dx_1dx_2du_\mp=0.
	\]\end{Lemma}

\subsection{The rigidity theorem 1}\label{IFT}

We now assume that the scattering fields $z_\pm(+\infty;x_1,x_2,u_\mp)$ vanish identically at infinities, i.e., 
\[z_\pm(+\infty;x_1,x_2,u_\mp)\equiv 0\  \text{on $\mathcal{F}_\pm$.}\]

Let $\epsilon<\varepsilon_0$ be an arbitrarily given small positive constant. According to Lemma \ref{lemma:15} and Lemma \ref{lemma:16}, there exists a $T_\epsilon>0$ such that we have the following smallness condition:
\[\sum_{+,-}\sum_{0\leqslant|\beta|\leqslant N_*}
\int_{\mathbb{R}^3}\big|\nabla^{\beta}z_{\pm}(T_\epsilon,x_1,x_2,u_\mp\mp T_\epsilon)\big|^2\langle u_\mp\rangle^{2\omega}dx_1dx_2du_\mp<\epsilon^2.
\]
We now study the position parameter $a$. At initial slice $\Sigma_0$, the position parameter $a_0=0$ is given so that we have constructed the solution $z_\pm(t,x_1,x_2,x_3)$. At time slice $\Sigma_{T_\epsilon}$, within the Cartesian coordinates, the above smallness condition can be expressed as
\[\sum_{+,-}\sum_{0\leqslant|\beta|\leqslant N_*}
\int_{\mathbb{R}^3}\big|\nabla^{\beta}z_{\pm}(T_\epsilon,x_1,x_2,x_3)\big|^2\big(1+|x_3\pm T_\epsilon|^2\big)^{\omega}dx_1dx_2dx_3<\epsilon^2.
\]
\begin{center}
\includegraphics[width=4in]{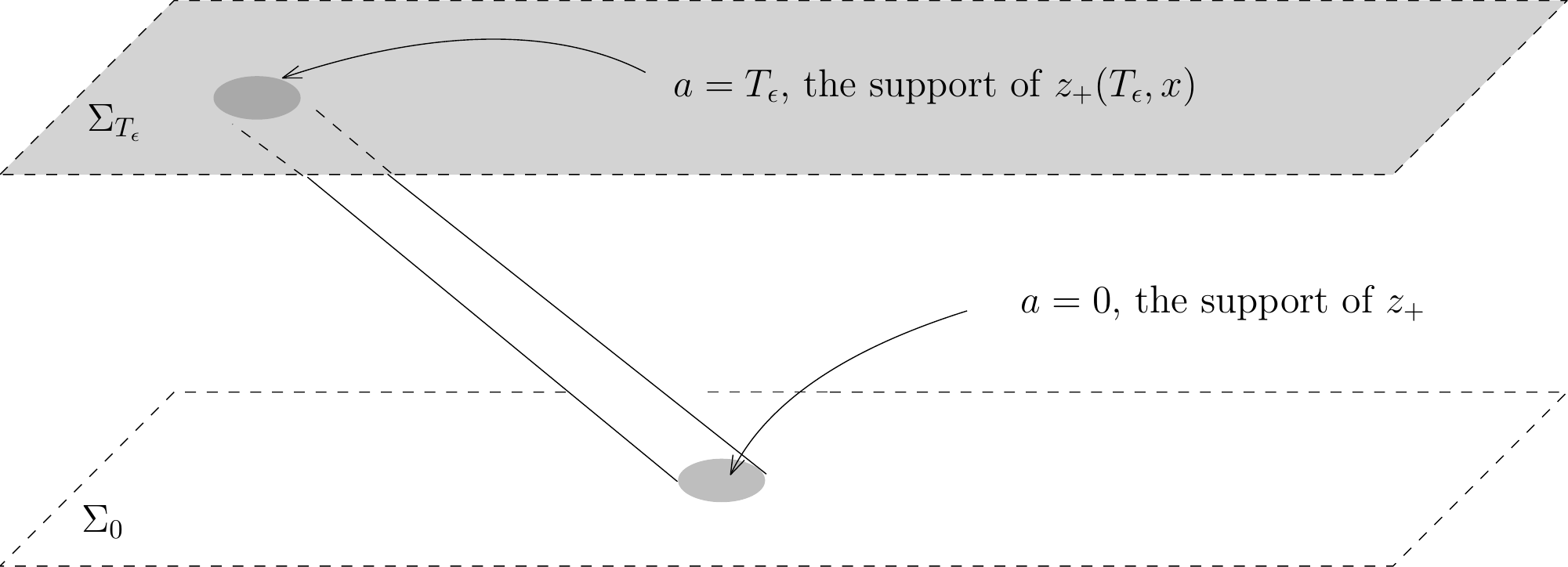}
\end{center}
\smallskip
\noindent
Therefore, if we take $a=T_\epsilon$ as the new position parameter and regard $\big(z_{+}(T_\epsilon,x),z_{-}(T_\epsilon,x)\big)$ as the initial data for the MHD system \eqref{eq:MHD}, we can solve the equations backwards in time. The smallness condition then can be re-expressed as 
\begin{equation*}
	\mathcal{E}^{N_*'}(T_\epsilon) 
	<\epsilon^2,
\end{equation*}
where $N_*'=6$. Since $N_*'\geqslant 5$, we can apply the Main Energy Estimates in this situation. 
In view of the fact that the estimates are independent of the position parameter, we conclude that at time slice $\Sigma_0$, there holds
\begin{equation*}
\mathcal{E}^{N_*'}(0)\leqslant C\epsilon^2,
\end{equation*}
where $C$ is a universal constant. We also remark that on $\Sigma_0$, the new weights associated to $a$ indeed coincide with the original weights, i.e., $\big(1+|x_3|^2\big)^{\omega}$.

Since $\epsilon$ is arbitrary, we arrive at the conclusion that 
\begin{equation*}
\mathcal{E}^{N_*'}(0)=0.
\end{equation*}
This means that the Alfv\'en waves $z_{+}(t,x)$ and $z_{-}(t,x)$ vanish identically.

\subsection{The rigidity theorem 2}\label{IFT2}
\begin{center}
\includegraphics[width=2.5in]{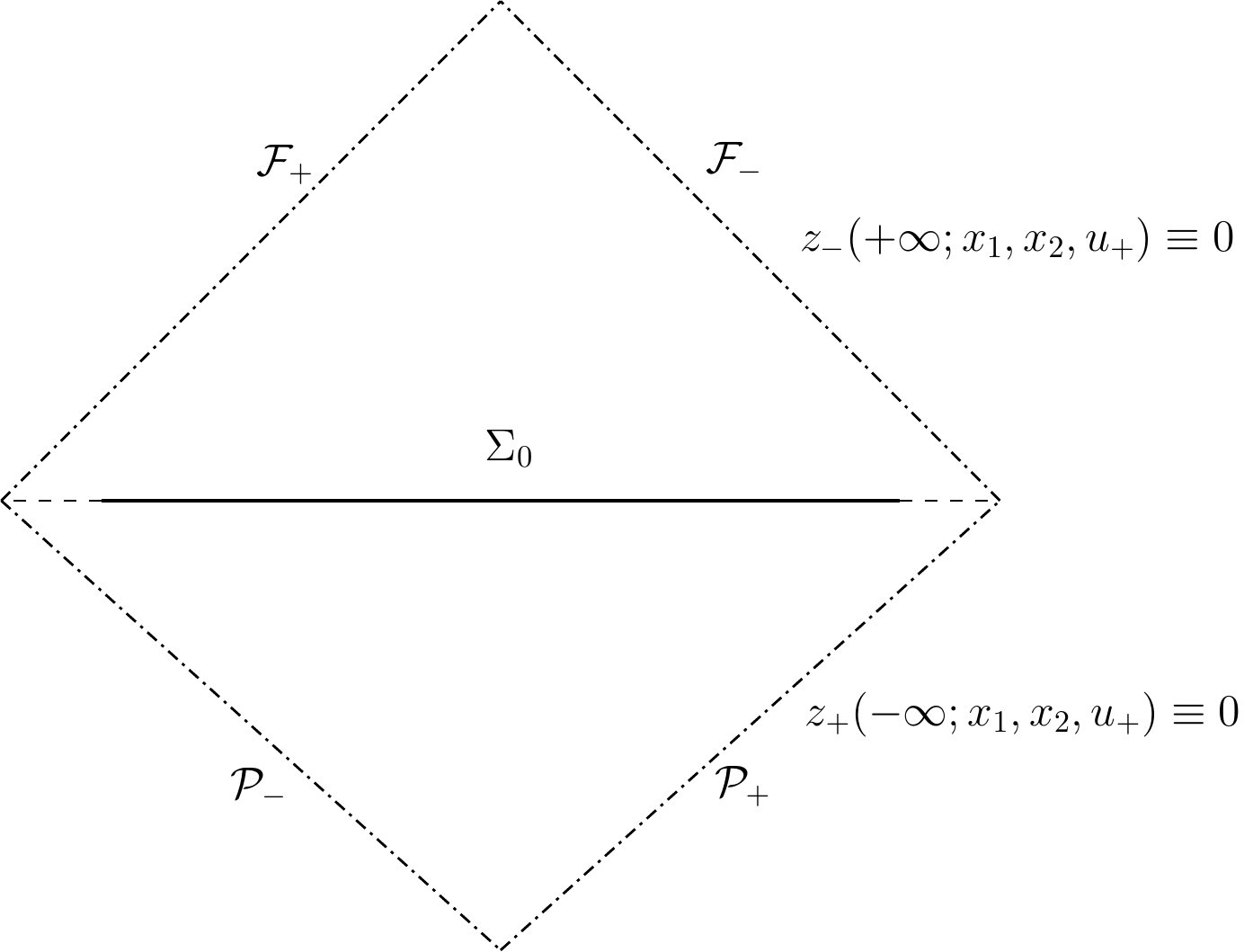}
\end{center}

We now assume that the scattering fields $z_-(+\infty;x_1,x_2,u_+)$ and $z_+(-\infty;x_1,x_2,u_-)$ vanish identically at infinities, i.e., 
\[z_-(+\infty;x_1,x_2,u_+)\equiv 0\  \text{on $\mathcal{F}_-$,}\]
\[z_+(-\infty;x_1,x_2,u_-)\equiv 0\  \text{on $\mathcal{P}_+$.}\]

The proof will make full use of  the Refine Energy Estimates. 	Let $\epsilon\ll\varepsilon_0$ be an arbitrarily given small positive constant. Thanks to Lemma \ref{lemma:15}, Lemma \ref{lemma:16} and Lemma \ref{lemma:16-}, there exists a $T_\epsilon>0$ such that we have the following smallness conditions:
\[\sum_{0\leqslant|\beta|\leqslant N_*}
\int_{\mathbb{R}^3}\big|\nabla^{\beta}z_{-}(T_\epsilon,x_1,x_2,u_++ T_\epsilon)\big|^2\langle u_+\rangle^{2\omega}dx_1dx_2du_+<\epsilon_-^2,
\]
\[\sum_{0\leqslant|\beta|\leqslant N_*}
\int_{\mathbb{R}^3}\big|\nabla^{\beta}z_{+}(-T_\epsilon,x_1,x_2,u_-+ T_\epsilon)\big|^2\langle u_-\rangle^{2\omega}dx_1dx_2du_-<\epsilon_+^2,
\]
where $\epsilon^2=\epsilon_+^2+\epsilon_-^2$. We now study the position parameter $a$. At initial slice $\Sigma_0$, the position parameter $a_0=0$ is given so that we have constructed the solution $z_\pm(t,x_1,x_2,x_3)$. At time slices $\Sigma_{T_\epsilon}$ and $\Sigma_{-T_\epsilon}$, within the Cartesian coordinates, the above smallness conditions can be expressed as
\[\sum_{0\leqslant|\beta|\leqslant N_*}
\int_{\mathbb{R}^3}\big|\nabla^{\beta}z_{-}(T_\epsilon,x_1,x_2,x_3)\big|^2\big(1+|x_3+ T_\epsilon|^2\big)^{\omega}dx_1dx_2dx_3<\epsilon_-^2,\]
\[\sum_{0\leqslant|\beta|\leqslant N_*}
\int_{\mathbb{R}^3}\big|\nabla^{\beta}z_{+}(-T_\epsilon,x_1,x_2,x_3)\big|^2\big(1+|x_3+ T_\epsilon|^2\big)^{\omega}dx_1dx_2dx_3<\epsilon_+^2.\]

\smallskip
 
On the one hand, if we take $a=T_\epsilon$ as the new position parameter and treat $\big(z_{+}(T_\epsilon,x),z_{-}(T_\epsilon,x)\big)$ as the initial data for the MHD system \eqref{eq:MHD}, we can solve the equations backwards in time. We also remark that on $\Sigma_0$, the new weights associated to $a$ indeed coincide with the original weights, i.e., $\big(1+|x_3|^2\big)^{\omega}$. The first smallness condition then can be rephrased as 
\begin{equation*}
\mathcal{E}_{-}^{N_*'}(T_\epsilon) 
<\epsilon_-^2,
\end{equation*}
where $N_*'=6$. 
We notice that 
\begin{equation*}
	\mathcal{E}_{+}^{N_*'}(T_\epsilon) 
	<\varepsilon_{+,0}^2.
\end{equation*}
Since $N_*'\geqslant 5$, we can apply the Refined Energy Estimates in this situation. Due to the fact that the estimates are independent of the position parameter, we infer that at time slice $\Sigma_0$,
\begin{equation*}
\mathcal{E}^{N_*'}_{-}(0)\leqslant C\epsilon_-^2+{C}\epsilon_-^2\varepsilon_{+,0}.
\end{equation*}

On the other hand, if we take $a=-T_\epsilon$ as the new position parameter and regard $\big(z_{+}(-T_\epsilon,x),z_{-}(-T_\epsilon,x)\big)$ as the initial data for the MHD system \eqref{eq:MHD}, we can solve the equations forwards in time. The second smallness condition then can be rephrased as  
\begin{equation*}
\mathcal{E}_{+}^{N_*'}(-T_\epsilon) 
<\epsilon_+^2,
\end{equation*}
where $N_*'=6$. 
We notice that 
\begin{equation*}\begin{split}
	\mathcal{E}_{-}^{N_*'}(-T_\epsilon) 
	<\varepsilon_{-,0}^2.
\end{split}\end{equation*}
Since $N_*'\geqslant 5$, we can also  apply the Refined Energy Estimates in this situation. According to the fact that the estimates are independent of the position parameter, we can derive that at time slice $\Sigma_0$,
\begin{equation*}
\mathcal{E}^{N_*'}_{+}(0)\leqslant C\epsilon_+^2+C\epsilon_+^2\varepsilon_{-,0}.
\end{equation*}

\smallskip

Adding these two results gives
\begin{equation*}
\mathcal{E}^{N_*'}(0)\leqslant C\epsilon^2+{C}\epsilon_-^2\varepsilon_{+,0}+{C}\epsilon_+^2\varepsilon_{-,0}.
\end{equation*}
Therefore, in view of the smallness of $\varepsilon_{+,0}$ and $\varepsilon_{-,0}$, the choice of $\epsilon$ that $\epsilon\ll\varepsilon_0$ leads us to
\begin{equation*}
\mathcal{E}^{N_*'}(0)\leqslant 2C\epsilon^2.
\end{equation*}

Since $\epsilon$ is arbitrary, we conclude that 
\begin{equation*}
\mathcal{E}^{N_*'}(0)=0.
\end{equation*}
This implies that the Alfv\'en waves $z_{+}(t,x)$ and $z_{-}(t,x)$ vanish identically.

\end{document}